\documentclass[11pt]{article}
\usepackage{amsfonts}
\usepackage{amsmath}
\usepackage{amssymb}
\usepackage{amsthm}
\usepackage{bbm}
\usepackage{color}
\usepackage[english]{babel}
\usepackage[latin1]{inputenc}

\renewcommand{\a}{\mathfrak{a}}
\newcommand{\p}{\mathfrak{p}}
\newcommand{\q}{\mathfrak{q}}

\newcommand{\s}{\mathfrak{s}}
\renewcommand{\u}{\mathfrak{u}}
\renewcommand{\v}{\mathfrak{v}}
\newcommand{\goQ}{\mathfrak{Q}}
\newcommand{\1}{\mathbbm{1}}
\newcommand{\C}{\mathbb{C}}
\newcommand{\N}{\mathbb{N}}

\newcommand{\R}{\mathbb{R}}
\renewcommand{\S}{\mathbb{S}}
\newcommand{\T}{\mathbb{T}}

\newcommand{\Z}{\mathbb{Z}}
\newcommand{\boC}{\mathcal{C}}
\newcommand{\boD}{\mathcal{D}}
\newcommand{\boE}{\mathcal{E}}

\newcommand{\boG}{\mathcal{G}}

\newcommand{\boK}{\mathcal{K}}
\newcommand{\boL}{\mathcal{L}}
\newcommand{\boO}{\mathcal{O}}
\newcommand{\boP}{\mathcal{P}}

\newcommand{\boS}{\mathcal{S}}

\newcommand{\boU}{\mathcal{U}}
\newcommand{\boV}{\mathcal{V}}

\newcommand{\eps}{\varepsilon}

\renewcommand{\deg}{{\rm deg}}
\newcommand{\dist}{{\rm dist}}
\renewcommand{\div}{{\rm div}}

\renewcommand{\Im}{{\rm Im}}
\newcommand{\on}{\ {\rm on} \ }
\renewcommand{\Re}{{\rm Re}}

\newcommand{\supp}{{\rm supp}}

\newtheorem{cor}{Corollary}
\newtheorem{claim}{Claim}
\newtheorem{lemma}{Lemma}
\newtheorem{prop}{Proposition}
\newtheorem{step}{Step}
\newtheorem{theorem}{Theorem}
\theoremstyle{definition}
\newtheorem*{acknowledgement}{Acknowledgements}
\newtheorem{case}{Case}
\newtheorem{remark}{Remark}

\begin{document}

\title{Travelling waves for the Gross-Pitaevskii equation II}
\author{
\renewcommand{\thefootnote}{\arabic{footnote}}
Fabrice B\'ethuel \footnotemark[1], Philippe Gravejat \footnotemark[2], Jean-Claude Saut \footnotemark[3]}
\footnotetext[1]{Laboratoire Jacques-Louis Lions, Universit\'e Pierre et Marie Curie, Bo\^ite Courrier 187, 75252 Paris Cedex 05, France. E-mail: bethuel@ann.jussieu.fr}
\footnotetext[2]{Centre de Recherche en Math\'ematiques de la D\'ecision, Universit\'e Paris Dauphine, Place du Mar\'echal De Lattre de Tassigny, 75775 Paris Cedex 16, France. E-mail: gravejat@ceremade.dauphine.fr}
\footnotetext[3]{Laboratoire de Math\'ematiques, Universit\'e Paris Sud, B\^atiment 425, 91405 Orsay Cedex, France. E-mail: Jean-Claude.Saut@math.u-psud.fr}
\date{}
\maketitle

\begin{abstract}
The purpose of this paper is to provide a rigorous mathematical proof of the existence of travelling wave solutions to the Gross-Pitaevskii equation in dimensions two and three. Our arguments, based on minimization under constraints, yield a full branch of solutions, and extend earlier results (see \cite{BethSau1, BetOrSm1, Chiron1}) where only a part of the branch was built. In dimension three, we also show that there are no travelling wave solutions of small energy.
\end{abstract}

\section{Introduction}

\subsection{Statement of the results}

In this paper, we investigate the existence of travelling waves to the Gross-Pitaevskii equation
\renewcommand{\theequation}{GP}
\begin{equation}
\label{GP}
i \partial_t \Psi = \Delta \Psi + \Psi (1 - |\Psi|^2) \on \R^N \times \R,
\end{equation}
in dimensions $N = 2$ and $N = 3$. Travelling waves are solutions to \eqref{GP} of the form
$$\Psi(x,t) = u(x_1 - ct, x_\perp), \ x_\perp = (x_2, \ldots, x_N).$$
Here, the parameter $c \in \R$ corresponds to the speed of the travelling waves (we may restrict to the case $c \geq 0$ using complex conjugation). The equation for the profile $u$ is given by
\renewcommand{\theequation}{TWc}
\begin{equation}
\label{TWc}
i c \partial_1 u + \Delta u + u (1 - |u|^2) = 0.
\end{equation}

The Gross-Pitaevskii equation appears as a relevant model in various areas of physics: non-linear optics, fluid mechanics, Bose-Einstein condensation... (see for instance \cite{Pitaevs1, Gross1, IordSmi1, JoneRob1, JonPuRo1, BethSau2, Berloff1}). At least on a formal level, the Gross-Pitaevskii equation is hamiltonian. The conserved Hamiltonian is a Ginzburg-Landau energy, namely
\renewcommand{\theequation}{\arabic{equation}}
\numberwithin{equation}{section}
\setcounter{equation}{0}
$$E(\Psi) = \frac{1}{2} \int_{\R^N} |\nabla \Psi|^2 + \frac{1}{4} \int_{\R^N} (1 - |\Psi|^2)^2 \equiv \int_{\R^N} e(\Psi).$$
Similarly, the momentum
$$\vec{P}(\Psi) = \frac{1}{2} \int_{\R^N} \langle i \nabla \Psi \ , \Psi - 1 \rangle,$$
is formally conserved. Here, the notation $\langle \ , \ \rangle$ stands for the canonical scalar product of the complex plane $\C$ identified to $\R^2$, i.e.
$$\langle z_1, z_2 \rangle = \Re(z_1) \Re(z_2) + \Im(z_1) \Im(z_2) = \Re \big( \overline{z_1} z_2 \big).$$
We will denote by $p$, the first component of $\vec{P}$, which is hence a scalar.

If $\Psi$ does not vanish, one may write
$$\Psi = \sqrt{\rho} \exp i \Phi,$$
which leads to the hydrodynamic form of the equation
\begin{equation}
\label{HDGP}
\left\{ \begin{array}{ll} \partial_t \rho + \div (\rho v) = 0, \\ \rho (\partial_t v + v. \nabla v) + \nabla \rho^2 = \rho \nabla \left( \frac{\Delta \rho}{\rho} - \frac{|\nabla \rho|^2}{2 \rho^2} \right),\end{array} \right.
\end{equation}
where
$$v = - 2 \nabla \Phi.$$

If one neglects the r.h.s of the second equation, which is often referred to as the quantum pressure, the system \eqref{HDGP} is similar to the Euler equation for a compressible fluid, with pressure law $p(\rho) = \rho^2$. In particular, the speed of sound waves near the constant solution $u = 1$ is given by
$$c_s = \sqrt{2}.$$

Travelling waves of finite energy play an important role in the dynamics of the Gross-Pitaevskii equation. They were considered by Iordanskii and Smirnov \cite{IordSmi1} in dimension three. They were thoroughly analyzed by Jones, Putterman and Roberts \cite{JoneRob1, JonPuRo1} both on a formal and numerical point of view, in dimensions two and three. For the two-dimensional case, they found a branch of solutions with speeds $c$ covering the full subsonic range $(0, \sqrt{2})$, and they conjectured the non-existence of travelling waves for supersonic speeds.

The program developed by Jones, Putterman and Roberts was approached more recently on a rigorous mathematical level. The non-existence of supersonic travelling waves was established in \cite{Graveja2}, and in \cite{Graveja4} for the sonic case that is $c = \sqrt{2}$ in dimension two (existence or non-existence of sonic travelling waves in dimension three remains still open). Therefore, we may assume throughout this paper that
$$0 < c < \sqrt{2} \ {\rm if } \ N = 2, \ {\rm and} \ 0 < c \leq \sqrt{2} \ {\rm if } \ N = 3.$$
The existence problem in dimensions two and three is the main focus of this paper: only part of the formal results provided in the physical literature has been established with rigorous mathematical proofs. It is a classical observation that one may obtain travelling waves by minimizing the energy $E$ keeping the momentum $p$ fixed. In this approach, equation \eqref{TWc} is the Euler-Lagrange equation to this constrained minimization problem. The speed $c$ appears as a Lagrange multiplier, and therefore is not fixed a priori. However an important advantage of this point of view is that it allows to address in a satisfactory way the question of stability. It was already used in \cite{BetOrSm1}, and will be followed here again. For a given $\p \geq 0$, we consider therefore the minimization problem
\begin{equation}
\label{eminent}
E_{\min}(\p) = \inf \{ E(v), v \in W(\R^N), p(v) = \p \},
\end{equation}
where $W(\R^N)$ represents a functional space adapted to the problem. The choice of $W(\R^N)$ has to fulfill two requirements. First, we wish the functionals $E$ and $p$ to be continuous on $W(\R^N)$. Second, all finite energy subsonic solutions to \eqref{TWc} have to belong to $W(\R^N)$. Properties of solutions of \eqref{TWc} have been studied recently extensively, in particular in \cite{Graveja1, Graveja3, Graveja5, Graveja6}. It is proved that any finite energy solution $u$ to \eqref{TWc} has a limit at infinity, which is a number of modulus one, which in view of the symmetry by rotation, we may take, without loss of generality, to be equal to $1$. The decay properties of solutions are summed up in Proposition \ref{superdecay}. In view of these properties, we introduce the space
\begin{equation}
\label{def-V}
V(\R^N) = \{ v: \R^N \mapsto \C, \ {\rm s.t.} \ (\nabla v, \Re(v)) \in L^2(\R^N)^2, \Im(v) \in L^4(\R^N) \ {\rm and} \ \nabla \Re(v) \in L^\frac{4}{3}(\R^N) \},
\end{equation}
and
\begin{equation}
\label{def-W}
W(\R^N) = \{ 1 \} + V(\R^N).
\end{equation}
It can be checked that the quantity $\langle i \partial_1 v, v-1 \rangle$ is integrable for a function $v \in W(\R^N)$, so that the so-called scalar momentum $p(v)$ is well-defined. This is a consequence of the identity
\begin{equation}
\label{identityp}
\langle i \partial_1 v, v-1 \rangle = \partial_1(\Re(v)) \Im(v) - \partial_1(\Im(v)) (\Re(v) - 1),
\end{equation}
and various H\"older's inequalities. Moreover, $W(\R^N)$ has the announced desired properties (see Corollary \ref{bienpose} and Lemma \ref{inject}).

\begin{remark}
If one may lift a map $v \in W(\R^N)$ as $v=\varrho \exp i\varphi$, then it can be shown under some suitable integrability assumptions (see Subsections \ref{nouveaudef} and \ref{propdecay}), that 
\begin{equation}
\label{celemoment}
p(v) = \frac{1}{2} \int_{\R^N} \langle i \partial_1 v \ , v - 1 \rangle = \frac{1}{2} \int_{\R^N} \eta \partial_1 \varphi,
\end{equation}
where, throughout the paper, we set
$$\eta = 1 - \varrho^2.$$
Notice that for maps which may be lifted, with $\varrho\geq \frac{1}{2}$, the last integral makes sense, even if we assume that $v$ only belongs to the energy space
$$\boE(\R^N) = \{ v: \R^N \mapsto \C, \ {\rm s.t.} \ E(v) < + \infty \}.$$
Hence the last integral in \eqref{celemoment} provides another definition for the momentum, which holds for maps which may be lifted, and which we use in most of this paper.
\end{remark}

In dimension two, our main existence theorem is the following.

\begin{theorem}
\label{dim2}
Let $\p > 0$. There exists a non-constant finite energy solution $u_\p \in W(\R^2)$ to equation \eqref{TWc}, with 
$c = c(u_\p)$, such that
$$p(u_\p) \equiv \frac{1}{2} \int_{\R^2} \langle i \partial_1 u_\p, u_\p - 1 \rangle = \p,$$
and such that $u_\p$ is solution to the minimization problem
$$E(u_\p) = E_{\min}(\p) = \inf \{ E(v), v \in W(\R^2), p(v) = \p \}.$$
\end{theorem}

\begin{remark}
In \cite{BethSau1}, existence of solutions is obtained for small prescribed speeds $c$. Instead of using minimization under constraint, the idea there is to introduce, for given $c$, the Lagrangian
$$F_c(v) = E(v) - c p(v),$$
whose critical points are solutions to \eqref{TWc}, and then to apply a mountain-pass argument. Although it is likely that the solutions obtained in \cite{BethSau1} correspond to the   solutions obtained in Theorem \ref{dim2} for large $\p$, we have no proof of this fact.
\end{remark}

\begin{remark}
Theorem \ref{dim2} shows in particular that there exist travelling wave solutions of arbitrary small energy. This suggests that scattering in the energy space is not likely to hold.
\end{remark}

In dimension three, the existence result is somewhat different.

\begin{theorem}
\label{dim3}
Assume $N = 3$. There exists some constant $\p_0 > 0$ such that:\\
i) For $0 < \p < \p_0$,
$$E_{\min}(\p) = \inf \{ E(v), v \in W(\R^3), p(v) = \p \} = \sqrt{2} \p,$$
and the infimum is not achieved in $W(\R^3)$.\\
ii) For $\p \geq \p_0$, there exists a non-constant finite energy solution $u_\p \in W(\R^3)$ to equation \eqref{TWc}, with 
$c = c(u_\p)$, such that $p(u_\p) = \p$, $E(u_{\p_0}) = E_{\min}(\p_0) = \sqrt{2} \p_0$ and, for $\p > \p_0$,
\begin{equation}
\label{sigmastrict}
E(u_\p) = E_{\min}(\p) = \inf \{ E(v), v \in W(\R^3), p(v) = \p \} < \sqrt{2} \p.
\end{equation}
iii) There exists some positive constant $\boE_0 \leq \sqrt{2} \p_0$, such that there are no non-trivial finite energy solutions $v$ to equation \eqref{TWc} satisfying
$$E(v) < \boE_0.$$
iv) We have
\begin{equation}
\label{trimaran}
\sup \{ c(u_\p), \p \geq \p_0 \} < \sqrt{2}.
\end{equation}
\end{theorem}

\begin{remark}
In \cite{BetOrSm1}, the existence of $u_\p$ was already established, but only for large values of the prescribed momentum $\p$. The solutions obtained there are the same as the one provided by Theorem \ref{dim3}. In \cite{Chiron1}, existence of non-trivial finite energy solutions to \eqref{TWc} was established for prescribed small speeds $c$, using a mountain-pass argument. As for the two-dimensional case, it is likely, but unknown, that the solutions obtained in \cite{Chiron1} correspond to the solutions obtained in Theorem \ref{dim2} for large $\p$.
\end{remark}

\begin{remark}
In contrast with the two-dimensional case, statement iii) in Theorem \ref{dim3} shows that there are no small energy travelling wave solutions in dimension three (see Lemma \ref{soniccase} in Subsection \ref{grandfou} for the proof of this statement). This opens the door to a small energy scattering theory. Such a theory has been developed in dimension $N \geq 4$ by Gustafson, Nakanishi and Tsai in \cite{GusNaTs1} (see also \cite{GusNaTs2, Nakanis1}). Very recently those authors have succeeded to present a scattering theory in dimension three (see \cite{GusNaTs3}).
\end{remark}

Besides the existence of minimizers, our analysis yields also properties of the curve $E_{\min}$ as well as of the speed $c(u_\p)$. First we prove that in both dimensions, $E_{\min}$ is a Lipschitz, non-decreasing and concave curve, which lies below the curve $\p \mapsto \sqrt{2} \p$, and tends to $+ \infty$ as $\p \to + \infty$ (see Theorem \ref{proemin} below). In particular, the function $E_{\min}$ is differentiable except possibly for a countable set of values: its derivative, at the points of differentiability is then given by the speed $c(u_\p)$, which satisfies for any $\p > 0$,
$$0 < c(u_\p) < \sqrt{2}.$$
It remains an open problem to determine whether the curve $E_{\min}$ is differentiable or not. This question is related to the problem of uniqueness up to symmetries for the minimizer $u_\p$, which is completely open as well. As a matter of fact, uniqueness for any $\p > 0$ of the minimizer would lead to the differentiability of the full curve, which in turn would provide a full interval of speeds. In dimension two, the spectrum of speeds would then be the interval $(0, \sqrt{2})$. Although we did not work out here a proof, we believe that our compactness result would show that, if at some point $E_{\min}$ were not differentiable, then there are at least two different minimizers with different speeds. In that case, the function $\p \mapsto c(u_\p)$ is not single valued. However, we can prove that it is a decreasing (possibly multivalued) function.

In dimension two, the function $E_{\min}$ has the following graph:

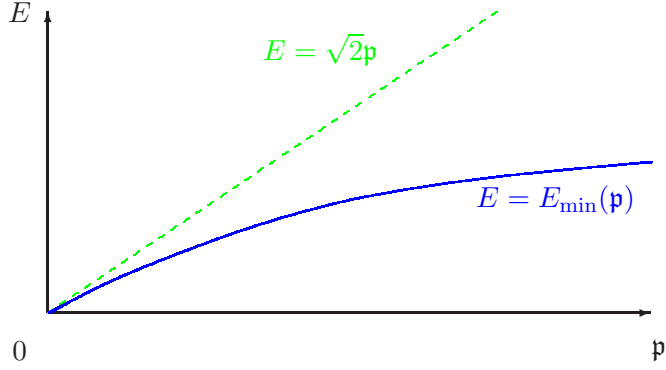
\begin{center}
\begin{picture}(90,50)(0,0)
\linethickness{0.2mm}
\put(10,10){\line(0,1){40}}
\put(10,50){\vector(0,1){0.12}}
\linethickness{0.2mm}
\put(10,10){\line(1,0){80}}
\put(90,10){\vector(1,0){0.12}}
{\color{green}
\linethickness{0.1mm}
\multiput(10,10)(1.64,1.1){37}{\multiput(0,0)(0.16,0.11){5}{\line(1,0){0.16}}}
}
{\color{blue}
\linethickness{0.2mm}
\qbezier(9,10)(9.29,10.17)(12.66,11.95)
\qbezier(12.66,11.95)(16.04,13.72)(19,15)
\qbezier(19,15)(26.49,18.14)(33.79,20.66)
\qbezier(33.79,20.66)(41.09,23.18)(49,25)
\qbezier(49,25)(60.4,27.33)(74.08,28.65)
\qbezier(74.08,28.65)(87.77,29.98)(89,30)
}
\put(5,5){\makebox(0,0)[cc]{$0$}}
\put(5,50){\makebox(0,0)[cc]{$E$}}
\put(90,5){\makebox(0,0)[cc]{$\p$}}
{\color{green}
\put(45,45){\makebox(0,0)[cc]{$E = \sqrt{2} \p$}}
}
{\color{blue}
\put(75,25){\makebox(0,0)[cc]{$E = E_{\min}(\p)$}}
}
\end{picture}
\end{center}

In dimension three, the graph of $E_{\min}$ has the following form:

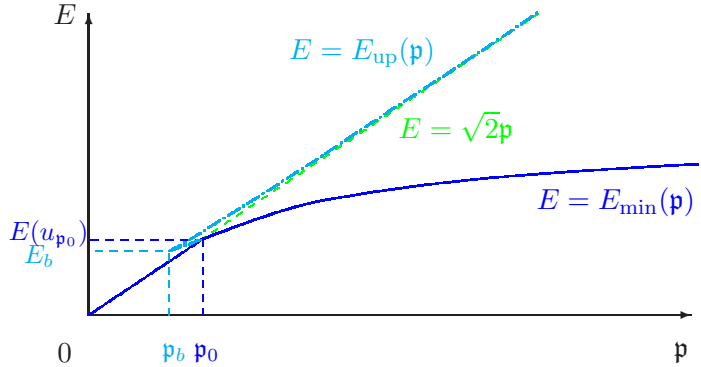
\begin{center}
\begin{picture}(90,50)(0,0)
\linethickness{0.2mm}
\put(10,10){\line(0,1){40}}
\put(10,50){\vector(0,1){0.12}}
\linethickness{0.2mm}
\put(10,10){\line(1,0){80}}
\put(90,10){\vector(1,0){0.12}}
{\color{green}
\linethickness{0.1mm}
\multiput(10,10)(1.64,1.1){37}{\multiput(0,0)(0.16,0.11){5}{\line(1,0){0.16}}}
}
{\color{blue}
\linethickness{0.1mm}
\multiput(24,10)(0,1.82){6}{\line(0,1){0.91}}
\linethickness{0.1mm}
\multiput(9,20)(2,0){8}{\line(1,0){1}}
}
{\color{blue}
\linethickness{0.2mm}
\multiput(7.4,10)(0.18,0.12){83}{\line(1,0){0.18}}
\linethickness{0.2mm}
\qbezier(22.4,20)(22.84,20.21)(27.9,22.07)
\qbezier(27.9,22.07)(32.96,23.94)(37.4,25)
\qbezier(37.4,25)(51.65,27.69)(68.75,28.86)
\qbezier(68.75,28.86)(85.86,30.03)(88.4,30)
}
{\color{cyan}
\linethickness{0.2mm}
\qbezier(21,20)(19.5,19.4)(19.5,19.4)
\qbezier(19,19.25)(18.9,19.2)(18.9,19.2)
\qbezier(18.4,19)(16.9,18.5)(16.9,18.5)
\qbezier(16.9,18.5)(18.4,19.3)(18.4,19.3)
\qbezier(18.9,19.6)(19,19.65)(19,19.65)
\qbezier(19.5,19.95)(21,20.85)(21,20.85)
\qbezier(21.5,21.2)(21.6,21.25)(21.6,21.25)
\qbezier(22.1,21.6)(23.6,22.55)(23.6,22.55)
\qbezier(24.1,22.85)(24.2,22.9)(24.2,22.9)
\qbezier(24.7,23.2)(26.2,24.15)(26.2,24.15)
\qbezier(26.7,24.45)(26.8,24.55)(26.8,24.55)
\qbezier(27.3,24.85)(28.8,25.75)(28.8,25.75)
\qbezier(29.3,26.1)(29.4,26.2)(29.4,26.2)
\qbezier(29.9,26.45)(31.4,27.4)(31.4,27.4)
\qbezier(31.9,27.75)(32,27.85)(32,27.85)
\qbezier(32.5,28.15)(34,29.1)(34,29.1)
\qbezier(34.5,29.4)(34.6,29.5)(34.6,29.5)
\qbezier(35.1,29.8)(36.6,30.75)(36.6,30.75)
\qbezier(37.1,31.1)(37.2,31.2)(37.2,31.2)
\qbezier(37.7,31.5)(39.2,32.5)(39.2,32.5)
\qbezier(39.7,32.8)(39.8,32.9)(39.8,32.9)
\qbezier(40.3,33.25)(41.8,34.25)(41.8,34.25)
\qbezier(42.3,34.55)(42.4,34.65)(42.4,34.65)
\qbezier(42.9,35)(44.4,35.95)(44.4,35.95)
\qbezier(44.9,36.25)(45,36.35)(45,36.35)
\qbezier(45.5,36.7)(47,37.7)(47,37.7)
\qbezier(47.5,38)(47.6,38.05)(47.6,38.05)
\qbezier(48.1,38.4)(49.6,39.35)(49.6,39.35)
\qbezier(50.1,39.75)(50.2,39.8)(50.2,39.8)
\qbezier(50.7,40.05)(52.2,41.05)(52.2,41.05)
\qbezier(52.7,41.4)(52.8,41.45)(52.8,41.45)
\qbezier(53.3,41.75)(54.8,42.75)(54.8,42.75)
\qbezier(55.3,43.1)(55.4,43.15)(55.4,43.15)
\qbezier(55.9,43.5)(57.4,44.5)(57.4,44.5)
\qbezier(57.9,44.85)(58,44.95)(58,44.95)
\qbezier(58.5,45.2)(60,46.2)(60,46.2)
\qbezier(60.5,46.6)(60.6,46.65)(60.6,46.65)
\qbezier(61.1,46.95)(62.6,47.95)(62.6,47.95)
\qbezier(63.1,48.35)(63.2,48.4)(63.2,48.4)
\qbezier(63.7,48.75)(65.2,49.75)(65.2,49.75)
\qbezier(65.7,50.1)(65.8,50.2)(65.8,50.2)
\linethickness{0.1mm}
\multiput(16.9,10)(0,1.82){5}{\line(0,1){0.91}}
\linethickness{0.1mm}
\multiput(7,18.5)(2,0){5}{\line(1,0){1}}
}
\put(3,5){\makebox(0,0)[cc]{$0$}}
\put(3,50){\makebox(0,0)[cc]{$E$}}
\put(85,5){\makebox(0,0)[cc]{$\p$}}
{\color{green}
\put(55,35){\makebox(0,0)[cc]{$E = \sqrt{2} \p$}}
}
{\color{blue}
\put(75,25){\makebox(0,0)[cc]{$E = E_{\min}(\p)$}}
}
{\color{cyan}
\put(40,45){\makebox(0,0)[cc]{$E = E_{\rm up}(\p)$}}
\put(15,5){\makebox(0,0)[cc]{{\small $\p_b$}}}
\put(-3,17.8){\makebox(0,0)[cc]{{\small $E_b$}}}
}
{\color{blue}
\put(18,5){\makebox(0,0)[cc]{{\small $\p_0$}}}
\put(-3,20.8){\makebox(0,0)[cc]{{\small $E(u_{\p_0})$}}}
}
\end{picture}
\end{center}

Notice that as a consequence of Theorem \ref{dim3}, $E(u_{\p_0}) = \sqrt{2} \p_0$, and that, in view of \eqref{trimaran}, the slope of the curve $E_{\min}$ at the point $(\p_0, \sqrt{2} \p_0)$ is strictly less than $\sqrt{2}$.

Our results are in full agreement with the corresponding figure given in \cite{JoneRob1}. In dimension three, the numerical value found in \cite{JoneRob1} for $\p_0$ is close to $80$. Jones and Roberts have also shown in \cite{JoneRob1}, mainly by numerical means, that in dimension three, the branch of solutions $u_\p$ can be extended past the curve $E = \sqrt{2} \p$. This curve is represented in the $E$-$\p$ diagram above as the curve $E_{\rm up}$. Starting from the point $(\p_0, \sqrt{2} \p_0)$, the curve $E_{\rm up}$ goes down to the left staying above the curve $E = E_{\min}(\p) = \sqrt{2} \p$, until it reaches the bifurcation point $(\p_b, E_b)$. After this bifurcation point, the curve $E_{\rm up}$ goes up to the right, and is asymptotic  from above to the curve $E = \sqrt{2} \p$, as $\p \to + \infty$. We believe that the presence of the bifurcation point $(\p_b, E_b)$ is due to the choice of our representation, and that the curve $\p \mapsto u_\p$ is actually differentiable.

At this stage, there is no mathematical proof of the existence of the upper branch of solutions. The fact that the slope of the curve at the point $(\p_0, \sqrt{2} \p_0)$ is strictly less than $\sqrt{2}$ leaves some hope to use an implicit function theorem to construct the curve $E_{\rm up}$, at least near $(\p_0, \sqrt{2} \p_0)$.

One important point which we have not addressed in this paper is the appearance of vortices (that is zeroes of solutions). It is known that in dimension two, solutions have two vortices for large $\p$ (see \cite{BethSau1}), whereas there are vortex rings in dimension three for large $\p$ (see \cite{BetOrSm1, Chiron1}). Jones, Putterman and Roberts conjectured in \cite{JoneRob1} the existence of some momentum $\p_1$ such that $u_\p$ has vortices for $\p \geq p_1$, and has no vortex otherwise. The numerical value found in \cite{JoneRob1} for $\p_1$ is close to $75$.

The next step in the analysis is to describe some properties of the solutions we obtained in Theorems \ref{dim2} and \ref{dim3}.

\begin{theorem}
\label{symetrie} 
Let $N = 2$ or $N = 3$, $\p > 0$ and assume that $E_{\min}(\p)$ is achieved by $u_\p$. Then $u_\p$ is real-analytic on $\R^N$, and is, up to a translation, axisymmetric. More precisely, there exists a function $\u_\p : \R \times \R_+$ such that
$$u_\p(x) = \u_\p(x_1, |x_\perp|), \ \forall x = (x_1, x_\perp) \in \R^N.$$
\end{theorem}

In another direction, the limit $\p \to + \infty$ has already been discussed in \cite{BethSau1, BetOrSm1, Chiron1}, and the analogy with solution of the incompressible Euler equations in fluid dynamics, stressed. In dimension two, we wish to initiate here the rigorous mathematical study of the other end of the $E_{\min}$ curve, namely the asymptotic properties in the limit $\p \to 0$, for which many interesting results have been derived in the physical literature.

In \cite{JoneRob1, JonPuRo1}, it is formally shown that, if $u_c$ is a solution to \eqref{TWc} in dimension two, then, after a suitable rescaling, the function $1 - |u_c|^2$ converges, as the speed $c$ converges to $\sqrt{2}$, to a solitary wave solution to the two-dimensional Kadomtsev-Petviashvili equation \eqref{KP}, which writes
\renewcommand{\theequation}{KP I}
\begin{equation}
\label{KP}
\partial_t u + u \partial_1 u + \partial_1^3 u - \partial_1^{- 1} (\partial_2^2 u) = 0.
\end{equation}
As \eqref{GP}, equation \eqref{KP} is hamiltonian, with Hamiltonian given by
\renewcommand{\theequation}{\arabic{equation}}
\numberwithin{equation}{section}
\setcounter{equation}{8}
\begin{equation}
\label{E-KP}
E_{KP}(u) = \frac{1}{2} \int_{\R^2} (\partial_1 u)^2 + \frac{1}{2} \int_{\R^2} (\partial_1^{-1}(\partial_2 u))^2 - \frac{1}{6} \int_{\R^2} u^3,
\end{equation}
and the $L^2$-norm of $u$ is conserved as well. Solitary-wave solutions $u(x, t) = w(x_1 - \sigma t, x_2)$ may be obtained in dimension two minimizing the Hamiltonian, keeping the $L^2$-norm fixed (see \cite{deBoSau3, deBoSau1}). The equation for the profile $w$ of a solitary wave of speed $\sigma = 1$ is given by
\begin{equation}
\label{SW}
\partial_1 w - w \partial_1 w - \partial_1^3 w + \partial_1^{- 1} (\partial_2^2 w) = 0.
\end{equation}
In contrast with the Gross-Pitaevskii equation, the range of speeds is the positive axis. Indeed, for any given $\sigma > 0$, a solitary wave $w_\sigma$ of speed $\sigma$ is deduced from a solution $w$ to \eqref{SW} by the scaling
$$w_\sigma(x_1, x_2) = \sigma w(\sqrt{\sigma} x_1, \sigma x_2).$$

The correspondence between the two equations is given as follows. Setting $\varepsilon \equiv \sqrt{2 - c^2}$ and $\eta_c \equiv 1 - |u_c|^2$, and performing the change of variables
\begin{equation}
\label{scaling}
w_c(x) = \frac{6}{\eps^2} \eta_c \Big( \frac{x_1}{\varepsilon}, \frac{\sqrt{2} x_2}{\eps^2} \Big),
\end{equation}
it is shown that $w$ approximatively solves \eqref{SW} as $c$ converges to $\sqrt{2}$. Set 
$$S(v) = E_{KP}(v) + \frac{1}{2} \int_{\R^2} v^2.$$
We will term ground-state a solution $w$ to \eqref{SW} which minimizes the action $S$ among all the solutions to \eqref{SW} (see \cite{deBoSau2} for more details). In dimension two, it is shown in \cite{deBoSau3} that $w$ is a ground state if and only if it minimizes the Hamiltonian keeping the $L^2$-norm fixed. The constant $\boS_{KP}$ denotes the action $S(w)$ of the ground-state solutions $w$ to equation \eqref{SW}. In the asymptotic limit $\p \to 0$, the value $E_{\min}(\p)$ relates to the constant $\boS_{KP}$ as the next result shows.

\begin{prop}
\label{T2bis}
Assume $N = 2$.\\
i) There exist some constants $\p_1 > 0$, $K_0$ and $K_1$ such that we have the asymptotic behaviours
\begin{equation}
\label{estimE2}
\frac{48 \sqrt{2}}{\boS_{KP}^2} \p^3 - K_0 \p^4 \leq \sqrt{2} \p - E_{\min}(\p) \leq K_1 \p^3, \forall 0 \leq \p \leq \p_1.
\end{equation}
ii) Let $u_\p$ be as in Theorem \ref{dim2}. Then, there exist some constants $\p_2 > 0$, $K_2 > 0$ and $K_3$ such that
\begin{equation}
\label{c-estim}
K_2 \p^2 \leq \sqrt{2} - c(u_\p) \leq K_3 \p^2, \forall 0 \leq \p < \p_2.
\end{equation}
Moreover, the map $u_\p$ verifies $|u_\p| \geq \frac{1}{2}$, so that we may write $u_\p = \varrho_\p \exp i \varphi_\p$, and we have the estimates
\begin{equation}
\label{p-cube-estim}
\int_{\R^2} \Big( |\nabla \varrho_\p|^2 + |\partial_2 u_\p|^2 \Big)+
\bigg| \int_{\R^2} (1 - \varrho_\p^2) |\nabla \varphi_\p|^2 \bigg| \leq K \p^3,
\end{equation}
and
\begin{equation}
\label{Min-estim}
|u_\p(0)| = \min_{x\in \R^2} |u_\p(x)| \leq 1 - K \p^2,
\end{equation}
where $K$ is some universal constant.
\end{prop}

In a separate paper, we will provide a rigorous proof of the asymptotic expansion given in \cite{IordSmi1, JonPuRo1}, under specific assumptions on the solutions $u_\p$: these assumptions are in particular verified by the solutions constructed in Theorem \ref{dim2}, thanks in particular to estimates \eqref{p-cube-estim}.

\begin{remark}
If $u_c$ is a solution to \eqref{TWc} in dimension three, then it is also formally shown in \cite{IordSmi1,JoneRob1,JonPuRo1}, that the function $w_c$ defined by
$$w_c(x) = \frac{6}{\eps^2} \bigg( 1 - \Big|v_c \Big( \frac{x_1}{\varepsilon}, \frac{\sqrt{2} x_2}{\eps^2},  \frac{\sqrt{2} x_3}{\eps^2} \Big) \Big|^2 \bigg),$$
converges, as the speed $c$ converges to $\sqrt{2}$, to a solitary wave solution $w$ to the three-dimensional Kadomtsev-Petviashvili equation \eqref{KP}, which writes
$$\partial_t u + u \partial_1 u + \partial_1^3 u - \partial_1^{- 1} (\partial_2^2 u + \partial_3^2 u) = 0.$$
In particular, the equation for the solitary wave $w$ is now written as
$$\partial_1 w - w \partial_1 w - \partial_1^3 w + \partial_1^{- 1} (\partial_2^2 w + \partial_3^2 w) = 0.$$
\end{remark}

\begin{remark}
For $N = 2$ and $N = 3$, the Cauchy problem for \eqref{GP} is known to be well-posed in the energy space $\boE(\R^N)$ (see \cite{Gerard1, Gerard2}), as well as in the space $\{ v \} + H^1(\R^N)$, where $v$ is a finite energy solution to \eqref{TWc} (see \cite{Gallo1}, and also \cite{BethSau1, Goubet1, GusNaTs1}). An important advantage of the space $\{ v \} + H^1(\R^N)$, besides the fact that it is affine, is that the momentum is well-defined (in contrast as mentioned above with the energy space). Moreover, it can be shown that the momentum, defined by \eqref{identityp} in the three-dimensional case, and after an integration by parts, by
$$p(v) = - \int_{\R^N} \partial_1 (\Im(v)) (\Re(v) - 1),$$
in the two-dimensional case, is a conserved quantity. In particular, the fact that $u_\p$ solves the minimization problem \eqref{eminent} strongly suggests that $u_\p$ is orbitally stable. A rigorous proof of the orbital stability of $u_\p$ would require, in addition to solving the Cauchy problem, to obtain compactness properties for minimizing sequences for \eqref{eminent}. We will not tackle this problem here.
\end{remark}

\subsection{Some elements in the proofs}

The starting point of our proofs is a careful analysis of the curve $p \mapsto E_{\min}(p)$.

\begin{theorem}
\label {proemin}
Let $N = 2$ or $N = 3$. For any $\p, \q \geq 0$, we have the inequality
\begin{equation}
\label{start}
|E_{\min}(\p) - E_{\min}(\q)| \leq \sqrt{2} |\p - \q|, 
\end{equation}
i.e. the real-valued function $\p \mapsto E_{\min}(\p)$ is Lipschitz, with Lipschitz's constant $\sqrt{2}$. Moreover, it is concave and non-decreasing on $\R_+$. Set
\begin{equation}
\label{discrep}
\Xi(\p) = \sqrt{2} \p - E_{\min}(\p).
\end{equation}
The function $\p \mapsto \Xi(\p)$ is nonnegative, convex, continuous, non-decreasing on $\R_+$, and tends to $+ \infty$ as $\p \to + \infty$. In particular, there exists some number $\p_0 \geq 0$ such that $\Xi(\p) = 0$, if $\p \leq \p_0$, and $\Xi(\p) > 0$, otherwise.
\end{theorem}

An important consequence of the concavity of the function $E_{\min}(\p)$ is the following

\begin{cor}
\label{subadditif}
The function $E_{\min}$ is subadditive that is, for any non-negative numbers $\p_1, \ldots, \p_\ell$, we have
\begin{equation}
\label{subadditivity}
\sum_{i = 1}^\ell E_{\min}(\p_i) \geq E_{\min}\Big( \sum_{i=1}^\ell \p_i \Big).
\end{equation}
Moreover, if $\ell \geq 2$ and \eqref{subadditivity} is an equality, then the function $E_{\min}$ is linear on $(0, \p)$, where 
$\p \equiv \underset{i = 1}{\overset{\ell}{\sum}} \p_i$.
\end{cor}

\begin{proof}
Since $E_{\min}(0) = 0$, and since $E_{\min}$ is concave, its graph lies above the segment joining $(0, 0)$ and $(\p, E_{\min}(\p))$. In particular, for any $0 \leq \q \leq \p$, we have
$$E_{\min}(\q) \geq \q \frac{E_{\min}(\p)}{\p}.$$
Specifying this relation for $\p_i$, and adding these inequalities, we obtain \eqref{subadditivity}. If \eqref{subadditivity} is an equality, then necessarily $E_{\min}(\p_i) = \p_i \frac{E_{\min}(\p)}{\p}$, and the graph has to be linear.
\end{proof}

Since the function $E_{\min}$ is Lipschitz, non-decreasing and concave, its left and right derivatives exist for any $\p \geq 0$, are equal except possibly on a countable subset $\goQ$ of $\R_+$, are non-negative and non-increasing, and satisfy the inequality
$$0 \leq \frac{d^+}{d\p} \big( E_{\min}(\p) \big) \leq \frac{d^-}{d\p} \big( E_{\min}(\p) \big) \leq \sqrt{2},$$
where we have set
$$\frac{d^\pm}{d\p} \big( E_{\min}(\p) \big) \equiv \lim_{\Delta \p \to 0^+} \frac{ E_{\min}(\p \pm \Delta \p) - E_{\min}(\p)}{\pm \Delta \p}.$$
Moreover, the derivatives are related to the speed $c(u_\p)$ as follows.

\begin{lemma}
\label{speed}
Let $\p > 0$ and assume that $E_{\min}(\p)$ is achieved by a solution $u_\p$ of \eqref{TWc} of speed $c(u_\p)$. Then we have
\begin{equation}
\label{lrspeed}
\frac{d^+}{d\p} \big( E_{\min}(\p) \big) \leq c(u_\p) \leq \frac{d^-}{d\p} \big( E_{\min}(\p) \big).
\end{equation}
\end{lemma} 

Strict concavity also plays an important role in our argument. In that direction, we have

\begin{lemma}
\label{gueri}
Let $0 \leq \p_1 < \p_2$ and assume the function $E_{\min}$ is affine on $(\p_1, \p_2)$. Then, for any $\p_1 < \p < \p_2$, the infimum $E_{\min}(\p)$ is not achieved in $W(\R^N)$.
\end{lemma}

So far, we have not addressed the existence problem for $u_\p$. Notice that as a direct consequence of \eqref{start}, one has
\begin{equation}
\label{start2}
E_{\min}(\p) \leq \sqrt{2}\p.
\end{equation}
Inequality \eqref{start2} corresponds in some sense to a linearization of the equation, or alternatively to an asymptotic situation where only quadratic terms in the functional are kept. To get some feeling for its proof, we consider a map $v \in \{ 1 \} + C_c^\infty(\R^N)$ such that 
$$\delta = \underset{x \in \R^N}{\inf} |v(x)| \geq \frac{1}{2},$$
so that we may write $v = \varrho \exp i \varphi$. To obtain \eqref{start2} we need to construct $v$ so that $E(v) \simeq \sqrt{2} |p(v)|$. In view of formula \eqref{celemoment} for the momentum, we have
\begin{equation}
\label{maininequality0}
|p(v)| = \Big| \frac{1}{2} \int_{\R^N} (1 - \varrho^2) \partial_1 \varphi \Big| \leq \frac{1}{2 \delta} \int_{\R^N} \Big| 1 - \varrho^2 \Big| \Big| \varrho \partial_1 \varphi \Big|.
\end{equation}
Using the inequality $a b \leq \frac{1}{2} (a^2 + b^2)$, we observe therefore that
$$|p(v)| \leq \frac{1}{\sqrt{2} \delta} \bigg( \frac{1}{2} \int_{\R^N} \varrho^2 |\nabla \varphi|^2 + \frac{1}{4} \int_{\R^N} \Big( 1 - \rho^2 \Big)^2 \bigg) \leq \frac{1}{\sqrt{2} \delta} E(v),$$
i.e. $\sqrt{2} \delta |p(v)| \leq E(v)$. To obtain a map such that $E(v) \simeq \sqrt{2} |p(v)|$, we need therefore to have $\delta$ close to $1$, and the inequality $a b \leq \frac{1}{2} (a^2 + b^2)$ close to an equality, that is $a \simeq b$ or in our case
$$\sqrt{2} \partial_1 \varphi \simeq 1 - \varrho^2.$$
This elementary observation is the starting point in the proof of the existence of solutions minimizing $E_{\min}$, in the case $\Xi(\p) > 0$, that is for $\p > \p_0$. As a matter of fact, the discrepancy term
$$\Sigma(v) = \sqrt{2} p(v) - E(v)$$
is central in our analysis. Notice in particular, that
$$\Xi(\p) = \sup \{ \Sigma(v), v \in W(\R^N), \ p(v) = \p \},$$
and that, in view of Theorem \ref{proemin}, a way to formulate the fact that $\p_0 < \p$ is to say that there exists a map $v \in W(\R^N)$ such that
$$\Sigma(v) > 0, \ {\rm and} \ p(v) = \p.$$
In this situation, we have
 
\begin{lemma}
\label{discrepancy}
Let $v \in W(\R^N)$ and assume $p(v) > 0$. Then we have
\begin{equation}
\label{normsup}
\inf_{x \in \R^N} |v(x)| \leq \max \Big\{ \frac{1}{2}, 1 - \frac{\Sigma(v)}{\sqrt{2} p(v)} \Big\}.
\end{equation}
\end{lemma}

\begin{proof}
Set as above
$$\delta \equiv \inf_{x \in \R^N} |v(x)|.$$
If $\delta \leq \frac{1}{2}$, there is nothing to prove. Otherwise, one may show that $v$ has a lifting, i.e. that we may write
$v = \varrho \exp i \varphi.$
It follows therefore from \eqref{maininequality0} that
$$\sqrt{2} \delta p(v) = \sqrt{2} \delta |p(v)| \leq E(v) = \sqrt{2} p(v) - \Sigma(v),$$
and hence
$$1 - \delta \geq \frac{\Sigma(v)}{\sqrt{2} p(v)},$$
which yields the conclusion.
\end{proof}

Lemma \ref{discrepancy} is the starting point in the analysis of minimizing sequences for \eqref{eminent}. In particular, it shows that their modulus cannot converge uniformly to $1$ in the case $\p_0 < \p$. However, to go beyond that simple observation, one needs to get a better control on minimizing sequences.

Among the analytical difficulties which have already been stressed in \cite{BethSau1,BetOrSm1,Chiron1}, the first one is presumably the lack of compactness of minimizing sequences or Palais-Smale sequences. Working directly with arbitrary Palais-Smale sequences, i.e. approximate solutions to \eqref{TWc} with an additional small $H^{-1}$-term  leads to substantial technical issues, which seem hard to remove. The lack of regularity of the $H^{-1}$-term raises some major problems, for instance concerning regularity of the functions under hand, as well as Pohozaev's type identities which turn out to be crucial in our arguments, in particular in order to bound the Lagrange multiplier
\footnote{This direct approach would have the important advantage to pave the way to the study of the orbital stability of the travelling waves.}.

To overcome the difficulties related to a direct approach, we specify the way minimizing sequences are constructed. There are presumably many ways to proceed. Here, we consider the corresponding minimization problem on expanding tori (as in \cite{BetOrSm1}). This choice has several advantages. First, the torus is compact, so that the existence of minimizers presents no major difficulty. Second, it has no boundary, so that the elliptic theory is essentially the local one and concentration near the boundary is avoided. The torus captures also some of the translation invariance for the problem on $\R^N$. Finally, Pohozaev's identities yield comfortable bounds for the Lagrange multipliers, which provide a uniform control on the ellipticity of \eqref{TWc}. Our strategy to obtain compactness for the sequence of approximate minimizers is then to develop the elliptic theory for the equation on tori, derive several estimates which do not rely on the size of the torus, and then to pass to the limit when the size of the torus tends to infinity.

More precisely, we introduce the flat torus, for $N = 2$ and $N = 3$, defined by
$$\T_n^N \simeq \Omega_n^N \equiv [- \pi n, \pi n]^N,$$
for any $n \in \N^*$ (with opposite faces identified), and the space 
$$X_n^N = H^1(\T_n^N, \C) \simeq H^1_{\rm per}(\Omega_n^N, \C)$$
of $2 n$-periodic $H^1$-functions. We define the energy $E_n$ and the momentum $p_n$ on $X_n^N$ by
$$E_n(v) = \frac{1}{2} \int_{\T_n^N} |\nabla v|^2 + \frac{1}{4} \int_{\T_n^N} (1 - |v|^2)^2 = \int_{\T_n^N} e(v),$$
and
$$p_n(v) = \frac{1}{2} \int_{\T_n^N} \langle i \partial_1 v , v \rangle,$$
which clearly defines a quadratic functional on $X_n^N$, as well as the discrepancy term 
$\Sigma_n(v) = \sqrt{2} p_n(v) - E_n(v)$.
We introduce the set $\Gamma_n^N(\p) $ defined in dimension three by
$\Gamma_n^3(\p) \equiv \{ u \in X_n^3, p_n(u) = \p \},$
whereas in dimension two, its definition is slightly more involved, and is given by
$$\Gamma_n^2(\p) \equiv \{ u \in X_n^2, p_n(u) = \p \} \cap \boS_n^0.$$
The set $\boS_n^0$ corresponds to a topological sector of the energy $E_n$, following the approach of Almeida \cite{Almeida1}. We will define it precisely in definition \eqref{supersecteur} of Subsection \ref{luis}: at this stage, let us just mention that we introduce the set $\boS_n^0$ to have appropriate lifting properties far from the possibly vorticity set. We consider the minimization problem
\renewcommand{\theequation}{{$\boP_n^N(\p)$}}
\begin{equation}
\label{PnNp}
E_{\min}^n(\p) \equiv \underset{u \in \Gamma_n^N(\p) }{\inf} \Big( E_n(u) \Big).
\end{equation}
The constraint is non void, so that it is possible to prove the existence of a minimizer for \eqref{PnNp}.

\begin{prop}
\label{existnn}
Assume $N = 3$, or $N = 2$ and $n \geq \tilde{n}(\p)$, where $\tilde{n}(\p)$ is some integer only depending on $E_{\min}(\p)$. Then, there exists a minimizer $u_\p^n \in \Gamma_n^N(\p) $ for $E_{\min}^n(\p)$, and some constant $c_\p^n \in \R$ such that $u_\p^n$ satisfies (TW$c_\p^n$), i.e.
$$i c_\p^n \partial_1 u_\p^n + \Delta u_\p^n + u_\p^n (1 - |u_\p^n|^2) = 0 \ {\rm on} \ \T_n^N.$$
In particular, $u_\p^n$ is smooth. Moreover if 
\renewcommand{\theequation}{\arabic{equation}}
\numberwithin{equation}{section}
\setcounter{equation}{22}
$$\Xi(\p) > 0,$$
then there exist a constant $K(\p)$, and an integer $n(\p)$, only depending on $\p$ such that
\begin{equation}
\label{supc}
|c_\p^n| \leq K(\p).
\end{equation}
for any $n \geq n(\p)$. In particular, for any $k \in \N$, there exists some constant $K_k(\p)$ only depending on $\p$ and $k$ such that we have
\begin{equation}
\label{fastoche}
\| u_\p^n \|_{C^k(\T_n^N)} \leq K_k(\p).
\end{equation}
\end{prop}

The last estimate in Proposition \ref{existnn} is a simple consequence of bound \eqref{supc} on $c_\p^n$, combined with standard elliptic estimates. 

In view of the invariance by translation of the problem \eqref{PnNp} on the torus $\T_n^N$, we may assume without loss of generality that the infimum of $|u_\p^n|$ is achieved at the point 0, that is
\begin{equation}
\label{translation}
|u_\p^n(0)| = \min_{x\in \T_n^N} |u_\p^n(x)|.
\end{equation}
On the other hand, the argument of the proof of Lemma \ref{discrepancy} carries over for continuous maps $v \in X_n^2 \cap \boS_n^0$, resp. $v \in X_n^3$, for $n$ sufficiently large, so that
\begin{equation}
\label{rebelote}
\min_{x \in \T_n^N} |v(x)| \leq \sup \Big\{ \frac{1}{2}, 1 - \frac{\Sigma_n(v)}{\sqrt{2} p_n(v)} \Big\},
\end{equation}
if $p_n(v) > 0$, and $n$ is sufficiently large. Indeed, it follows from Lemma \ref{sectorisation}, resp. Lemma \ref{massimo3}, that continuous maps $v \in X_n^2 \cap \boS_n^0$, resp. $v \in X_n^3$, such that $|v| \geq \frac{1}{2}$ on $\T_n^N$, have a lifting on $\T_n^N$, so that the argument of the proof of Lemma \ref{discrepancy} applies without any change. Combining \eqref{translation} and \eqref{rebelote}, we are led to
$$\underset{n \to + \infty}{\limsup} \big( |u_\p^n(0)| \big) \leq \sup \Big\{ \frac{1}{2}, 1 - \frac{\Xi(\p)}{\sqrt{2} \p} \Big\}.$$
If $\Xi(\p) > 0$, then we may use Ascoli's and Rellich's compactness theorems to assert

\begin{prop}
\label{troisfoisrien} 
Let $N = 2$ or $N = 3$, $\p > 0$, and assume that
\begin{equation}
\label{themaincondition}
\Xi(\p) > 0.
\end{equation}
Then, there exists a non-trivial finite energy solution $u_\p$ to \eqref{TWc} such that, passing possible to a subsequence, we have
\begin{equation}
\label{racinededeux}
u_\p^n \underset{n \to + \infty}{\to} u_\p \ {\rm in} \ C^k(K),
\end{equation}
for any $k \in \N$, and any compact set $K$ in $\R^N$.Moreover, we have
$$E(u_\p) \leq E_{\min}(\p),$$
and 
$$|u_\p(0)| \leq \sup \Big\{ \frac{1}{2}, 1 - \frac{\Xi(\p)}{\sqrt{2} \p} \Big\} < 1.$$
\end{prop}

Notice that, provided one is able to verify condition \eqref{themaincondition}, Proposition \ref{troisfoisrien} already provides the existence of non-trivial travelling wave solutions. To go further and establish the statements of Theorems \ref{dim2} and \ref{dim3}, one needs to pass to the limit in the integral quantities $E$ and $p$, the main difficulty being that the domain is not bounded, and therefore we have only weak convergences in $L^2$. The possible failure of convergence is described in the next proposition, which is a classical concentration-compactness result.

\begin{prop}
\label{concentrationcompacite}
Let $N = 2$ or $N = 3$, $\p > 0$, and assume \eqref{themaincondition} is satisfied. Let $u_\p^n$ and $u_\p$ be as in Proposition \ref{troisfoisrien}. Then, there exists an integer $\ell_0$ depending only on $\p$, and $\ell$ points $x_1^n = 0$, $\ldots$, $x_\ell^n$ depending on $n$, $\ell$ finite energy solutions $u_1 = u_\p$, $\ldots$, $u_\ell$ to \eqref{TWc}, and a subsequence of $(u_\p^n)_{n \in \N^*}$ still denoted $(u_\p^n)_{n \in \N^*}$ such that $\ell \leq \ell_0$, 
\begin{equation}
\label{intermarche}
|x_i^n - x_j^n| \to + \infty, \ {\rm as} \ n \to + \infty,
\end{equation} 
for any $i \neq j$, and
\begin{equation}
\label{limitlimit}
u_\p^n( \cdot + x_i^n) \to u_i(\cdot) \ {\rm in} \ C^k(K), \ {\rm as} \ n \to + \infty,
\end{equation}
for any compact set $K \subset \R^N$ and any $k \in \N$. Moreover, we have the identities
$$E_{\min}(\p)= \sum_{i = 1}^\ell E(u_i), \ {\rm and} \ \p = \sum_{i = 1}^\ell p(u_i).$$
The maps $u_i$ are minimizers for $E_{\min}(\p_i)$, where $\p_i = p(u_i)$, and
$$0 < c(u_\p) < \sqrt{2}.$$
\end{prop}

Combining Corollary \ref{subadditif}, Lemma \ref{gueri} and Proposition \ref{concentrationcompacite} we obtain

\begin{theorem}
\label{principaltheo}
Assume $N = 2$ or $N = 3$. If $\p > \p_0$, where $\p_0 \geq 0$ is defined in Theorem \ref {proemin}, then $E_{\min}(\p)$ is achieved by the map $u_\p \in W(\R^N)$ constructed in Proposition \ref{troisfoisrien}.
\end{theorem}

Theorems \ref{dim2} and \ref{dim3} are then deduced from Theorem \ref{principaltheo} and various aspects of the analysis presented above. We will show in particular, thanks to Proposition \ref{T2bis}, that $\p_0 = 0$ in dimension two, whereas $\p_0 > 0$ in dimension three. This fact accounts for a large part for the differences in the statements of Theorems \ref{dim2} and \ref{dim3}.

\subsection{Outline of the paper}

The paper is organized as follows. In the next section, we present some known and also some new properties of finite energy travelling wave solutions. In Section \ref{PropEmin}, we provide properties of $E_{\min}$, in particular, we prove Theorems \ref{symetrie} and \ref{proemin}. In Section \ref{torus}, we study solutions on tori, whereas in Section \ref{Torusinfini}, we study their asymptotic properties on expanding tori. In particular, we prove the concentration-compactness result. In Section \ref{epinfini}, we consider the minimization problem on tori, and we provide the asymptotic limits of the energy and momentum on expanding tori. In Section \ref{proofmain}, we finally complete the proofs of Theorems \ref{dim2}, \ref{dim3} and \ref{principaltheo}.

\numberwithin{cor}{section}
\numberwithin{lemma}{section}
\numberwithin{prop}{section}
\numberwithin{remark}{section}
\numberwithin{theorem}{section}
\section{Properties of finite energy solutions on $\R^N$}

In this section, we recall some known facts about finite energy solutions to \eqref{TWc} on $\R^N$, and supplement them with some results which enter in our analysis. 

\subsection{Pointwise estimates}

The following results were proved in \cite{Farina1} (see also \cite{Tarquin1}).

\begin{lemma}[\cite{Farina1, Tarquin1}]
\label{tarquini10} 
Let $v$ be a finite energy solution to \eqref{TWc} on $\R^N$. Then, $v$ is a smooth, bounded function on $\R^N$. Moreover, there exist some constants $K(N)$ and $K(c, k, N)$ such that
\begin{equation}
\label {elinfini10}
\Big\| 1 - |v| \Big\|_{L^\infty(\R^N)} \leq \max \Big\{ 1 , \frac{c}{2} \Big\},
\end{equation}
\begin{equation}
\label{elinfinigrad0}
\| \nabla v \|_{L^\infty(\R^N)} \leq K(N) \Big( 1 + \frac{c^2}{4} \Big)^\frac{3}{2}, 
\end{equation}
and more generally,
\begin{equation}
\label{ckestimates0}
\| v \|_{C^k(\R^N)} \leq K(c, k, N), \forall k \in \N.
\end{equation}
\end{lemma}

\begin{proof}
In view of \cite{Farina1, Tarquin1} (see also \cite{BethSau1, Graveja3}), a finite energy solution $v$ to equation \eqref{TWc} is a smooth, bounded function on $\R^N$, such that
\begin{equation}
\label{nada}
|v(x)| \to 1, \ {\rm as} \ |x| \to + \infty.
\end{equation}
In particular, we have
$$\| v \|_{L^\infty(\R^N)} \geq 1.$$

In order to prove inequalities \eqref{elinfini10}, \eqref{elinfinigrad0} and \eqref{ckestimates0}, we compute the laplacian of $|v|^2$. By the inequality $a b \leq \frac{1}{2} (a^2 + b^2)$, we have
\begin{align*}
\Delta |v|^2 = 2 \langle v \ , \Delta v \rangle + 2 |\nabla v|^2 = & 2 |\nabla v|^2 - 2 c \langle i \partial_1 v \ , v \rangle - 2 |v|^2 (1 -|v|^2)\\
\geq & 2 |\nabla v|^2 - 2 |\partial_1 v|^2 - \frac{c^2}{2} |v|^2 - 2 |v|^2 (1 -|v|^2),
\end{align*}
so that
\begin{equation}
\label{nadie}
\Delta |v|^2 + 2 |v|^2 \Big( 1 + \frac{c^2}{4} - |v|^2 \Big) \geq 0.
\end{equation}
When $\| v \|_{L^\infty(\R^N)} > 1$, it follows from \eqref{nada} and \eqref{nadie} that we can apply the weak maximum principle to $|v|^2$ to obtain
\begin{equation}
\label{v-linfini}
|v|^2 \leq 1 + \frac{c^2}{4}.
\end{equation}
When $\| v \|_{L^\infty(\R^N)} = 1$, inequality \eqref{v-linfini} is straightforward, so that it holds in any case. In particular, the function $1 - |v|$ satisfies
$$\Big\| 1 - |v| \Big\|_{L^\infty(\R^N)} \leq \max \Big\{ 1, \sqrt{1 + \frac{c^2}{4}} - 1 \Big\} \leq \max \Big\{ 1, \frac{c}{2} \Big\},$$
so that inequality \eqref{elinfini10} is established.

We turn to \eqref{elinfinigrad0}. Consider the function $w$ defined by
\begin{equation}
\label{v-exp}
w(x) = v(x) \exp \Big( i \frac{c}{2} x_1 \Big), \forall x \in \R^N.
\end{equation}
By equation \eqref{TWc}, $w$ satisfies
$$\Delta w + w \Big( 1 + \frac{c^2}{4} - |w|^2 \Big)=0.$$
Letting $x_0 \in \R^N$, and denoting $B(x_0, 1) = \{ y \in \R^N, \ {\rm s.t.} \ |y - x_0| < 1 \}$, the ball with center $x_0$ and radius $1$, it holds from inequality \eqref{v-linfini} that
\begin{equation}
\label{Deltaw-bound}
\| \Delta w \|_{L^\infty(B(x_0, 1))} \leq \frac{2}{3 \sqrt{3}} \Big( 1 + \frac{c^2}{4} \Big)^\frac{3}{2}.
\end{equation}
By standard elliptic theory, there exists some constant $K(N)$ such that
$$|\nabla w(x_0)| \leq K(N) \Big( \| \Delta w \|_{L^\infty(B(x_0, 1))} + \| w \|_{L^\infty(B(x_0, 1))} \Big),$$
so that, by inequalities \eqref{v-linfini} and \eqref{Deltaw-bound}, and definition \eqref{v-exp},
$$|\nabla w(x_0)| \leq 2 K(N) \Big( 1 + \frac{c^2}{4} \Big)^\frac{3}{2}.$$
Hence, by definition \eqref{v-exp} once more,
$$|\nabla v(x_0)| \leq |\nabla w(x_0)| + \frac{c}{2} |v(x_0)| \leq (2 K(N) + 1) \Big( 1 + \frac{c^2}{4} \Big)^\frac{3}{2},$$
which gives inequality \eqref{elinfinigrad0}. Finally, one invokes standard estimates for the laplacian to prove \eqref{ckestimates0}.
\end{proof}

\begin{lemma}[\cite{Tarquin1}]
\label{tarquini20}
Let $r > 0$, and assume that $v$ is a finite energy solution to \eqref{TWc} on $\R^N$. There exists some constant $K(N)$ such that for any $x_0 \in \R^N$,
\begin{equation}
\label {elinfini20}
\Big\| 1 - |v| \Big\|_{L^\infty (B(x_0, \frac{r}{2}))} \leq \max \Big\{ K(N) \Big( 1 + \frac{c^2}{4} \Big)^2 E \big( v, B(x_0, r) \big)^\frac{1}{N+2}, \frac{K(N)}{r^N} E \big( v, B(x_0, r) \big)^\frac{1}{2} \Big\},
\end{equation}
where $E(v, B(x_0, r)) \equiv \int_{B(x_0, r)} e(v)$.
\end{lemma}

\begin{proof}
Let $\eta = 1 - |v|^2$. By Lemma \ref{tarquini10}, the function $\eta$ is smooth on $\R^N$, and satisfies
$$\| \nabla \eta \|_{L^\infty(\R^N)} \leq 2 \| v \nabla v \|_{L^\infty(\R^N)} \leq K(N) \Big( 1 + \frac{c^2}{4} \Big)^2.$$
Let $\overline{x}$ be some point in $\overline{B(x_0, \frac{r}{2})}$ such that 
$$|\eta(\overline{x})| = \underset{y \in B(x_0, \frac{r}{2})} \sup |\eta(y)|.$$
We compute
$$|\eta(y)| \geq |\eta(\overline{x})| - K(N) \Big( 1 + \frac{c^2}{4} \Big)^2 |y - \overline{x}|,$$
so that
$$|\eta(y)| \geq \frac{|\eta(\overline{x})|}{2},$$
for any point $y \in B(\overline {x}, \mu)$ where $\mu= \frac{|\eta(\overline{x})|}{2 K(N) (1 + \frac{c^2}{4})^2}$. Therefore, we are led to
\begin{equation}
\label{inf-E-B}
\begin{split}
E(v , B(x_0, r)) \geq & \frac{1}{16} \int_{B \big( \overline{x}, \min \{ \mu, \frac{r}{2} \} \big)} \eta(\overline{x})^2 dy\\
= & \min \Big\{ \frac{|\eta(\overline{x})|^{N + 2} |B(x_0, 1)|}{2^{N + 4} K(N)^N (1 + \frac{c^2}{4})^{2 N}}, \frac{|\eta(\overline{x})|^2 r^N |B(x_0, 1)|}{2^{N + 4}} \Big\}.
\end{split}
\end{equation}
In conclusion,
$$\| 1 - |v| \|_{L^\infty(B(x_0, 1))} \leq \| \eta \|_{L^\infty(B(x_0, 1))} = |\eta(\overline{x})|,$$
so that inequality \eqref{elinfini20} follows from inequality \eqref{inf-E-B}.
\end{proof}

If we next consider the class of functions
$$\Lambda(\R^N) = \Big\{ v \in C^0(\R^N), E(v) < + \infty \ {\rm and} \ \exists R(v) > 0 \ {\rm s.t.} \ |v(x)| \geq \frac{1}{2}, \forall |x| \geq R(v) \Big\},$$
a rather direct consequence of Lemma \ref{tarquini20} is

\begin{cor}
\label{tarquosi}
Let $v$ be a finite energy solution to \eqref{TWc} on $\R^N$. Then $v$ belongs to $\Lambda(\R^N)$.
\end{cor}

\begin{proof}
Indeed, the energy of $v$ being finite, there exists some radius $R(v) > 1$ so that
$$E(v , B(x_0, 1)) \leq \int_{\R^N \setminus B(0, R(v) - 1)} e(v) \leq \frac{1}{\Big( 2 K(N) \big( 1 + \frac{c^2}{4} \big)^2 \Big)^{N + 2}}, \forall |x_0| \geq R(v),$$
where $K(N)$ is the constant in inequality \eqref{elinfini20}. Hence, $v$ belongs to $\Lambda(\R^N)$ in view of \eqref{elinfini20}.
\end{proof}

\subsection{Alternate definitions of the momentum}
\label{nouveaudef}

If $v \in \Lambda(\R^N)$, we may write, for $|x| > R(v)$,
\begin{equation}
\label{lereleve}
v = \varrho \exp i \varphi,
\end{equation}
where $\varphi$ is a real function on $\R^N \setminus B(0, R(v))$ defined modulo a multiple of $2 \pi$. Notice that, if $v$ may be written as in \eqref{lereleve}, then we have
$$\partial_j v = \Big( i \varrho \partial_j \varphi + \partial_j \varrho \Big) \exp i \varphi,$$
so that
\begin{equation}
\label{relevenergy0}
\langle i \partial_1 v , v \rangle = - \varrho^2 \partial_1 \varphi, \ {\rm and } \ e(v) = \frac{1}{2} \Big( |\nabla \varrho|^2 + \varrho^2 |\nabla \varphi|^2 \Big) + \frac{1}{4} \Big( 1 - \varrho^2 \Big)^2.
\end{equation}
In this context, the next elementary observation will be used in several places of our work.

\begin{lemma}
\label{colisee}
Let $\varrho$ and $\varphi$ be $C^1$ scalar functions on a domain $\boU$ in $\R^N$, such that $\varrho$ is positive. Set
$v = \varrho \exp i \varphi$. Then, we have the pointwise bound
$$\Big| (\varrho^2 - 1) \partial_1 \varphi \Big| \leq \frac{\sqrt{2}}{\varrho} e(v).$$
\end{lemma}

\begin{proof}
Notice that we have by \eqref{relevenergy0},
$$e(v) = \frac{1}{2} \Big( |\nabla \varrho|^2 + \varrho^2 |\nabla \varphi|^2 \Big) + \frac{1}{4} \Big( 1 - \varrho^2 \Big)^2 \geq \frac{1}{2} \Big( \varrho^2 |\partial_1 \varphi|^2 + \frac{1}{2} (1 - \varrho^2)^2 \Big).$$
The conclusion then follows from the inequality $a b \leq \frac{1}{2} (a^2 + b^2)$ applied to $a = \frac{1}{\sqrt{2}} (\varrho^2 - 1)$ and $b = \varrho \partial_1 \varphi$.
\end{proof}
 
The previous observations lead as in \cite{BetOrSm1, Graveja3} to an alternate definition of the momentum on the space $\Lambda(\R^N)$. For that purpose, consider the function
$$g(v) = \langle i \partial_1 v, v \rangle + \partial_1 \big( (1 - \chi) \varphi \big),$$
where $v = \varrho \exp i \varphi$ on $\R^N \setminus B(0, R(v))$, and $\chi$ is an arbitrary smooth function with compact support such
that $\chi \equiv 1$ on $B(0, R(v))$ and $0 \leq \chi \leq 1$. It follows from \eqref{relevenergy0} that

\begin{lemma}
\label{gdev}
If $v$ belongs to $\Lambda(\R^N)$, then $g(v)$ belongs to $L^1(\R^N)$. Moreover the integral
$$\tilde{p}(v) \equiv \frac{1}{2} \int_{\R^N} g(v) = \frac{1}{2} \int_{\R^N} \Big( \langle i \partial_1 v, v \rangle + \partial_1 \big( (1 - \chi) \varphi \big) \Big)$$
does not depend on the choice of $\chi$.
\end{lemma}

\begin{proof}
As a consequence of \eqref{relevenergy0}, we verify that
$$g(v) = (1 - \varrho^2) \partial_1 \varphi, \ {\rm on} \ \R^N \setminus \supp(\chi),$$ 
so that in particular, it follows from Lemma \ref{colisee} that
$$|g(v)| \leq 2 \sqrt{2} e(v) \ {\rm on} \ \R^N \setminus \supp(\chi),$$
and hence, since $v$ is smooth on $\R^N$, and its energy is bounded, the function $g(v)$ is integrable on $\R^N$. The last assertion is a direct consequence of the integration by parts formula.
\end{proof}

\subsection{Decay at infinity}
\label{propdecay}

We will use in several places decay properties of solutions which have been established in \cite{Graveja1, Graveja3, Graveja4}. They play, among other things, a central role in our concentration-compactness results.

\begin{prop}[\cite{Graveja1, Graveja3, Graveja4}]
\label{superdecay}
Let $v$ be a finite energy solution to \eqref{TWc}.\\
i) There exists a constant $v_\infty$, such that $|v_\infty| = 1$ and
$$v(x) \to v_\infty, \ {\rm as} \ |x| \to \infty.$$
Without loss of generality, we may assume $v_\infty = 1$.\\
ii) Assume $c(v) < \sqrt{2}$. Then, there exists some constant $K > 0$ depending only on $c(v)$, $E(v)$ and the dimension $N$, such that the following estimates hold for any $x \in \R^N$,
\begin{equation}
\label{decay-sub-cart}
\begin{split}
& |\Im(v(x))| \leq \frac{K}{1 + |x|^{N-1}}, \ |\Re(v(x)) - 1| \leq \frac{K}{1+|x|^N},\\
& |\nabla \Im(v(x))| \leq \frac{K}{1+|x|^N}, \ |\nabla \Re(v(x))| \leq \frac{K}{1+|x|^{N+1}}.
\end{split}
\end{equation}
iii) Assume $N = 3$ and $c(v) = \sqrt{2}$. Then, $\Re(v) - 1$ and $\nabla \Im(v)$ belong to $L^p(\R^3)$, for any $p > \frac{5}{3}$, $\nabla \Re(v)$ belongs to $L^p(\R^3)$, for any $p > \frac{5}{4}$, whereas $\Im(v)$ belongs to $L^p(\R^3)$, for any $p > \frac{15}{4}$.
\end{prop}

A first consequence is 
 
\begin{cor}
\label{bienpose}
Let $v$ be a finite energy solution to \eqref{TWc}, and assume $v_\infty = 1$. Then $v$ belongs to $W(\R^N)$.
\end{cor}

\begin{proof}
In view of definitions \eqref{def-V} and \eqref{def-W}, Corollary \ref{bienpose} directly follows from the decay estimates (ii) and (iii) of Proposition \ref{superdecay}.
\end{proof}

\begin{remark}
\label{decay-pol}
Since any finite energy solution $v$ to \eqref{TWc} has a limit $v_\infty$ at infinity, we may write $v = \varrho \exp i \varphi$ outside some ball $B(0, R)$, for some $R > 0$, where $\varphi$ is a smooth function on $\R^N \setminus B(0, R)$, which is defined up to an integer multiple of $2 \pi$. Moreover the function $\varphi$ has a limit at infinity $\varphi_\infty$, which we may take equal to $0$, if we assume that $v_\infty = 1$. The statements given in \cite{Graveja1, Graveja3, Graveja4} are actually expressed in terms of the real functions $\varrho$ and $\varphi$ as follows:\\
(i) if $0 \leq c(v) < \sqrt{2}$, there exists some constant $K > 0$ depending only on $c(v)$, $E(v)$ and $N$ such that
\begin{equation}
\label{decay-sub-pol}
\begin{split}
|\varphi(x)| \leq \frac{K}{1+|x|^{N-1}}, & \ |1 - \varrho(x)| \leq \frac{K}{1+|x|^{N}},\\
|\nabla \varphi(x)| \leq \frac{K}{1+|x|^{N}}, & \ |\nabla \varrho(x)| \leq \frac{K}{1+|x|^{N+1}}.
\end{split}
\end{equation}
(ii) if $c(v) = \sqrt{2}$, the function $\varphi$ belongs to $L^p(\R^3 \setminus B(0, R))$, for any $p > \frac{15}{4}$, $\varrho-1$ and $\nabla \varphi$ belong to $L^p(\R^3 \setminus B(0, R))$, for any $p > \frac{5}{3}$, whereas $\nabla \varrho$ belongs to $L^p(\R^3 \setminus B(0, R))$, for any $p > \frac{5}{4}$.

The previous inequalities are easily seen to be equivalent to those given in Proposition \ref{superdecay}. Indeed, we have $v = \varrho \cos(\varphi) + i \varrho \sin(\varphi)$, so that $\Re(v) - 1 = \varrho \cos(\varphi) - 1$ and $\Im(v) = \varrho \sin(\varphi)$, and hence
\begin{align*}
|\Re(v) - 1| & \leq K \big( |\varrho - 1| + \varphi^2 \big), \ |\Im(v)| \leq K |\varphi|,\\
|\nabla \Re(v)| &\leq K \big( |\nabla \varrho| + |\varphi| |\nabla \varphi| \big), \ |\nabla \Im(v)| \leq k \big( |\nabla \varphi| + |\varphi| |\nabla \varrho| \big),
\end{align*}
where $K > 0$ is some constant.
\end{remark}

A remarkable consequence of Proposition \ref{superdecay} is

\begin{prop}
\label{cestlemoment}
Let $v$ be a finite energy solution to \eqref{TWc} on $\R^N$. Then, we have
$$\tilde{p}(v) = p(v).$$
\end{prop}

\begin{proof}
Let $R(v) > 0$ be such that $|v| \geq \frac{1}{2}$ on $\R^N \setminus B(0, R(v))$. As in Proposition \ref{superdecay} and Remark \ref{decay-pol}, we may assume without loss of generality that $v_\infty = 1$ and $\varphi_\infty = 0$. If $x$ is sufficiently large, the expansion of the $\sin$ function yields
$$\bigg| \frac{\Im(v(x))}{\varrho(x)} - \varphi(x) \bigg| \leq \frac{|\varphi(x)|^3}{6}.$$
Let $R > R(v)$ be sufficiently large. We have, integrating by parts
$$\int_{B(0, R)} \langle i \partial_1 v, 1 \rangle = - \frac{1}{R} \int_{\partial B(0, R)} \Im(v(x)) x_1 dx, \ {\rm and} \ \int_{B(0, R)} \partial_1 \big( (1 - \chi) \varphi \big) = \frac{1}{R} \int_{\partial B(0, R)} \varphi(x) x_1 dx,$$
so that it follows 
$$\int_{B(0, R)} \Big( \langle i \partial_1 v, v - 1 \rangle - g(v) \Big) = \frac{1}{R} \int_{\partial B(0, R)} \Big( \Im(v(x)) - \varphi(x) \Big) x_1 dx.$$
We write $\Im(v) - \varphi = \big( \frac{\Im(v)}{\varrho} - \varphi \big) + \Im(v) \frac{\varrho - 1}{\varrho}$, so that
\begin{equation}
\label{grouic}
\bigg| \int_{B(0, R)} \Big( \langle i \partial_1 v, v - 1 \rangle - g(v) \Big) \bigg| \leq \int_{\partial B(0, R)} \Big( \frac{|\varphi|^3}{6} + 2 |\Im(v)| |\varrho - 1| \Big).
\end{equation}
We next distinguish two cases.

\begin{case}
{\bf ${\bf N = 2}$ or ${\bf N = 3}$, and ${\bf c(v) < \sqrt{2}}$.} It follows from \eqref{decay-sub-cart} and \eqref{decay-sub-pol} that
$$\frac{|\varphi(x)|^3}{6} + 2 |\Im(v(x))| |\varrho(x) - 1| \leq \frac{K}{1 + |x|^N},$$
so that \eqref{grouic} yields
$$\int_{B(0, R)} \Big| \langle i \partial_1 v, v - 1 \rangle - g(v) \Big| \to 0, \ {\rm as} \ R \to + \infty,$$
which yields the conclusion.
\end{case}

\begin{case}
{\bf ${\bf N = 3}$ and ${\bf c(v) = \sqrt{2}}$.} It follows from Remark \ref{decay-pol} that $\varphi$ belongs to $L^q(\R^3 \setminus B(0, R(v)))$ for any $q > \frac{15}{4}$, and $\varrho - 1$ belongs to $L^q(\R^3 \setminus B(0, R(v)))$ for any $q > \frac{5}{3}$, so that by Proposition \ref{superdecay}, the function $f \equiv \frac{|\varphi|^3}{6} + 2 |\Im(v)| |\varrho - 1|$ belongs to $L^q(\R^3 \setminus B(0, R(v)))$ for any $q > \frac{5}{4}$. Given $R > R(v)$ and $q > \frac{5}{4}$ to be determined later, we may find some $R \leq R' \leq 2 R$, and some constant $K(q)$ only depending on $q$, such that
$$\int_{\partial B(0, R')} f^q \leq \frac{K(q)}{R},$$
so that by H\"older's inequality, we are led to
$$\int_{\partial B(0, R')} f \leq K(q) R^{2 - \frac{3}{q}}.$$
Choosing $q = \frac{4}{3}$, we obtain that $\int_{\partial B(0, R')} f \to 0$, as $R \to + \infty$. This yields by \eqref{grouic}
$$\bigg| \int_{B(0, R')} \Big( \langle i \partial_1 v, v - 1 \rangle - g(v) \Big) \bigg| \to 0, \ {\rm as} \ R \to + \infty,$$
which yields the conclusion again, since the integrand is integrable by Lemma \ref{gdev} and Corollary \ref{bienpose}. 
\end{case}
\end{proof}

\subsection{Pohozaev's type identities}

\begin{lemma}
\label{encorepoho}
Let $v$ be a finite energy solution to \eqref{TWc} on $\R^N$, with speed $c = c(v)$. We have the identities
$$E(v) = \int_{\R^N} |\partial_1 v|^2,$$
and for any $2 \leq j \leq N$,
$$E(v) = \int_{\R^N} |\partial_j v|^2 + c(v) p(v).$$
Moreover, if $c(v) > 0$ and $v$ is not constant, then $p(v) > 0$.
\end{lemma}

\begin{proof}
The first identity was established in \cite{Graveja2}, whereas, concerning the second identity, it was proved there that for any $2 \leq j \leq N$,
$$E(v) = \int_{\R^N} |\partial_j v|^2 + c(v) \tilde{p}(v).$$
The conclusion then follows from Proposition \ref{cestlemoment}.
\end{proof}

Notice that adding the identities in Lemma \ref{encorepoho} we obtain 
\begin{equation}
\label{theverypoho}
\frac{N - 2}{2} \int_{\R^N} |\nabla v|^2 + \frac{N}{4} \int_{\R^N} (1 - |v|^2)^2 - c(v) (N - 1) p(v) = 0.
\end{equation} 

Notice also, that introducing the quantities $\Sigma(v) = \sqrt{2} p(v)- E(v)$ and $\varepsilon(v) = \sqrt{2 - c(v)^2}$, the second identity in Lemma \ref{encorepoho} may be recast as 
\begin{equation}
\label{acapulco}
\int_{\R^N} |\partial_j v|^2 + \Sigma(v) = \Big( \sqrt{2} - \sqrt{2-\varepsilon(v)^2} \Big) p(v) = \frac{\varepsilon(v)^2}{\sqrt{2} + \sqrt{2 - \varepsilon(v)^2}} p(v).
\end{equation}

\begin{cor}
\label{cordepoho}
Let $v$ be a finite energy solution to \eqref{TWc} on $\R^N$, with speed $c = \sqrt{2}$ and such that $\Sigma (v) \geq 0$. Then, v is a constant.
\end{cor}
 
\begin{proof}
Since $\varepsilon(v) = 0$ and $\Sigma(v) \geq 0$, identity \eqref{acapulco} implies that, for any $2 \leq j \leq N$,
$$\int_{\R^N} |\partial_j v|^2 = 0,$$
so that $v$ depends only on the $x_1$ variable. Since the energy is finite, this is impossible, unless $v$ is constant.
\end{proof}

Notice more generally that, if $\Sigma(v) > 0$, then identity \eqref{acapulco} gives
$$\int_{\R^N} |\partial_j v|^2 \leq \varepsilon(v) p(v), \forall \, 2 \leq j \leq N.$$
In connection with the previous inequality, the next result gives a more quantitative version of Corollary \ref{cordepoho}.

\begin{lemma}
\label{bonlemme}
Let $v$ be a finite energy solution to \eqref{TWc} on $\R^N$. Then, there exists a constant $K(c) > 0$, possibly depending on $c$, such that
$$\| \eta \|_{L^\infty(\R^N)}^{N + 1} \leq K(c) \int_{\R^N} \Big( \lambda |\partial_j v|^2 + \frac{\eta^2}{\lambda} \Big),$$
for any $2 \leq j \leq N$, and for any $\lambda > 0$. In particular, we have
$$\| \eta \|_{L^\infty(\R^N)}^{N + 1} \leq K(c) \Big( \lambda \Big( \varepsilon(v) p(v) - \Sigma(v) \Big) + \frac{E(v)}{\lambda} \Big).$$
\end{lemma}

\begin{proof}
Set $\eta_\infty = \| \eta \|_{L^\infty(\R^N)}$. We may assume without loss of generality, that $|\eta(0)| = \eta_\infty$. In view of uniform bound \eqref{elinfinigrad0}, there exists some constant $K(c)$ depending only on $c$ such that 
\begin{equation}
\label{bouffonvert}
|\eta(x)| \geq \frac{\eta_{\infty}}{2}, \forall x \in B \Big( 0, \frac{\eta_\infty}{2 K(c)} \Big).
\end{equation}
We next consider for any point $a = (a_1, \ldots, a_{j - 1}, 0, a_{j + 1}, \ldots, a_N)$, the line $D_j(a)$ parallel to the axis $x_j$, that is the set $D_j(a) = \{ \a_j(x) \equiv (a_1, \ldots, a_{j - 1}, x, a_{j + 1}, \ldots, a_N), x \in \R \}$. We claim that
\begin{equation}
\label{greengoblin}
|\eta_\infty|^2 \leq 4 \int_{D_j(a)} \Big( \lambda (\partial_j \eta)^2 + \frac{\eta^2}{\lambda} \Big),
\end{equation}
for any $a = (a_1, \ldots, a_{j - 1}, 0, a_{j + 1}, \ldots, a_N) \in B(0, \frac{\eta_\infty}{2 K(c)})$. Indeed, since $\eta(x) \to 0$, as $|x| \to + \infty$, we have by integration,
\begin{equation}
\begin{split}
\eta(a)^2 & = 2 \int_{- \infty}^0 \partial_j \eta(\a_j(x)) \eta(\a_j(x)) dx\\
& \leq \int_\R \Big( \lambda \big( \partial_j \eta(\a_j(x)) \big)^2 + \frac{\eta(\a_j(x))^2}{\lambda} \Big) dx.
\end{split} 
\end{equation}
Invoking inequality \eqref{bouffonvert}, one derives \eqref{greengoblin}. The conclusion then follows integrating inequality \eqref{bouffonvert} on $a = (a_1, \ldots, a_{j - 1}, 0, a_{j + 1}, \ldots, a_N) \in B(0, \frac{\eta_\infty}{2 K(c)})$.
\end{proof}
 
\subsection{Analyticity}

The proofs of Theorem \ref{symetrie} and Lemma \ref{gueri} rely on the following result, which is of independent interest. 

\begin{prop} 
\label{analyticity} 
Let $v$ be a finite energy solution of \eqref{TWc} on $\R^N$, with $0 \leq c < \sqrt{2}$. Then, each component of $v$ is real-analytic on $\R^N$.
\end{prop} 

This is a consequence of the next more general result.

\begin{theorem}
\label{analyticity0}
Let $N \geq 2$ and let $v$ be a finite energy solution to \eqref{TWc} on $\R^N$, with $0 \leq c < \sqrt{2}$. There exists some number $\lambda_0 > 0$, possibly depending on $v$, such that $\Re(v)$ and $\Im(v)$ extend to analytical functions on the cylinder $\boC_{\lambda_0} = \{ z \in \C^N, |\Im(z)| < \lambda_0 \}$.
\end{theorem}

\begin{proof}
The argument is reminiscent of \cite{BonaLi1, BonaLi2} (see also \cite{KatoPip1, Maris1, Maris2}). The idea is to prove the convergence of the Taylor series of $v$
\begin{equation}
\label{taylor}
T_{v,x}(z) = \sum_{\alpha \in \N^N} \frac{\partial^\alpha v(x)}{\alpha!} (z - x)^\alpha,
\end{equation}
on a complex neighbourhood of an arbitrary point $x \in \R^N$, the required estimates for the derivatives being provided by the partial differential equation, standard $L^q$-multiplier theory, and Sobolev's embedding theorem. We apply here this strategy to the functions $v_1 = \Re(v) - 1$ and $v_2 = \Im(v)$, which satisfy the equations
\begin{align}
\label{eq-v1}
\Delta^2 v_1 - 2 \Delta v_1 + c^2 \partial_1^2 v_1 = & \Delta F_1(v_1, v_2) + c \partial_1 F_2(v_1, v_2),\\
\label{eq-v2}
\Delta^2 v_2 - 2 \Delta v_2 + c^2 \partial_1^2 v_2 = & - c \partial_1 F_1(v_1, v_2) - 2 F_2(v_1, v_2) + \Delta F_2(v_1, v_2),
\end{align}
where the functions $F_1$ and $F_2$ are defined from $\C^2$ to $\C$ by
\begin{equation}
\label{non-lin-F}
F_1(z_1, z_2) = 3 z_1^2 + z_2^2 + z_1^3 + z_1 z_2^2, \ F_2(z_1, z_2) = 2 z_1 z_2 + z_1^2 z_2 + z_2^3.
\end{equation}
Indeed, by equation \eqref{TWc},
\begin{align}
\label{eq-v1-0}
\Delta v_1 - c \partial_1 v_2 - 2 v_1 = & F_1(v_1, v_2),\\
\label{eq-v2-0}
\Delta v_2 + c \partial_1 v_1 = & F_2(v_1, v_2),
\end{align}
so that equation \eqref{eq-v1} is derived applying the differential operator $\Delta$ to equation \eqref{eq-v1-0}, the operator $c \partial_1$ to equation \eqref{eq-v2-0}, and adding the corresponding relations, whereas equation \eqref{eq-v2} is derived applying the differential operator $\Delta - 2$ to equation \eqref{eq-v2-0}, the operator $- c \partial_1$ to equation \eqref{eq-v1-0}, and adding the corresponding relations. Taking the Fourier transforms of equations \eqref{eq-v1} and \eqref{eq-v2} and denoting $H_{j,k}$, $H_{1,j,k}$ and $K_{j,k}$, the kernels defined by
$$\widehat{H_{j,k}}(\xi) = \frac{\xi_j \xi_k |\xi|^2}{|\xi|^4 + 2 |\xi|^2 - c^2 \xi_1^2}, \ \widehat{H_{1,j,k}}(\xi) = \frac{\xi_1 \xi_j \xi_k}{|\xi|^4 + 2 |\xi|^2 - c^2 \xi_1^2}, \ \widehat{K_{j,k}}(\xi) = \frac{\xi_j \xi_k (2 + |\xi|^2)}{|\xi|^4 + 2 |\xi|^2 - c^2 \xi_1^2},$$
for any $1 \leq j, k \leq N$, equations \eqref{eq-v1} and \eqref{eq-v2} may be recast as
\begin{align}
\label{eq-fou-v1}
\partial^2_{jk} v_1 = H_{j,k}*F_1(v_1, v_2) - i c H_{1,j,k} * F_2(v_1, v_2),\\
\label{eq-fou-v2}
\partial^2_{jk} v_2 = i c H_{1,j,k}*F_1(v_1, v_2) + K_{j,k} * F_2(v_1, v_2).
\end{align}
In order to compute $L^q$-estimates of $\partial^2_{jk} v_1$ and $\partial^2_{jk} v_2$, we show that $\widehat{H_{j,k}}$, $\widehat{H_{1,j,k}}$ and $\widehat{K_{j,k}}$ are $L^q$-multipliers for any $1 < q < + \infty$. For that purpose, we use Theorem 8 of \cite{Lizorki1}
\footnote{Estimate \eqref{ibanez} in Theorem \ref{thelizo} is more precisely a consequence of the proof of Theorem 8, and Lemma 6 of \cite{Lizorki1}.}.

\begin{theorem}[\cite{Lizorki1}]
\label{thelizo}
Let $0 \leq \alpha < 1$. Consider a bounded function $\widehat{K}$ in $C^N(\R^N \setminus \{ 0 \})$, and assume that
\begin{equation}
\label{Lizo-assu-1}
\underset{j = 1}{\overset{N}{\Pi}} \xi_j^{\alpha + k_j} \partial_1^{k_1} \ldots \partial_N^{k_N} \widehat{K}(\xi) \in L^\infty(\R^N)
\end{equation}
for any $(k_1, \ldots, k_N) \in \{ 0, 1 \}^N$ such that $k_1 + \ldots + k_N \leq N$. Then, $\widehat{K}$ is a multiplier from $L^q(\R^N)$ to $L^\frac{q}{1 - \alpha q}(\R^N)$ for any $1 < q < \frac{1}{\alpha}$ (with the usual convention $\frac{1}{0} = + \infty$). More precisely, there exists a constant $K(q)$ only depending on $q$, such that the multiplier operator $\boK$ defined by
$$\widehat{\boK(f)}(\xi) = \widehat{K}(\xi) \widehat{f}(\xi),$$
satisfies for any $1 < q < \frac{1}{\alpha}$,
\begin{equation}
\label{ibanez}
\| \boK(f) \|_{L^\frac{q}{1 - \alpha q}(\R^N)} \leq K(q) M(\widehat{K}) \| f \|_{L^q(\R^N)},
\end{equation}
where
\begin{equation}
\label{mignoni}
M(\widehat{K}) \equiv \sup \Big\{ \underset{j = 1}{\overset{N}{\Pi}} |\xi_j|^{\alpha + k_j} \Big| \partial_1^{k_1} \ldots \partial_N^{k_N} \widehat{K}(\xi) \Big|, \xi \in \R^N, (k_1, \ldots, k_N) \in \{ 0, 1 \}^N, k_1 + \ldots + k_N \leq N \Big\}.
\end{equation}
\end{theorem}

By an inductive argument, the kernels $\widehat{H_{j,k}}$, $\widehat{H_{1,j,k}}$ and $\widehat{K_{j,k}}$ satisfy assumption \eqref{Lizo-assu-1} for $\alpha = 0$ and any $(k_1, \ldots, k_N) \in \{ 0, 1 \}^N$ such that $k_1 + \ldots + k_N \leq N$. Therefore, they are $L^q$-multipliers for any $1 < q < + \infty$. This implies

\setcounter{step}{0}
\begin{step}
\label{Lq-estim}
Let $1 \leq j, k \leq N$, $\alpha \in \N^N$ and $1 < q < + \infty$. There exists some positive number $K_1(q)$, possibly depending on $q$, but not on $\alpha$, such that
\begin{equation}
\label{estim-Lq}
\| \partial^\alpha \partial^2_{jk} v_1 \|_{L^q(\R^N)} + \| \partial^\alpha \partial^2_{jk} v_2 \|_{L^q(\R^N)} \leq K_1(q) \Big( \| \partial^\alpha F_1(v_1, v_2) \|_{L^q(\R^N)} + \| \partial^\alpha F_2(v_1, v_2) \|_{L^q(\R^N)} \Big).
\end{equation}
\end{step}

By \cite{Graveja3}, $\partial^\alpha v_1$ and $\partial^\alpha v_2$ belong to $L^q(\R^N)$ for any $\frac{N}{N-1} < q < + \infty$, if $\alpha = 0$, $1 < q < + \infty$, elsewhere. Using the chain rule, it follows that $\partial^\alpha F_1(v_1, v_2)$ and $\partial^\alpha F_2(v_1, v_2)$ are in $L^q(\R^N)$ for any $1 < q < + \infty$. On the other hand, by \eqref{eq-fou-v1} and \eqref{eq-fou-v2},
\begin{align*}
\partial^\alpha \partial^2_{jk} v_1 = & H_{j,k}*\partial^\alpha F_1(v_1, v_2) - i c H_{1,j,k}*\partial^\alpha F_2(v_1, v_2),\\
\partial^\alpha \partial^2_{jk} v_2 = & i c H_{1,j,k}*\partial^\alpha F_1(v_1, v_2) + K_{j,k}*\partial^\alpha F_2(v_1, v_2),
\end{align*}
and inequality \eqref{estim-Lq} follows from the fact that 
$\widehat{H_{j,k}}$, $\widehat{H_{1,j,k}}$ and $\widehat{K_{j,k}}$ are $L^q$-multipliers for any $1 < q < + \infty$. We are now in position to obtain uniform estimates of $\partial^\alpha v_1$ and $\partial^\alpha v_2$.

\begin{step}
\label{unif-estim}
Let $1 \leq j \leq N$, $\alpha \in \N^N$ and $\frac{N}{2} < q < + \infty$. There exist some positive numbers $K_2(q)$ and $K_3(q)$, possibly depending on $q$, but not on $\alpha$, such that
\begin{equation}
\label{estim-unif}
\begin{split}
\| \partial^\alpha v_1 \|_{L^\infty(\R^N)} + \| \partial^\alpha v_2 \|_{L^\infty(\R^N)} \leq & K_2(q) F_q(\alpha),\\
\| \partial^\alpha \partial_j v_1 \|_{L^q(\R^N)} + \| \partial^\alpha \partial_j v_2 \|_{L^q(\R^N)} \leq & K_3(q) F_q(\alpha),
\end{split}
\end{equation}
where we have set
\begin{equation}
\label{F-q-alpha}
F_q(\alpha) = \underset{0 \leq \beta \leq \alpha}{\max} \Big( \| \partial^\beta F_1(v_1, v_2) \|_{L^q(\R^N)} + \| \partial^\beta F_2(v_1, v_2) \|_{L^q(\R^N)} \Big).
\end{equation}
\end{step}

By Sobolev's embedding theorem and inequality \eqref{estim-Lq}, we have
$$\| \partial^\alpha v_1 \|_{L^\infty(\R^N)} \leq K_S(q) \Big( \| \partial^\alpha v_1 \|_{L^q(\R^N)} + \| \partial^\alpha d^2 v_1 \|_{L^q(\R^N)} \Big) \leq 2 K_S(q) K_1(q) F_q(\alpha).$$
Using the same argument for $\partial^\alpha v_2$, the first estimate of \eqref{estim-unif} holds with $K_2(q) = 4 K_S(q) K_1(q)$. On the other hand, using Gagliardo-Nirenberg's inequality and inequality \eqref{estim-Lq}, we are led to
$$\| \partial^\alpha \partial_j v_1 \|_{L^q(\R^N)} \leq K_{GN}(q) \| \partial^\alpha v_1 \|_{L^q(\R^N)}^\frac{1}{2} \| \partial^\alpha d^2 v_1 \|_{L^q(\R^N)}^\frac{1}{2} \leq K_{GN}(q) K_1(q) F_q(\alpha),$$
and the second inequality of \eqref{estim-unif} also holds with $K_3(q) = 2 K_{GN}(q) K_1(q)$.

We now come back to the convergence of the Taylor series $T_{v_1, x}(z)$ and $T_{v_2, x}(z)$, defined in \eqref{taylor}. Using the uniform estimates of Step \ref{unif-estim}, it suffices to prove the convergence of the series
\begin{equation}
\label{Sq-x0}
S_{q,x_0}(z) = \underset{\alpha \in \N^N}{\sum} \frac{F_q(\alpha)}{\alpha !}| z - x_0|^{|\alpha|},
\end{equation}
for $z$ sufficiently close to $x_0$, and some suitable exponent $q$. This follows from the next estimate.

\begin{step}
\label{F-alpha-estim}
Let $\alpha \in \N^N$ and $\frac{N}{2} < q < + \infty$. There exists some positive number $K_4(q)$, possibly depending on $q$, but not on $\alpha$, such that
\begin{equation}
\label{estim-F-alpha}
F_q(\alpha) \leq K_4(q)^{|\alpha|} \alpha^{\tilde{\alpha}},
\end{equation}
where we have set $\tilde{\alpha} = (\max \{ \alpha_1 - 1, 0 \}, \ldots, \max \{ \alpha_N - 1, 0 \})$.
\end{step}

The proof is by induction on $l = |\alpha|$. Inequality \eqref{estim-F-alpha} being valid for $0 \leq l \leq 5$ and any constant $K_4(q)$ sufficiently large, we assume that it holds for any multi-index $\alpha$ such that $|\alpha| \leq l$, and consider some multi-index $\alpha = (\beta_1, \ldots, \beta_j + 1, \ldots, \beta_N)$ such that $|\alpha| = |\beta| + 1 = l + 1$. Using the inductive assumption for $\beta$, we claim that

\begin{claim} 
\label{tri-estim}
Let $(a, b, c) \in \{ 1, 2 \}^3$. Then,
\begin{equation}
\label{estim-tri}
\begin{split}
\| \partial_j \partial^\beta (v_a v_b) \|_{L^q(\R^N)} \leq & 2^{2N+1} K_2(q) K_3(q) K_4(q)^l \beta^{\tilde{\beta}},\\
\| \partial_j \partial^\beta (v_a v_b v_c) \|_{L^q(\R^N)} \leq & 4^{2N+1} K_2(q) K_3(q)^2 K_4(q)^l \beta^{\tilde{\beta}}.
\end{split}
\end{equation}
\end{claim}

We postpone the proof of Claim \ref{tri-estim} and first complete the proof of Step \ref{F-alpha-estim}. By definitions \eqref{non-lin-F} and estimates \eqref{estim-tri},
$$\| \partial_j \partial^\beta F_1(v_1, v_2) \|_{L^q(\R^N)} + \| \partial_j \partial^\beta F_2(v_1, v_2) \|_{L^q(\R^N)} \leq 4^{2 N + 2} K_2(q) K_3(q) (1 + K_3(q)) K_4(q)^l \beta^{\tilde{\beta}},$$
so that choosing $K_4(q) = 4^{2 N + 2} K_2(q) K_3(q) (1 + K_3(q))$, and using the above definition of $\alpha$, we are led to
$$\| \partial^\alpha F_1(v_1, v_2) \|_{L^q(\R^N)} + \| \partial^\alpha F_2(v_1, v_2) \|_{L^q(\R^N)} \leq K_4(q)^{l+1} \beta^{\tilde{\beta}},$$
which finally gives
$$F_q(\alpha) \leq K_4(q)^{l+1} \alpha^{\tilde{\alpha}},$$
and completes the inductive proof of Step \ref{F-alpha-estim}. 

\begin{proof}[Proof of Claim \ref{tri-estim}]
Let $(a, b) \in \{ 1, 2 \}^2$. The chain rule gives
$$\| \partial_j \partial^\beta (v_a v_b) \|_{L^q} \leq \underset{0 \leq \gamma \leq \beta}{\sum} \frac{\beta!}{\gamma! (\beta - \gamma)!} \Big( \| \partial_j \partial^\gamma v_a \|_{L^q} \| \partial^{\beta - \gamma} v_b \|_{L^\infty} + \| \partial^\gamma v_a \|_{L^\infty} \| \partial_j \partial^{\beta - \gamma} v_b \|_{L^q} \Big),$$
so that using estimates \eqref{estim-unif},
$$\| \partial_j \partial^\beta (v_a v_b) \|_{L^q(\R^N)} \leq 2 \underset{0 \leq \gamma \leq \beta}{\sum} \frac{\beta !}{\gamma ! (\beta - \gamma) !} K_2(q) K_3(q) F_q(\gamma) F_q(\beta - \gamma).$$
Hence, by the inductive assumption,
\begin{equation}
\label{bi-estim}
\| \partial_j \partial^\beta (v_a v_b) \|_{L^q(\R^N)} \leq 2 K_2(q) K_3(q) K_4(q)^l \underset{0 \leq \gamma \leq \beta}{\sum} \frac{\beta !}{\gamma ! (\beta - \gamma) !} \gamma^{\tilde{\gamma}} (\beta - \gamma)^{\widetilde{\beta - \gamma}}.
\end{equation}
At this stage, we require the next elementary lemma.

\begin{lemma}
\label{bauges}
Let $\beta \in \N^N \setminus \{ 0 \}$. Then,
\begin{equation}
\label{culoz}
\underset{0 \leq \gamma \leq \beta}{\sum} \frac{\beta !}{\gamma ! (\beta - \gamma) !} \gamma^{\tilde{\gamma}} (\beta - \gamma)^{\widetilde{\beta - \gamma}} \leq 4^N \beta^{\tilde{\beta}}.
\end{equation}
\end{lemma}

\begin{proof}[Proof of Lemma \ref{bauges}]
Assume first that $N = 1$ and $\beta \geq 2$. By Abel's identity (as in \cite{Maris1}),
$$\underset{0 \leq \gamma \leq \beta}{\sum} \frac{\beta !}{\gamma ! (\beta - \gamma) !} \gamma^{\tilde{\gamma}} (\beta - \gamma)^{\widetilde{\beta - \gamma}} = 2 \beta^{\beta - 1} + \underset{1 \leq \gamma \leq \beta - 1}{\sum} \frac{\beta !}{\gamma ! (\beta - \gamma) !} \gamma^{\gamma - 1} (\beta - \gamma)^{\beta - \gamma - 1} = 2 (2 \beta - 1) \beta^{\beta - 2},$$
so that
$$\underset{0 \leq \gamma \leq \beta}{\sum} \frac{\beta !}{\gamma ! (\beta - \gamma) !} \gamma^{\tilde{\gamma}} (\beta - \gamma)^{\widetilde{\beta - \gamma}} \leq 4 \beta^{\beta - 1} = 4 \beta^{\tilde{\beta}}.$$
Since this inequality is straightforward in the cases $\beta = 0$ or $\beta = 1$, \eqref{culoz} is proved for $N = 1$. If $N > 1$, we invoke a little algebra and the one-dimensional case to conclude by the relations
$$\underset{0 \leq \gamma \leq \beta}{\sum} \frac{\beta !}{\gamma ! (\beta - \gamma) !} \gamma^{\tilde{\gamma}} (\beta - \gamma)^{\widetilde{\beta - \gamma}} = \underset{i = 1}{\overset{N}{\prod}} \bigg( \underset{0 \leq \gamma_i \leq \beta_i}{\sum} \frac{\beta_i !}{\gamma_i ! (\beta_i - \gamma_i) !} \gamma_i^{\tilde{\gamma_i}} (\beta_i - \gamma_i)^{\widetilde{\beta_i - \gamma_i}} \bigg) \leq \underset{i = 1}{\overset{N}{\prod}} \Big( \beta_i^{\tilde{\beta_i}} \Big) = 4^N \beta^{\tilde{\beta}}.$$
\end{proof}

Using Lemma \ref{bauges}, we are in position to prove Claim \ref{tri-estim}. Inequalities \eqref{bi-estim} and \eqref{culoz} yields the first inequality of \eqref{estim-tri}. The second follows using the chain rule, uniform estimates \eqref{estim-unif} and Lemma \ref{bauges} twice.
\end{proof}

\begin{step}
\label{taylor-conv}
The Taylor series $T_{v_1, x_0}(z)$ and $T_{v_2, x_0}(z)$ are uniformly convergent with respect to $x_0 \in \R^N$ on the set $\boD(x_0, \lambda_0) = \{ z \in \C^N, |z - x_0| < \lambda_0 \}$, where $\lambda_0 = \frac{e}{K_4(N)}$.
\end{step}

We choose $q = N$. By Stirling's formula, the series $\underset{\alpha \geq 0}{\sum} \frac{K_4(N)^{|\alpha|} \alpha^\alpha}{\alpha !} z^\alpha$ converges on $\boD(0, \lambda_0)$, so that using estimate \eqref{estim-F-alpha}, the series $S_{N, x_0}(z)$ converges for any $z \in \boD(x_0, \lambda_0)$, uniformly with respect to $x_0 \in \R^N$. By assertion \eqref{Sq-x0}, $T_{v_1, x_0}(z)$ and $T_{v_2, x_0}(z)$ converge the same way.

By Step \ref{taylor-conv}, we conclude that there exist some positive number $\lambda_0$ and two analytic functions $V_1$ and $V_2$ on $\boC_{\lambda_0}$, such that $v_1$, respectively $v_2$, are identically equal to $V_1$, respectively $V_2$ on $\R^N$. Therefore, $\Re(v) = 1 + v_1$ and $\Im(v) = v_2$ extend to analytic functions $1 + V_1$ and $V_2$ on $\boC_{\lambda_0}$, which completes the proof of Theorem \ref{analyticity0}.
\end{proof}

\subsection{Solutions without vortices}

In this subsection, we consider only solutions $v$ to \eqref{TWc} on $\R^N$ which do not vanish. In particular, we assume throughout that 
\begin{equation}
\label{novortex}
|v| \geq \frac{1}{2},
\end{equation}
so that $v$ may be written as in \eqref{lereleve}, $v = \varrho \exp i \varphi$. Using \eqref{relevenergy0}, the energy writes in the variables $\varrho$ and $\varphi$,
\begin{equation}
\label{PolE}
E(v) = E(\varrho, \varphi) \equiv \frac{1}{2} \int_{\R^N} \bigg( |\nabla \varrho|^2 + \varrho^2 |\nabla \varphi|^2 + \frac{( 1 - \rho^2)^2}{2} \bigg),
\end{equation}
whereas for the momentum, we have
$\langle i \partial_1 v , v \rangle = - \varrho^2 \partial_1 \varphi.$
Therefore, it follows from Proposition \ref{cestlemoment} that
\begin{equation}
\label{Polp}
p(v) = \tilde{p}(v) = \frac{1}{2} \int_{\R^N} \Big( - \varrho^2 \partial_1 \varphi + \partial_1 \big( (1 - \chi) \varphi \big) \Big) =
\frac{1}{2} \int_{\R^N} (1 - \varrho^2) \partial_1 \varphi.
\end{equation}
The system for $\varrho$ and $\varphi$ writes
\begin{equation}
\label{PolTWc}
\left\{ \begin{array}{ll} \frac{c}{2} \partial_1 \varrho^2 + \div \Big( \varrho^2 \nabla \varphi \Big) = 0,\\
c \varrho \partial_1 \varphi - \Delta \varrho - \varrho \Big( 1 - \varrho^2 \Big) + \varrho |\nabla \varphi|^2 = 0. \end{array} \right.
\end{equation}
Notice that the quantity $\eta = 1 - \varrho^2$ satisfies the equation 
\begin{equation}
\label{equeta}
\Delta^2 \eta - 2 \Delta \eta + c^2 \partial_1^2 \eta= - 2 \Delta (|\nabla v|^2 + \eta^2 - c \eta \partial_1 \varphi) - 2 c \partial_1 \div(\eta \nabla \varphi),
\end{equation}
where the l.h.s is linear with respect to $\eta$, whereas the r.h.s is quadratic.

A first elementary remark is
 
\begin{lemma}
\label{elementary}
Let $v$ be a finite energy solution to \eqref{TWc} on $\R^N$ satisfying \eqref{novortex}. Then, we have the identity
\begin{equation}
\label{zeroieme}
c p(v) = \int_{\R^N} \varrho^2 |\nabla \varphi|^2.
\end{equation}
\end{lemma} 

\begin{proof}
The identity is obtained multiplying the first equation in \eqref{PolTWc} by $\varphi$ and integrating by parts using the decay properties of Remark \ref{decay-pol}.
\end{proof}
 
Lemma \ref{elementary} has the following remarkable consequence.

\begin{lemma}
\label{remarkable}
Let $v$ be a finite energy solution to \eqref{TWc} on $\R^N$ satisfying \eqref{novortex}. Then,
$$E(v)\leq 7 c(v)^2 \int_{\R^N} \eta^2.$$
\end{lemma}

\begin{proof}
We have, in view of assumption \eqref{novortex}, Lemma \ref{elementary} and Cauchy-Schwarz's inequality,
\begin{align*}
\int_{\R^N} \varrho^2 |\nabla \varphi |^2 = \frac{c}{2} \int_{\R^N} ( 1 - \varrho^2) \partial_1 \varphi & \leq c \bigg( \int_{\R^N} \eta^2 \bigg)^\frac{1}{2} \bigg( \int_{\R^N} |\nabla \varphi|^2 \bigg)^\frac{1}{2}\\
& \leq 2 c \bigg( \int_{\R^N} \eta^2 \bigg)^\frac{1}{2} \bigg( \int_{\R^N} \varrho^2 |\nabla \varphi|^2 \bigg)^\frac{1}{2}.
\end{align*}
Hence, we deduce
\begin{equation}
\label{premiere}
\int_{\R^N} \varrho^2 |\nabla \varphi|^2 \leq 4 c^2 \int_{\R^N} \eta^2.
\end{equation}
It remains to bound the integral of $|\nabla \varrho|^2$. For that purpose we multiply the second equation in \eqref{PolTWc} by $\varrho^2 - 1$ and integrate by parts on $\R^N$ using the decay properties of Remark \ref{decay-pol}. This yields
\begin{equation}
\label{alpha1}
\int_{\R^N} \bigg( 2 \varrho |\nabla \varrho|^2 + \varrho (1 - \varrho^2)^2 \bigg) = c \int_{\R^N} \varrho (1 - \varrho^2) \partial_1 \varphi + \int_{\R^N} \varrho (1 - \varrho^2) |\nabla \varphi|^2,
\end{equation}
so that
$$\int_{\R^N} \bigg( |\nabla \varrho|^2 + \frac{1}{2} (1 - \varrho^2)^2 \bigg) \leq c \bigg( \int_{\R^N} \varrho^2 |\partial_1 \varphi|^2 \bigg)^\frac{1}{2} \bigg( \int_{\R^N} (1 - \varrho^2)^2 \bigg)^\frac{1}{2} + 2 \int_{\R^N} \varrho^2 |\nabla \varphi|^2,$$
by \eqref{novortex} and Cauchy-Schwarz's inequality. Invoking \eqref{premiere} in order to bound the r.h.s of this identity in terms of the integral of $\eta^2$, we deduce the desired inequality.
\end{proof}

\subsection{Subsonic vortexless solutions}

We next specify a little further the analysis assuming that the solution $v$ verifies the additional condition
\begin{equation}
\label{subsoniccase}
0 < c(v) < \sqrt{2}.
\end{equation}
We set, for such a solution
$$\varepsilon(v) = \sqrt{2 - c(v)^2}.$$
We first have the bound 
 
\begin{prop}
\label{minoration}
Let $v$ be a non-trivial finite energy solution to \eqref{TWc} on $\R^N$ satisfying \eqref{subsoniccase}. Then,
$$\Big\| 1 - |v| \Big\|_{L^\infty(\R^N)} \geq \frac{\varepsilon(v)^2}{10}.$$
\end{prop}

\begin{proof}
Set $\delta = \| 1 - |v| \|_{L^\infty(\R^N)}$. If $\delta \geq \frac{1}{2}$, then the proof is straightforward. Otherwise, assumption \eqref{novortex} is satisfied and going back to identity \eqref{zeroieme}, we observe that by Lemma \ref{colisee},
$$c p(v) = \frac{c}{2} \bigg| \int_{\R^N} (1 - \varrho^2) \partial_1 \varphi \bigg| \leq \frac{c}{\sqrt{2} (1 - \delta)} \int_{\R^N} e(v),$$
so that, using identity \eqref{Polp} and Lemma \ref{elementary},
\begin{equation}
\label{star1}
\int_{\R^N} \varrho^2 |\nabla \varphi|^2 \leq \frac{c}{\sqrt{2} (1 - \delta)} \int_{\R^N} e(v).
\end{equation}
Similarly, we have for the r.h.s of identity \eqref{alpha1}
\begin{equation}
\label{alpha3}
c \bigg| \int_{\R^N} \varrho (1 - \varrho^2) \partial_1 \varphi \bigg| \leq \sqrt{2} c \int_{\R^N} e(v),
\end{equation}
and using \eqref{novortex},
\begin{equation}
\label{alpha4}
\bigg| \int_{\R^N} \varrho (1 - \varrho^2) |\nabla \varphi|^2
\bigg| \leq 6 \delta \int_{\R^N} e(v).
\end{equation}
Combining \eqref{alpha1}, \eqref{alpha3} and \eqref{alpha4}, and using the fact that $\rho \geq 1 - \delta > 0$, we are led to
\begin{equation}
\label{star2}
\int_{\R^N} \bigg( \frac{|\nabla \varrho|^2}{2} + \frac{(1 - \varrho^2)^2}{4} \bigg) \leq \frac{\sqrt{2} c + 6 \delta}{4 (1 - \delta)} \int_{\R^N} e(v).
\end{equation}
Finally, from \eqref{star1} and \eqref{star2}, we derive 
$$\lambda \int_{\R^N} e(v)\leq 0,$$
where we have set $\lambda = 1 - \frac{c}{\sqrt{2} (1 - \delta)} - \frac{3 \delta}{2 (1 - \delta)}$. Since $v$ is non-trivial, its energy is not equal to $0$ so that $\lambda \leq 0$ and $$\delta \geq \frac{2}{5} \bigg( 1 - \frac{c}{\sqrt{2}} \bigg) = \frac{2}{5} \bigg( 1 - \sqrt{1 - \frac{\varepsilon^2}{2}} \bigg) \geq \frac{\varepsilon^2}{10},$$
which is the desired inequality. 
\end{proof}

Combining Lemma \ref{encorepoho} with Lemma \ref{elementary}, we are led to

\begin{lemma}
\label{tropbien}
Let $v$ be a finite energy solution to \eqref{TWc} on $\R^N$ satisfying \eqref{novortex} and \eqref{subsoniccase}. Then,
\begin{equation}
\label{tropbien1}
\Sigma(v) + \frac{1}{N} \int_{\R^N} |\nabla \varrho|^2 = \frac{\varepsilon(v)^2}{\sqrt{2} + c(v)} p(v).
\end{equation}
Moreover, if $N = 2$, then we have
\begin{equation}
\label{tropbien2}
\int_{\R^2} |\nabla \varrho|^2 \Big( 1 + \frac{1}{\varrho^2} \Big) = \int_{\R^2} \eta |\nabla \varphi|^2.
\end{equation}
\end{lemma}

\begin{proof}
For the proof of equality \eqref{tropbien1}, we first add equality \eqref{theverypoho} and the identity provided in Lemma \ref{elementary}, to obtain
$$\frac{N - 2}{2} \int_{\R^N} |\nabla \varrho|^2 + \frac{N}{2} \int_{\R^N} \varrho^2 |\nabla \varphi|^2 + \frac{N}{4} \int_{\R^N} \eta^2 = c N p(v),$$ 
using the identity $|\nabla v|^2 = |\nabla \varrho|^2 + \varrho^2 |\nabla \varphi|^2$. This yields
$$E(v) - c p(v) = \frac{1}{N} \int_{\R^N} |\nabla \varrho|^2,$$
and equality \eqref{tropbien1} follows from the definitions of $\Sigma(v)$ and $\varepsilon(v)$.

For equality \eqref{tropbien2}, we multiply the second equation in \eqref{PolTWc} by $\frac{\eta}{\varrho}$, and integrate on $\R^N$. This yields
$$\int_{\R^N} \Big( c \eta \partial_1 \varphi - \frac{\Delta \varrho}{\varrho} \eta - \eta^2 + \eta |\nabla \varphi|^2 \Big) = 0.$$
Since by definition $\eta = 1 - \varrho^2$, we obtain, integrating by parts
$$\int_{\R^N} \frac{\Delta \varrho}{\varrho} \eta = - \int_{\R^N} \nabla \varrho \nabla \Big( \frac{1 - \varrho^2}{\varrho} \Big) = \int_{\R^N} |\nabla \varrho|^2 \Big( 1 + \frac{1}{\varrho^2} \Big).$$
On the other hand, for $N = 2$, identities \eqref{theverypoho} and \eqref{Polp} yield 
$$\int_{\R^2} \eta^2 = 2 c p(v) = c \int_{\R^2} \eta \partial_1 \varphi.$$
Conclusion \eqref{tropbien2} follows combining the last three relations. 
\end{proof}

\subsection{Use of Fourier transform}
\label{grandfou}

We consider for $\xi \in \R^N$, and a function $f$ defined on $\R^N$, its Fourier transform $\widehat{f}(\xi)$ defined by the integral
$$\widehat{f}(\xi) = \int_{\R^N} f(x) e^{-i x.\xi} dx.$$

\begin{lemma}
\label{equenfourier} 
Let $v$ be a finite energy solution to \eqref{TWc} on $\R^N$ satisfying \eqref{novortex}. Then, for any $\xi \in \R^N$, we have
\begin{equation}
\label{equenfourier1}
\bigg( |\xi|^2 + 2 - c^2 \frac{\xi_1^2}{|\xi|^2} \bigg) \widehat{\eta}(\xi) = 2 \widehat{R_0}(\xi) - 2 c \sum_{j = 2}^N \frac{\xi_j^2}{|\xi|^2} \widehat{R_1}(\xi) + 2 c \sum_{j = 2}^N \frac{\xi_1 \xi_j}{|\xi|^2} \widehat{R_j}(\xi),
\end{equation}
where we have set $R_0 = |\nabla v|^2 + \eta^2 \ {\rm and} \ R_j = \eta \partial_j \varphi$.
\end{lemma}

\begin{proof}
It suffices to consider the Fourier transform of \eqref{equeta}.
\end{proof}

Using \eqref{equenfourier1}, we deduce in the two-dimensional case,
 
\begin{lemma}
\label{bornevarepsilon}
Let $N = 2$ and let $v$ be a finite energy solution to \eqref{TWc} on $\R^2$ satisfying \eqref{novortex} and \eqref{subsoniccase}. Then, there exists some universal constant $K > 0$ such that
$$\varepsilon(v) \leq K E(v).$$
\end{lemma}
 
\begin{proof}
We first notice that for any integrable function $f$ on $\R^N$ and any $\xi \in \R^N$, we have in view of the definition of $\widehat{f}(\xi)$,
$|\widehat{f}(\xi)| \leq \| f \|_{L^1(\R^N)},$
so that
\begin{equation}
\| \widehat{R_i} \|_{L^\infty(\R^N)} \leq \| R_i \|_{L^1(\R^N)} \leq K E(v).
\end{equation}
It follows from integrating \eqref{equenfourier1} that we have for any $1 \leq q \leq + \infty$,
$$\int_{\R^N} |\widehat{\eta}(\xi)|^q d\xi \leq K (1 + c^q) \bigg( \int_{\R^N} |L_\varepsilon(\xi)|^q d\xi \bigg) E(v)^q,$$
where we have set, for any $\xi \in \R^N$,
\begin{equation}
\label{L-epsilon}
L_\varepsilon(\xi) = \frac{1}{|\xi|^2 + 2 - c^2 \frac{\xi_1^2}{|\xi|^2}} = \frac{|\xi |^2}{|\xi|^4 + 2| \xi |^2 - c^2 \xi_1^2}.
\end{equation}
In the case $q = 2$, this leads in view of Plancherel's formula to
\begin{equation}
\label{superbe}
\int_{\R^N} \eta(x)^2 dx \leq K (1 + c^2) \bigg( \int_{\R^N} |L_\varepsilon(\xi)|^2 d\xi \bigg) E(v)^2.
\end{equation}
We claim that
\begin{equation}
\label{claim0}
\int_{\R^2} L_\varepsilon(\xi)^2 d\xi = \frac{\pi}{\sqrt{2} \varepsilon(c)}.
\end{equation}
We postpone the proof of the claim, and complete the proof of Lemma \ref{bornevarepsilon}. Combining \eqref{subsoniccase}, \eqref{superbe} and Lemma \ref{remarkable}, we have
$$E(v) \leq 7 c^2 \int_{\R^N} \eta^2 \leq \frac{K}{\varepsilon(c)} E(v)^2,$$
where we used claim \eqref{claim0} for the second inequality. This concludes the proof of Lemma \ref{bornevarepsilon}.
\end{proof}

\begin{proof}[Proof of Claim \eqref{claim0}] 
Introducing polar coordinates, we have 
$$\int_{\R^2} L_{\varepsilon}(\xi)^2 d\xi = \int_0^{+ \infty} \int_0^{2 \pi} \frac{r dr d\theta}{(r^2 + 2 - c^2 \cos^2(\theta))^2} = \frac{1}{2} \int_0^{2 \pi} \frac{d\theta}{2 - c^2 \cos^2(\theta)}.$$
Using the change of variables $t = \tan(\theta)$ and $u = \sqrt{\frac{2}{2 - c^2}} t$, we obtain
\begin{equation*}
\int_{\R^2} L_{\varepsilon}(\xi)^2 d\xi = 2 \int_0^{+ \infty} \frac{dt}{2 - c^2 + 2 t^2} = \sqrt{\frac{2}{2 - c^2}} \int_0^{+ \infty} \frac{du}{1 + u^2} = \frac{\pi}{\sqrt{2} \varepsilon(c)}.
\end{equation*}
\end{proof}
 
In dimension three, the previous analysis leads to a result of a very different nature.
 
\begin{lemma}
\label{nopetitsolution}
Let $v$ be a finite energy solution to \eqref{TWc} on $\R^3$ satisfying \eqref{novortex} and \eqref{subsoniccase}. Then, there exists some universal constant $K > 0$ such that
$$E(v) \geq \boE_0(c) \equiv \frac{K}{c (1 + c^2) \arcsin \Big( \frac{c}{\sqrt{2}} \Big)}.$$
\end{lemma} 

\begin{proof}
The argument is parallel to the argument of Lemma \ref{bornevarepsilon} up to \eqref{superbe}. However, in contrast with \eqref{claim0} we have in dimension three,
\begin{equation}
\label{claim1}
\int_{\R^3} L_\varepsilon(\xi)^2 d\xi = \frac{\pi^2}{c} \arcsin \Big( \frac{c}{\sqrt{2}} \Big).
\end{equation}
We postpone the proof of \eqref{claim1}, and complete the proof of Lemma \ref{nopetitsolution}. By \eqref{superbe} and Lemma \ref{remarkable}, there exists some constant $K > 0$ such that
$$E(v) \leq 7 c^2 \int_{\R^3} \eta^2 \leq K c^2 (1 + c^2) \bigg( \int_{\R^3} L_\varepsilon(\xi)^2 d\xi \bigg) E(v)^2.$$
Lemma \ref{nopetitsolution} follows using \eqref{claim1}.
\end{proof}

\begin{proof}[Proof of Claim \eqref{claim1}]
Introducing spherical coordinates, and using the change of variables $u = \cos (\theta)$, we have 
$$\int_{\R^3} L_\varepsilon(\xi)^2 d\xi = 2 \pi \int_0^{+ \infty} \bigg( \int_0^\pi \frac{r^2 \sin(\theta) d\theta}{(r^2 + 2 - c^2 \cos^2(\theta))^2} \bigg) dr = 4 \pi \int_0^{+ \infty} \bigg( \int_0^1 \frac{r^2 du}{(r^2 + 2 - c^2 u^2)^2} \bigg) dr.$$
An integration by parts with respect to the variable $r$ gives
$$\int_{\R^3} L_\varepsilon(\xi)^2 d\xi = 2 \pi \int_0^{+ \infty} \bigg( \int_0^1 \frac{du}{r^2 + 2 - c^2 u^2} \bigg) dr .$$
Using the change of variables $v = \frac{r}{\sqrt{2 - c^2 u^2}}$ and $w = \frac{c u}{\sqrt{2}}$, we obtain
\begin{align*}
\int_{\R^3} L_\varepsilon(\xi)^2 d\xi = & 2 \pi \int_0^{+ \infty} \bigg( \int_0^1 \frac{du}{\sqrt{2 - c^2 u^2}} \bigg) \frac{dv}{1 + v^2} = \pi^2 \int_0^1 \frac{du}{\sqrt{2 - c^2 u^2}} = \frac{\pi^2}{c} \int_0^\frac{c}{\sqrt{2}} \frac{dw}{\sqrt{1 - w^2}}\\
= &\frac{\pi^2}{c} \arcsin \Big( \frac{c}{\sqrt{2}} \Big).
\end{align*}
\end{proof}

The previous analysis extends to any finite energy subsonic and sonic solutions, providing a proof of iii) in Theorem \ref{dim3}.
 
\begin{lemma}
\label{soniccase}
Let $v$ be a non-trivial finite energy solution to \eqref{TWc} on $\R^3$. Then
$$E(v) \geq \boE_0,$$
where $\boE_0$ is some positive universal constant. 
\end{lemma}

\begin{proof}
In view of the results of \cite{Graveja2}, we know that there are no supersonic travelling waves on the whole space that is $0 < c(v) \leq \sqrt{2}$. Moreover, in dimensions $N > 2$, it follows from Lemma \ref{tarquini20} that
$$\| 1 - |v| \|_{L^\infty(\R^N)} \leq K E(v)^\frac{1}{N + 2},$$
so that, choosing possibly a smaller constant $\boE_0$, we may assume that $v$ satisfies \eqref{novortex}, and $v$ may be written as in \eqref{lereleve}, $v = \varrho \exp i \varphi$. If $v$ is subsonic, i.e. $c(v) < \sqrt{2}$, then Lemma \ref{nopetitsolution} yields the conclusion, since the function $c \mapsto c (1 + c^2) \arcsin (\frac{c}{\sqrt{2}})$ is bounded on $(0, \sqrt{2})$. In the sonic case, one observes similarly that
$$\int_{\R^3} L_0(\xi)^2 d\xi = \frac{\pi^3}{2 \sqrt{2}} < + \infty,$$
and the same proof as in Lemma \ref{nopetitsolution} applies to yield the conclusion.
\end{proof}

In the same spirit, but with more involved methods, we may prove 
 
\begin{lemma} 
\label{piupiccolo}
Let $\frac{5}{3} < q < + \infty$, and let $v$ be a finite energy solution to \eqref{TWc} on $\R^3$ satisfying \eqref{novortex}. Then, there exists a constant $K(q)$ only depending on $q$ such that
$$\| \eta \|_{L^q(\R^3)} \leq K(q) E(v)^{\frac{1}{q} + \frac{2}{5}}.$$
\end{lemma}
 
\begin{proof}
We first recall that there are no non-trivial finite energy supersonic travelling waves on $\R^3$. Hence, we may assume that $0 \leq c \leq \sqrt{2}$. In view of equation \eqref{equenfourier1}, we have
\begin{equation}
\label{eqmultfou}
\widehat{\eta}(\xi) = L_\varepsilon(\xi) \bigg( 2 \widehat{R_0}(\xi) - 2 c \sum_{j = 2}^3 \frac{\xi_j^2}{|\xi|^2} \widehat{R_1}(\xi) + 2 c \sum_{j = 2}^3 \frac{\xi_1 \xi_j}{|\xi|^2} \widehat{R_j}(\xi) \bigg),
\end{equation}
so that the proof reduces to estimate the r.h.s of \eqref{eqmultfou} by using multipliers theory developed in \cite{Lizorki1}. Indeed, using Lemma \ref{tarquini10}, and the fact that $0 \leq c \leq \sqrt{2}$, we first notice that there exists some universal constant $K$ such that
$$|R_j| \leq K e(v).$$
Invoking Lemma \ref{tarquini10} and the fact that $0 \leq c \leq \sqrt{2}$ once more, we have for any $1 < q < + \infty$,
$$\| R_j \|_{L^q(\R^3)} \leq K \bigg( \int_{\R^3} e(v)^q \bigg)^\frac{1}{q} \leq K E(v)^\frac{1}{q}.$$
On the other hand, it follows from standard Riesz-operator theory (see \cite{Stein1}) that the functions $\xi \mapsto \frac{\xi_j \xi_k}{|\xi|^2}$ are $L^q$-multipliers for any $1 < q < + \infty$ and any $1 \leq j, k \leq 3$. Hence, there exists some constant $K(q)$, possibly depending on $q$, so that the function $F$ defined by
\begin{equation}
\label{chabal}
\widehat{F}(\xi) = 2 \widehat{R_0}(\xi) - 2 c \sum_{j = 2}^3 \frac{\xi_j^2}{|\xi|^2} \widehat{R_1}(\xi) + 2 c \sum_{j = 2}^3 \frac{\xi_1 \xi_j}{|\xi|^2} \widehat{R_j}(\xi),
\end{equation}
belongs to $L^q(\R^3)$ for any $1 < q < + \infty$, and satisfies
\begin{equation}
\label{dominici}
\| F \|_{L^q(\R^3)} \leq K(q) \sum_{j = 0}^3 \| R_j \|_{L^q(\R^3)} \leq K(q) E(v)^\frac{1}{q}.
\end{equation}
Finally, in view of Theorem \ref{thelizo}, $L_\varepsilon$ is a multiplier from $L^q(\R^3)$ to $L^\frac{5 q}{5 - 2 q}(\R^3)$ for any $1 < q < \frac{5}{2}$. More precisely, denoting $\boL_\varepsilon$, the multiplier operator given by
\begin{equation}
\label{jauzion}
\widehat{\boL_\varepsilon(f)}(\xi) = L_\varepsilon(\xi) \widehat{f}(\xi),
\end{equation}
there exists some constant $K(q)$ possibly depending on $q$ but not on $c$, such that, for any $1 < q < \frac{5}{2}$,
\begin{equation}
\label{heymans}
\| \boL_\varepsilon(f) \|_{L^\frac{5 q}{5 - 2 q}(\R^3)} \leq K(q) \| f \|_{L^q(\R^3)}, \forall f \in L^q(\R^3).
\end{equation}
We postpone the proof of claim \eqref{heymans}, and complete the proof of Lemma \ref{piupiccolo}. Indeed, it follows from \eqref{eqmultfou}, \eqref{chabal} and \eqref{jauzion} that
$$\eta = \boL_\varepsilon(F).$$
Therefore, by \eqref{dominici} and \eqref{heymans}, we have
$$\| \eta \|_{L^\frac{5 q'}{5 - 2 q'}(\R^3)} \leq K(q') E(v)^\frac{1}{q'},$$
for any $1 < q' < \frac{5}{2}$. Letting $q = \frac{5 q'}{5 - 2 q'}$ that is $\frac{1}{q'} = \frac{1}{q} + \frac{2}{5}$, this ends the proof of Lemma \ref{piupiccolo}.
\end{proof}

\begin{proof}[Proof of Claim \eqref{heymans}]
Claim \eqref{heymans} is a consequence of Theorem \ref{thelizo} applied to $L_\varepsilon$. Indeed, we may check that $L_\varepsilon$ satisfies the assumptions of Theorem \ref{thelizo}, i.e. that the quantity
$$M(L_\varepsilon) \equiv \sup \Big\{ \underset{j = 1}{\overset{3}{\Pi}} |\xi_j|^{\alpha + k_j} \Big| \partial_1^{k_1} \partial_2^{k_2} \partial_3^{k_3} L_\varepsilon(\xi) \Big|, \xi \in \R^3, (k_1, k_2, k_3) \in \{ 0, 1 \}^3, k_1 + k_2 + k_3 \leq 3 \Big\},$$
is finite for some suitable choice of $\alpha$. This follows from the next computation of some derivatives of $L_\varepsilon$,
\begin{equation}
\label{collins}
\partial_j L_\varepsilon(\xi) = \frac{2 \xi_j}{(|\xi|^4 + 2 |\xi|^2 - c^2 \xi_1^2)^2} \Big( - |\xi|^4 - c^2 \xi_1^2 + c^2 \delta_{1,j} |\xi|^2 \Big),
\end{equation}
for any $1 \leq j \leq 3$,
\begin{equation}
\label{mccaw}
\begin{split}
\partial^2_{jk} L_\varepsilon(\xi) = & \frac{4 \xi_j \xi_k}{(|\xi|^4 + 2 |\xi|^2 - c^2 \xi_1^2)^3} \Big( 2 |\xi|^6 + c^2 \big( 6 \xi_1^2 |\xi|^2 + 4 \xi_1^2 - (\delta_{1,j} + \delta_{1,k}) (3 |\xi|^4 + 2 |\xi|^2 + c^2 \xi_1^2) \big) \Big),
\end{split}
\end{equation}
for any $1 \leq j \neq k \leq 3$, and
\begin{equation}
\label{sooialo}
\begin{split}
\partial^3_{123} L_\varepsilon(\xi) = \frac{16 \xi_1 \xi_2 \xi_3}{(|\xi|^4 + 2 |\xi|^2 - c^2 \xi_1^2)^4} \Big( - & 3 |\xi|^8 + c^2 \big( 6 |\xi|^6 -18 \xi_1^2 |\xi|^4 + 8 |\xi|^4 + 6 (c^2 - 4) \xi_1^2 |\xi|^2 - 3 c^2 \xi_1^4\\
+ & 4 |\xi|^2 + 4 c^2 \xi_1^2 - 12 \xi_1^2 \big) \Big).
\end{split}
\end{equation}
Considering some multi-index $(k_1, k_2, k_3) \in \{ 0, 1 \}^3$ such that $k_1 + k_2 + k_3 \leq 3$, it follows from \eqref{L-epsilon}, \eqref{collins}, \eqref{mccaw} and \eqref{sooialo}, and the fact that $0 \leq c \leq \sqrt{2}$, that there exists some universal constant $K$ such that, for any $|\xi| \geq 1$,
\begin{equation}
\label{carter}
\underset{j = 1}{\overset{3}{\Pi}} |\xi_j|^{\alpha + k_j} \Big| \partial_1^{k_1} \partial_2^{k_2} \partial_3^{k_3} L_\varepsilon(\xi) \Big| \leq \frac{K}{|\xi|^{2 - 3 \alpha }} \leq K,
\end{equation}
provided $\alpha \leq \frac{2}{3}$. On the other hand, if $|\xi| \leq 1$, we deduce from \eqref{L-epsilon}, \eqref{collins}, \eqref{mccaw} and \eqref{sooialo}, that there exists some universal constant $K$ such that
$$\big| \partial_1^{k_1} \partial_2^{k_2} \partial_3^{k_3} L_\varepsilon(\xi) \big| \leq K \frac{|\xi_1|^{k_1} |\xi_2|^{k_2} |\xi_3|^{k_3}}{(|\xi|^4 + 2 |\xi|^2 - c^2 \xi_1^2)^{1 + k_1 + k_2 + k_3}} \Big( \big( 1 - k_1 \big) |\xi|^2 + k_1 \big( |\xi|^4 + 2 |\xi|^2 - c^2 \xi_1^2 \big) \Big).$$
Therefore, denoting $\xi = \rho \sigma$ where $\rho \geq 0$ and $\sigma = (\sigma_1, \sigma_\perp) \in \S^{N-1}$, we are led to
\begin{equation}
\label{jack}
\begin{split}
\prod_{j = 1}^3 |\xi_j|^{\alpha + k_j} \Big| \partial_1^{k_1} \partial_2^{k_2} \partial_3^{k_3} L_\varepsilon(\xi) \Big| & \leq K \frac{\rho^{2 ( 1 + k_2 + k_3) + 3 \alpha} |\sigma_\perp|^{2 (\alpha + k_2 + k_3)}}{\rho^{2 (1 + k_2 + k_3)} (\rho^2 + 2 |\sigma_\perp|^2)^{1 + k_2 + k_3}} \\ & \leq K \max \{ \rho, |\sigma_\perp| \}^{5 \alpha - 2 } \leq K,
\end{split}
\end{equation}
provided that $\alpha \geq \frac{2}{5}$. Using \eqref{carter} and \eqref{jack}, and choosing $\alpha = \frac{2}{5}$, the quantity $M(L_\varepsilon)$ is bounded by some constant $K$ not depending on $\varepsilon$, so that, by Theorem \ref{thelizo}, $L_\varepsilon$ is a multiplier from $L^q(\R^3)$ to $L^\frac{5 q}{5 - 2 q}(\R^3)$ for any $1 < q < \frac{5}{2}$. Hence, inequality \eqref{heymans} follows from \eqref{ibanez} and \eqref{mignoni}.
\end{proof}

We will use the following consequence.

\begin{cor}
\label{tresutile} 
There exists some constants $K > 0$ and $\alpha > 0$ such that
$$\| 1 - |v| \|_{L^\infty(\R^3)} \geq \frac{K}{E(v)^\alpha}.$$
\end{cor}

\begin{proof}
If $\| 1 - |v| \|_{L^\infty(\R^3)} \geq \frac{1}{2}$, then we are done in view of Lemma \ref{soniccase}, for $K = \frac{\boE_0^\alpha}{2}$ and any $\alpha > 0$. Therefore, we may assume $\| 1 - |v| \|_{L^\infty(\R^3)} \leq \frac{1}{2}$. In that case, using Lemmas \ref{remarkable} and \ref{piupiccolo}, and the fact that $0 \leq c \leq \sqrt{2}$, we write for any $\frac{5}{3} < q < 2$,
$$E(v) \leq 7 c^2 \| \eta \|_{L^2(\R^3)}^2 \leq 14 \| \eta \|_{L^q(\R^3)}^q \| \eta \|_{L^\infty(\R^3)}^{2 - q} \leq K(q) \| \eta \|_{L^\infty(\R^3)}^{2 - q} E(v)^{1 + \frac{2 q}{5}}.$$
Hence,
$$\| 1 - |v| \|_{L^\infty(\R^3)} \geq \frac{\| \eta \|_{L^\infty(\R^3)}}{1 + \| \varrho \|_{L^\infty(\R^3)}} \geq \frac{2}{5}\| \eta \|_{L^\infty(\R^3)} \geq \frac{K(q)}{E(v)^\frac{2 q}{5(2 - q)}}.$$
We conclude, choosing for instance $q = \frac{7}{4}$ and $\alpha = \frac{14}{5}$.
\end{proof}

Combining the previous result with Lemma \ref{bonlemme}, we obtain the following bound for $\varepsilon(v)$.
 
\begin{lemma}
\label{bornepsfin}
Let $v$ be a finite energy solution to \eqref{TWc} on $\R^3$, such that $\Sigma(v) > 0$. Then, there exists some constant $K(c)$ depending only on $c$, and some universal constant $\alpha > 0$, such that
$$\varepsilon(v) p(v) \geq \frac{K(c)}{E(v)^{8 \alpha + 1}}.$$
\end{lemma}

\begin{proof}
In view of Lemma \ref{bonlemme}, we have
$$\lambda \varepsilon(v) p(v) + \frac{E(v)}{\lambda} \geq K(c) \| \eta \|_{L^\infty(\R^3)}^4 \geq \frac{K(c)}{E(v)^{4 \alpha}}, \forall \lambda > 0,$$
where we have made use of Corollary \ref{tresutile} for the last inequality. The choice $\lambda = \frac{2 E(v)^{4 \alpha + 1}}{K(c)}$ yields the desired result.
\end{proof}

\section{ Properties of the function $E_{\min}(\p)$}
\label{PropEmin}

The main purpose of this section is to provide the proofs to Theorem \ref{proemin}, Lemma \ref{speed} and Lemma \ref{gueri}, as well as the proof of Theorem \ref{symetrie}. 

\subsection{Proof of Theorem \ref{proemin}}

We begin this subsection with a number of elementary observations. 

\begin{lemma} 
\label{inject}
For $N = 2$ and $N = 3$, we have the inclusion $W(\R^N) \subset \boE(\R^N)$. Moreover, the functions $E$ and $p$ are continuous on $W(\R^N)$.
\end{lemma}

\begin{proof}
Concerning the momentum $p$, we have already seen that, in view of \eqref{identityp} and H\"older's inequality, it is well-defined and continuous on $W(\R^N)$. For the energy $E$, we start with the identity
\begin{equation}
\label{lequalite}
(1 - |1 + w|^2)^2 = 4 \Re(w)^2 + 4 \Re(w) |w|^2 + |w|^4,
\end{equation}
for any $w \in V(\R^N)$, so that
$$(1 - |1 + w|^2)^2 \leq 8 \Re(w)^2 + 4 \Re(w)^4 + 4 \Im(w)^4,$$
and the l.h.s of this identity belongs to $L^1(\R^N)$, whenever $w$ belongs to $V(\R^N)$. Hence, $W(\R^N)$ is included in $\boE(\R^N)$, and the $L^1$-norm of the l.h.s of \eqref{lequalite} being continuous on $V(\R^N)$, $E$ is also continuous in $W(\R^N)$. 
\end{proof}

\begin{lemma}
\label{compactsupport} 
Assume $N = 2$ or $N = 3$ and let $v = 1 + w$ be in $W(\R^N)$. There exists a sequence of maps $(w_n)_{n \in \N}$ in $C_c^\infty(\R^N)$ such that $w_n \to w$ in $V(\R^N)$, as $n \to + \infty$,
$$p(v_n) = p(v), \ {\rm and} \ E(v_n) \to E(v), \ {\rm as} \ n \to + \infty.$$
In particular, given any $\p \geq 0$, there exists a sequence of maps $(w_n)_{n \in N}$ in $C_c^\infty(\R^N)$ such that 
$$p(1 + w_n) = \p, \ {\rm and} \ E(1 + w_n) \to E_{\min}(\p), \ {\rm as} \ n \to + \infty,$$
so that
\begin{equation}
\label{eminsmooth}
E_{\min}(\p) = \inf \{ E(1 + v), v \in C_c^\infty(\R^N), p(1 + v) = \p \}.
\end{equation}
\end{lemma}

\begin{proof}
In view of continuity properties stated in Lemma \ref{inject} and the density of $C_c^\infty(\R^N)$ into $V(\R^N)$, given any $w \in V(\R^N)$, there exists a sequence of maps $(\tilde{w}_n)_{n \in N}$ such that
$$p(\tilde{v}_n) \to p(v), \ {\rm and} \ E(\tilde{v_n}) \to E(v), \ {\rm as} \ n \to + \infty,$$
where we have set $\tilde{v}_n = 1 + \tilde{w}_n$, and $v = 1 + w$. In order to prove the first assertion, we distinguish two cases. First, if $p(v) \neq 0$, then we may assume without loss of generality that $p(v) > 0$, so that by continuity, we have $p(\tilde{v}_n) > 0$, for $n$ sufficiently large. In this case, we set
$$w_n = \sqrt{\frac{p(v)}{p(\tilde{v}_n)}} \tilde{w}_n, \ {\rm and} \ v_n = 1 + w_n,$$
so that $p(v_n)=p(v)$, whereas $w_n \to w$ in $V(\R^N)$, as $n \to + \infty$, which yields the conclusion.

The case $p(v) = 0$ (which is actually not the most relevant one for our discussion) is treated by an approximation argument. Indeed, in this case, we may assume that $p(\tilde{v}_n) \neq 0$ for $n$ sufficiently large (otherwise, up to a subsequence, the conclusion holds for $w_n = \tilde{w}_n$). For given $\delta > 0$, we may construct (see Lemma \ref{quasilineaire} below) a map $f_\delta \in C_c^\infty(\R^N)$ such that $p(1 + f_\delta) = \delta$, $E(1 + f_\delta) \leq 2 |\delta|$ and $\| f_\delta \|_{V(\R^N)} \leq K \sqrt{|\delta|}$ for some universal constant $K$. Denoting $\check{f_\delta}(x_1, x_\perp) = f_\delta(-x_1, x_\perp)$, this construction is also possible for any $\delta < 0$. We then consider the map $w_n = \tilde{w}_n + f_{\delta_n}(\cdot - a_n)$, where $\delta_n = - p(1 + \tilde{w}_n) \to 0$, as $n \to + \infty$, and the point $a_n \in \R^N$ is chosen sufficiently large so that the supports of $\tilde{w}_n$ and $f_{\delta_n}(\cdot - a_n)$ do not intersect. Denoting $v_n = 1 + w_n$, we have $p(v_n) = p(\tilde{w}_n) + p(f_{\delta_n}) = 0$, $|E(v_n) - E(v)| \leq |E(\tilde{v}_n) - E(v)| + |E(1 + f_{\delta_n})| \to 0$, and $\| w_n - w \|_{V(\R^N)} \leq \| \tilde{w}_n - w \|_{V(\R^N)} + \| f_{\delta_n} \|_{V(\R^N)} \to 0$, as $n \to + \infty$, which completes the proof of the first assertion.

The two last assertions follow, once it is proved that
$$\Gamma^N(\p) = \{ w \in W(\R^N), \ {\rm s.t.} \ p(w) = \p \}$$
is not empty. This is again a consequence of Lemma \ref{quasilineaire} below.
\end{proof}

As a rather direct consequence of Lemma \ref{compactsupport}, we have

\begin{cor}
\label{assezdirect}
Let $\p > 0$. Then,
$$\limsup_{n \to +\infty} \big( E_{\min}^n(\p) \big) \leq E_{\min}(\p).$$
\end{cor}

\begin{proof}
In view of identity \eqref{eminsmooth}, given any $\delta > 0,$ there exists a map $v = 1 + w \in \{1 \} + C_c^\infty(\R^N)$ such that
$$E_{\min}(\p) \leq E(v) \leq E_{\min}(\p) + \delta, \ {\rm and} \ p(v) = \p.$$
Since $w$ has compact support in some ball $B(0, R)$, for some radius $R > 0$, the restriction of $w$ to the set $\Omega_n^N$ vanishes on the boundary $\partial \Omega_n^N$, provided $\pi n > R$, and hence defines a map in $H^1(\T_n^N)$. Adding to $w$ the constant function $1$, we have similarly $v \in H^1(\T_n^N, \C)$. Moreover, in the two-dimensional case, if $n \geq \big( \frac{R}{\pi} \big)^\frac{4}{3}$, then $w \in \boS_n^0$ (see definition \eqref{supersecteur} in Subsection \ref{luis} below). This implies
$$E(v) \geq E_{\min}^n(\p), \ \forall n \geq \Big( \frac{R}{\pi} \Big)^\frac{4}{3}.$$
Hence, 
$$E_{\min}^n(\p) \leq E_{\min}(\p) + \delta,$$
and the conclusion follows letting $\delta$ tends to zero.
\end{proof}

Next, we have

\begin{lemma}
\label{quasilineaire}
Let $N \geq 2$ and $\s >0$ be given. There exists a sequence of non-constant maps $(\gamma_n)_{n \in \N}$ in $\{ 1 \} + C_c^\infty(\R^N)$ such that
$$p(\gamma_n) = \s, \| \gamma_n \|_{W(\R^N)} \leq K \sqrt{\s}, \ {\rm and} \ E(\gamma_n) \to \sqrt{2} \s, \ {\rm as} \ n \to + \infty,$$
where $K$ is some universal constant. In particular, $E_{\min}(\p)\leq \sqrt{2} \p$, for any $\p \geq 0$, and the map $\p \mapsto \Xi(\p)$ is non-negative.
\end{lemma}

\begin{proof}
Recall that if $v = \varrho \exp i \varphi \in \{ 1 \} + C_c^\infty(\R^N)$, then the energy and momentum write
$$E(v) = \frac{1}{2} \int_{\R^N} \Big( |\nabla \varrho|^2 + |\nabla \varphi|^2 + \frac{\eta^2}{2} \Big) - \frac{1}{2} \int_{\R^N} \eta |\nabla \varphi|^2, \ {\rm and} \ p(v) = \frac{1}{2} \int_{\R^N} \eta \partial_1 \varphi.$$
As mentioned in the introduction, if one keeps only the quadratic terms, minimizing the energy for fixed momentum amounts to have $\sqrt{2} \partial_1 \varphi \simeq \eta$. For this simplified problem, the infimum is not achieved, and for minimizing sequences, transverse derivatives tend to zero, as well as the modulus. In view of this observation, we take an arbitrary map $\varphi \in \C_c^\infty(\R^N)$, and we construct by scaling and multiplication by a scalar a sequence which has the properties announced in the statements of Lemma \ref{quasilineaire}. Inspired by the scaling \eqref{scaling}, we introduce three parameters $\alpha > 1$, $\lambda > 1$ and $0 < \mu < 1$, and consider the map $\Gamma = \rho \exp i \Phi$, where the phase $\Phi$ is given by $\Phi(x) = \sqrt{2} \mu \varphi \big( \frac{x_1}{\lambda}, \frac{x_\perp}{\lambda^\alpha} \big)$ and the modulus $\rho$ by $\rho(x) = 1 - \frac{\mu}{\lambda} \partial_1 \varphi \big( \frac{x_1}{\lambda}, \frac{x_\perp} {\lambda^\alpha} \big)$. Notice in particular that, if $\mu \to 0$ and $\lambda \to + \infty$, then 
$$\sqrt{2} \partial_1 \Phi = 2(1 - \rho) \simeq 1 - \rho^2 \equiv \eta,$$
and that the transverse derivatives are of lower order. Next, we compute
\begin{equation}
\label{nallet}
\begin{split}
\int_{\R^N} |\partial_1 \Gamma|^2 = & \mu^2 \lambda^{(N - 1) \alpha - 1} \bigg( 2 \int_{\R^N} \Big( 1 - \frac{\mu}{\lambda} \partial_1 \varphi \Big)^2 (\partial_1 \varphi)^2 + \frac{1}{\lambda^2} \int_{\R^N} (\partial^2_1 \varphi)^2 \bigg),\\
\int_{\R^N} |\nabla_\perp \Gamma|^2 = & \mu^2 \lambda^{(N - 3) \alpha + 1} \bigg( 2 \int_{\R^N} \Big( 1 - \frac{\mu}{\lambda} \partial_1 \varphi \Big)^2 |\nabla_\perp \varphi|^2 + \frac{1}{\lambda^2} \int_{\R^N} |\nabla_\perp \partial_1 \varphi|^2 \bigg),\\
\frac{1}{4} \int_{\R^N} \eta^2 = & \mu^2 \lambda^{(N - 1) \alpha - 1} \int_{\R^N} \Big( 1 - \frac{\mu}{2 \lambda} \partial_1 \varphi \Big)^2 (\partial_1 \varphi)^2,
\end{split}
\end{equation}
whereas
\begin{equation}
\label{thion}
p(\Gamma) = \sqrt{2} \mu^2 \lambda^{(N - 1) \alpha - 1} \int_{\R^N} \Big(1 - \frac{\mu}{2\lambda} \partial_1 \varphi \Big) (\partial_1 \varphi)^2.
\end{equation}
For given $n \in \N$, we choose $\lambda = n$, and determine the parameter $\mu$ so that $p(\Gamma) = \s$. In particular, this choice leads to
$$\mu \sim \frac{\sqrt{\s}}{2^\frac{1}{4} \| \partial_1 \varphi \|_{L^2(\R^N)}} n^\frac{1 - (N-1) \alpha}{2}, \ {\rm as} \ n \to + \infty,$$
so that $\mu \to 0$ and $\frac{\mu}{\lambda} \to 0$, as $n \to + \infty$. In view of \eqref{nallet} and \eqref{thion}, choosing $\gamma_n = \Gamma$ (with the particular choices of $\lambda$ and $\mu$ above), we are led to
$$p(\gamma_n) = \s, \ {\rm and} \ E(\gamma_n) \sim 2 \mu^2 \lambda^{(N - 1) \alpha - 1} \int_{\R^N} (\partial_1 \varphi)^2 \sim \sqrt{2} \s, \ {\rm as} \ n \to + \infty,$$
for any $\alpha > 1$. In order to complete the proof of Lemma \ref{quasilineaire}, we now turn to the norm of the function $\gamma_n$ in $W(\R^N)$. Using the fact that $\frac{\mu}{\lambda} \to 0$, as $n \to + \infty$, we compute
$$|\Re(\gamma_n) - 1| \leq K_\varphi (|\rho - 1| + \Phi^2), |\Im(\gamma_n)| \leq K_\varphi |\Phi|, \ {\rm and} \ |\nabla \Re(\gamma_n)| \leq K_\varphi (|\nabla \rho| + |\Phi| |\nabla \Phi|),$$
where $K_\varphi$ is some constant possibly depending on $\varphi$, but not on $\alpha$ and $n$. Hence, we are led to
\begin{align*}
\int_{\R^N} |\Re(\gamma_n) - 1|^2 \leq & K_\varphi \mu^2 \lambda^{(N - 1) \alpha - 1} \bigg( \int_{\R^N} (\partial_1 \varphi)^2 + \mu^2 \lambda^2 \int_{\R^N} \varphi^4 \bigg),\\
\int_{\R^N} |\Im(\gamma_n)|^4 \leq & K_\varphi \mu^4 \lambda^{(N - 1) \alpha + 1} \int_{\R^N} \varphi^4,\\
\int_{\R^N} |\partial_1 \Re(\gamma_n)|^\frac{4}{3} \leq & K_\varphi \mu^\frac{4}{3} \lambda^{(N - 1) \alpha - \frac{5}{3}} \bigg( \int_{\R^N} (\partial^2_1 \varphi)^\frac{4}{3} + \mu^\frac{4}{3} \lambda^\frac{4}{3} \int_{\R^N} |\varphi|^\frac{4}{3} |\partial_1 \varphi|^\frac{4}{3} \bigg),\\
\int_{\R^N} |\nabla_\perp \Re(\gamma_n)|^\frac{4}{3} \leq & K_\varphi \mu^\frac{4}{3} \lambda^{(N - \frac{7}{3}) \alpha - \frac{1}{3}} \bigg( \int_{\R^N} |\nabla_\perp \partial_1 \varphi|^\frac{4}{3} + \mu^\frac{4}{3} \lambda^\frac{4}{3} \int_{\R^N} |\varphi|^\frac{4}{3} |\nabla_\perp \varphi|^\frac{4}{3} \bigg),
\end{align*}
so that, assuming that $\alpha = \frac{3}{N-1}$,
$$\| \gamma_n \|_{W(\R^N)} \sim K_\varphi \sqrt{\s}, \ {\rm as} \ n \to + \infty.$$
This concludes the proof of Lemma \ref{quasilineaire}. Indeed, the last assertions of this lemma are direct consequences of the definitions \eqref{eminent} and \eqref{discrep} of $E_{\min}$ and $\Xi$.
\end{proof}

\begin{lemma}
\label{lipschitzattitude}
We have, for any $\p, \q \geq 0$,
\begin{equation}
\label{yachvili}
|E_{\min}(\p) - E_{\min}(\q)| \leq \sqrt{2} |\p - \q|.
\end{equation}
In particular the function $\p \mapsto E_{\min}(\p)$ is Lipschitz continuous on $\R_+$, with Lipschitz's constant $\sqrt{2}$, and the function $p \mapsto \Xi(\p)$ is non-negative, non-decreasing and continuous on $\R_+$.
\end{lemma}

\begin{proof} 
We may assume without loss of generality that $\q \geq \p$. We show first that
\begin{equation}
\label{polystyrene}
E_{\min}(\q) \leq E_{\min}(\p) + \sqrt{2} (\q - \p).
\end{equation}
For that purpose, let $\delta > 0$ be given, and consider a map $v_\delta = 1 + w_\delta$, where $w_\delta \in C_c^\infty(\R^N)$, such that
$$p(v_\delta) = \p, \ {\rm and} \ E (v_\delta) \leq E_{\min}(\p) + \frac{\delta}{2}.$$
The existence of such a map $v_\delta$ follows from identity \eqref{eminsmooth} in Lemma \ref{compactsupport}. Set $\s = \q - \p$ and let $f_\delta$ be in $C_c^\infty(\R^N)$ such that $p(1 + f_\delta) = \s$, and $E(1 + f_\delta) \leq \sqrt{2} \s+ \frac{\delta}{2}$. The existence of such a map $f_\delta$ follows from Lemma \ref{quasilineaire}. We set
$$v = 1 + w_\delta + f_\delta(\cdot - a_\delta),$$
where $a_\delta \in \R^N$ is chosen so that the support of $w_\delta$ and $f_\delta(\cdot - a_\delta)$ do not intersect. In particular, we have
$$E(v) = E(v_\delta) + E(1 + f_\delta), \ {\rm and} \ p(v) = p(v_\delta) + p(1 + f_\delta) = \p + \s = \q.$$
It follows that
$$E_{\min}(\q) \leq E(v) = E(v_\delta) + \sqrt{2} \s+ \frac{\delta}{2} \leq E_{\min}(\p) + \sqrt{2}(\q - \p) + \delta,$$
which yields \eqref{polystyrene} in the limit $\delta \to 0$. Next we turn to the inequality
\begin{equation}
\label{polystyrene2}
E_{\min}(\p) \leq E_{\min}(\q) + \sqrt{2} (\q - \p).
\end{equation}
We similarly consider a map $\tilde{v}_\delta = 1 + \tilde{w}_\delta$, where $\tilde{w}_\delta \in C_c^\infty( \R^N)$, such that
$$p(\tilde{v}_\delta) = \q, \ {\rm and} \ E (\tilde{v}_\delta) \leq E_{\min}(\q) + \frac{\delta}{2}.$$
We set 
$$\tilde{v} = 1 + \tilde{w}_\delta + \check{f}_\delta(\cdot - b_\delta),$$
where the map $\check{f}_\delta$ is defined as in the proof of Lemma \ref{compactsupport}, and where $b_\delta$ is chosen so that the support of $\tilde{w}_\delta$ and $\check{f}_\delta(\cdot - b_\delta)$ do not intersect. Notice that we have
$$E \big( 1 + \check{f}_\delta(\cdot - b_\delta) \big) = E(1 + f_\delta) \leq \sqrt{2} \s + \frac{\delta}{2},$$
and
$$p \big(1 + \check{f}_\delta(\cdot - b_\delta) \big) = - p(1 + f_\delta) = - \s,$$
so that $p(\tilde{v}) = p(\tilde{v}_\delta) - \s = \p$. Hence we have
$$E_{\min}(\p) \leq E(\tilde{v}) = E(\tilde{v}_\delta) + E \big( 1 + \check{f}_\delta(\cdot - b_\delta) \big) \leq E_{\min}(\q) + \sqrt{2} \s + \delta,$$
and the conclusion \eqref{polystyrene2} follows letting $\delta$ tends to $0$. This completes the proof of Lemma \ref{lipschitzattitude}, the last assertion being a consequence of \eqref{yachvili}.
\end{proof} 

\begin{lemma}
\label{concavite}
Let $\p, \q \geq 0$. Then,
$$E_{\min} \Big( \frac{\p + \q}{2} \Big) \geq \frac{E_{\min}(\p) + E_{\min}(\q)}{2}.$$
\end{lemma}

\begin{proof}
The main idea is to construct comparison maps using a reflexion argument. For that purpose, for a map $f \in W(\R^N)$, and $a \in \R$, we consider the map $T_a^\pm f$ defined by $T_a^\pm f = f \circ P_a^\pm$, where $P_a^+$ (resp. $P_a^-$) restricted to the set $\Gamma_a^+ = \{ x = (x_1, \ldots, x_N) \in \R^N, x_N \geq a \}$ (resp. the set $\Gamma_a^- = \{ x = (x_1, \ldots, x_N) \in \R^N, x_N \leq a \}$) is the identity, whereas its restriction to the set $\Gamma_a^-$ (resp. $\Gamma_a^+$) is the symmetry with respect to the hyperplane of equation $x_n = a$. In coordinates, this reads as
\begin{align*}
T_a^+ f(x_1, \ldots, x_N) = & f(x_1, \ldots, x_N) \ {\rm if} \ x_N \geq a,\\
T_a^+ f(x_1, \ldots, x_N) = & f(x_1, \ldots, 2 a - x_N) \ {\rm if} \ x_N \leq a.
\end{align*}
One similarly defines $T_a^- f$, reversing the inequalities at the end of each line. We verify that $T_a^\pm f$ belongs to $W(\R^N)$, and that
\begin{equation}
\label{eq:tincteur}
E(T_a^\pm f) = 2E(f, \Gamma_a^\pm), \ {\rm and} \ p(T_a^\pm f) = 2 \bigg( \frac{1}{2} \int_{\Gamma_a^\pm} \langle i \partial_1 f, f - 1 \rangle \bigg).
\end{equation}
We also notice that the function $a \mapsto p(T_a^+ f)$ is continuous and, by Lebesgues's theorem, tends to zero, as $a \to + \infty$, and to $2 p(f)$, as $a \to - \infty$. Therefore, it follows by continuity that, for every $\alpha \in (0, p(f))$, there exists a number $a \in \R$ such that
\begin{equation}
\label{depart0}
p(T_a^+ f) = 2 \alpha, \ {\rm and} \ p(T_a^- f) = 2 (p(f) - \alpha).
\end{equation}
Next, we consider, for any $\p, \q \geq 0$ and any $\delta > 0$, a map $v \in W(\R^N)$ such that
$$p(v) = \frac{\p + \q}{2}, \ {\rm and} \ E(v) \leq E_{\min} \Big( \frac{\p + \q}{2} \Big) + \frac{\delta}{2}.$$
Invoking \eqref{depart0} for $f = v$ and $\alpha = \frac{p}{2}$, we may find some $a \in \R$ such that
$$p(T_a^+ v) = \p, \ {\rm and} \ p(T_a^- v) = \q.$$
It then follows from \eqref{eq:tincteur} that
$$E_{\min}(\p) \leq E(T_a^+ v) \leq 2 E(v, \Gamma_a^+),$$
and
$$E_{\min}(\q) \leq E(T_a^- v ) \leq 2 E(v, \Gamma_a^-).$$
Adding these relations, we obtain
$$E_{\min}(\p) + E_{\min}(\q) \leq 2 E(v, \Gamma_a^-) + 2 E(v, \Gamma_a^+) = 2 E(v) \leq 2 E_{\min} \Big( \frac{\p + \q}{2} \Big) + \delta.$$
The conclusion follows, letting $\delta$ tends to $0$.
\end{proof}

\begin{cor}
\label{concdecroit}
The function $\p \mapsto E_{\min}(\p)$ is concave and non-decreasing on $\R_+$.
\end{cor} 

\begin{proof}
Continuous functions $f$ satisfying the inequality
$$f \Big( \frac{\p + \q}{2} \Big) \geq \frac{f(\p) + f(\q)}{2}$$
are concave. Similarly, concave non-negative functions on $\R_+$ are non-decreasing, so that, in view of Lemmas \ref{lipschitzattitude} and \ref{concavite}, $E_{\min}$ is concave and non-decreasing on $\R_+$.
\end{proof}

\begin{proof}[Proof of Theorem \ref{proemin} completed]
Combining Lemma \ref{lipschitzattitude} and Corollary \ref{concdecroit}, all the statements in Theorem \ref{proemin} are proved, except the fact that $\Xi(\p)$ tends to $+ \infty$, as $\p \to +\infty$ (the existence of $\p_0$ being a consequence of the properties of $\Xi$). This fact is a direct consequence of the vortex solutions constructed in dimension two in \cite{BethSau1}, and of the vortex ring solutions constructed in dimension three in \cite{BetOrSm1, Chiron1}. As a matter of fact, these results show that
\begin{equation}
\label{E-infini-2}
E_{\min}(\p) \leq 2 \pi \ln(\p) + K, \ {\rm as} \ \p \to + \infty,
\end{equation}
in case $N = 2$, respectively
\begin{equation}
\label{E-infini-3}
E_{\min}(\p) \sim \pi \sqrt{\p} \ln(\p), \ {\rm as} \ \p \to + \infty,
\end{equation}
in case $N = 3$, so that $\Xi (\p) \sim \sqrt{2} \p$, as $\p \to + \infty$.
\end{proof}

\begin{remark}
We actually believe that the arguments in \cite{BethSau1} might lead to the estimate
$$E_{\min}(\p) \sim 2 \pi \ln(\p), \ {\rm as} \ \p \to + \infty.$$
\end{remark}

\subsection{Proof of Lemma \ref{speed}}

Let $\p > 0$ be given, and assume that $E_{\min}$ is achieved by a solution $u = u_\p$ of \eqref{TWc} of speed $c = c(u_\p)$. Equation \eqref{TWc}, which is the Euler-Lagrange equation for the constrained minimization problem $E_{\min}$, may be recast in a more abstract form as
$$c dp(u_\p) = dE(u_\p),$$
where $dp$ and $dE$ denote the Fr\'ech\^et differentials of $p$ and $E$ given, for any $\psi \in C_c^\infty(\R^N)$, by
$$dp(u_\p)(\psi) = \int_{\R^N} \langle i \partial_1 u_\p,\psi \rangle, \ {\rm and} \ dE(u_\p)(\psi) = - \int_{\R^N} \langle \Delta u_\p + u_\p (1 - |u_\p|^2), \psi \rangle.$$
We claim that $dp(u_\p) \neq 0$. Indeed, if we take formally $\psi_0 = u_\p - 1$, then $dp(u_\p)(\psi_0) = 2 \p \neq 0$. By density of smooth functions with compact support in $V(\R^N)$, the claim follows. Let therefore $\psi_1$ be a function in $C_c^\infty(\R^N)$ such that 
$$dp(u_\p)(\psi_1) = 1.$$
We consider the curve $\gamma : \R \mapsto W(\R^N)$ defined by $\gamma(t) = u_\p + t \psi_1$. Since the functions $E$ and $\p$ are smooth on $W(\R^N)$, we have
$$p(\gamma(t)) = \p + \s, \ {\rm where} \ \s = t + p(\psi_1) t^2,$$
$$E(\gamma(t)) = E_{\min}(\p) + c t + \underset{t \to 0}{\boO} (t^2),$$
so that
$$E_{\min}(\p + \s) - E_{\min}(\p) \leq E(\gamma(t)) - E_{\min}(\p) \leq c \s + \underset{\s \to 0}{\boO} (\s^2).$$
Conclusion \eqref{lrspeed} follows, letting $\s \to 0^\pm$.

\subsection{Proof of Lemma \ref{gueri}}

We consider in this subsection two numbers $0 \leq \p_1 < \p_2$ and assume throughout this section that $E_{\min}$ is affine on the interval $(\p_1, \p_2)$ that is
\begin{equation}
\label{hyplemme3}
E_{\min} \big( \theta \p_1 + (1 - \theta) \p_2 \big) = \theta E_{\min}(\p_1) + (1 - \theta) E_{\min}(\p_2), \forall \theta \in [0, 1].
\end{equation}
The first observation is
 
\begin{lemma}
\label{constantc}
Assume that assumption \eqref{hyplemme3} holds and that, for some $0 \leq \p_1 < \p < \p_2$ the infimum $E_{\min}(\p)$ is achieved by some function $u_\p$. Then, we have
\begin{equation}
\label{speedo}
c(u_\p) = \frac{E_{\min}(\p_2) - E_{\min}(\p_1)}{\p_2 - \p_1}.
\end{equation}
Moreover, $0 < c(u_p) < \sqrt{2}$.
\end{lemma}

\begin{proof}
Identity \eqref{speedo} is a direct consequence of Lemma \ref{speed}. For the second statement, we first notice that, in view of \eqref{start}, \eqref{speedo} and the monotonicity of $E_{\min}$, we have the inequality $0 \leq c(u_p) \leq \sqrt{2}$. Next, assuming that $c(u_\p) = 0$, the concavity of $E_{\min}$ ensures that
$$0 \leq E_{\min}(\q) \leq E_{\min}(\p_1), \forall \q \geq \p_1,$$
so that $E_{\min}$ is bounded on $\R_+$, which leads to a contradiction with \eqref{E-infini-2} or \eqref{E-infini-3} according to the value of the dimension. On the other hand, assuming that $c(u_\p) = \sqrt{2}$, in view of the concavity of $E_{\min}$, the left and right derivatives of $E_{\min}$ are non-decreasing, so that, since they are bounded by $\sqrt{2}$, they are equal to $\sqrt{2}$ on $(0, \p_2)$. By integration, we obtain that $E_{\min}(\p) = \sqrt{2} \p$ on $(0, \p_2)$, which yields a contradiction with Corollary \ref{cordepoho}. 
\end{proof}

\begin{lemma}
\label{murereflexion}
Assume that assumption \eqref{hyplemme3} holds and that, for some $0 \leq \p_1 < \p < \p_2$, the infimum $E_{\min}(\p)$ is achieved by $u_\p$. Let $\s$ be such that $(\p - \s, \p + \s) \subset (\p_1, \p_2)$. Then, there exists some number $a(\s) \in \R$ such that
\begin{equation}
\label{murrefex}
E \big( T_{a(\s)}^\pm u_\p \big) = E_{\min}(\p \pm \s), \ {\rm and} \ p \big( T_{a(\s)}^\pm u_\p \big) = \p \pm \s, 
\end{equation}
so that $E_{\min}(\p \pm \s)$ is achieved by $T_{a(\s)}^\pm u_\p$.
Moreover the map $\s \mapsto a(\s)$ is decreasing.
\end{lemma} 

\begin{proof}
We proceed as in Lemma \ref{concavite}. In view of \eqref{depart0}, we choose the value $a(\s)$ so that
$$p \big( T_{a(\s)}^\pm u_\p \big) = \p \pm \s,$$
which yields a decreasing function $\s \mapsto a(\s)$. It follows from \eqref{eq:tincteur} that
\begin{equation}
\label{pasegalite}
E_{\min}(\p \pm \s) \leq E(T_{a(\s)}^\pm u_\p) = 2 E(u_\p, \Gamma_{a(\s)}^{\pm}).
\end{equation}
Adding the relations for the $+$ and $-$ signs, we obtain
$$E_{\min}(\p - \s) + E_{\min}(\p + \s) \leq 2 E(u_\p, \Gamma_{a(\s)}^-) + 2 E(u_\p, \Gamma_{a(\s)}^+) = 2 E(u_\p) = 2 E_{\min}(\p).$$
On the other hand, by assumption \eqref{hyplemme3}, $E_{\min}(\p - \s) + E_{\min}(\p + \s) = 2 E(u_\p)$, which is only possible if we have equality in \eqref{pasegalite}, i.e. if we have identities \eqref{murrefex}.
\end{proof}

\begin{cor}
\label{yesindeed}
Assume that assumption \eqref{hyplemme3} holds and that, for some $0 \leq \p_1 < \p < \p_2$, the infimum $E_{\min}(\p)$ is achieved by some map $u_\p$.\\
i) There exist real numbers $a_1 \neq a_2$, such that
$$\partial_N u_\p =0, \on \R^{N-1} \times (a_1, a_2).$$
ii) The infimum $E_{\min}(\p)$ is achieved for some $0 \leq \p_1 < \p < \p_2$.
\end{cor}

\begin{proof}
Since for every $\s$ such that $(\p - \s, \p + \s) \subset (\p_1, \p_2)$, the infimum $E_{\min}(\p \pm \s)$ is achieved by $T_{a(\s)}^\pm u_\p$, the map $T_{a(\s)}^\pm u_\p$ solves \eqref{TWc} for some $c \geq 0$, and hence is smooth by Lemma \ref{tarquini10}. Since $T_{a(\s)}^\pm u_\p$ is obtained through a reflexion of $u_\p$ along the hyperplane of equation $x_n = a(\s)$, $T_{a(\s)}^\pm u_\p$ is of class $C^1$ if and only if 
$$\partial_N u_\p = 0 \on \R^{N - 1} \times \{ a(\s) \}.$$
The proof of statement i) follows letting $\s$ varies. For statement ii), we notice that the infimum $E_{\min}(\q)$ is achieved for every $\q \in (\p - \s, \p + \s)$, where $\s$ is such that $(\p - \s, \p + \s) \subset (\p_1, \p_2)$. The conclusion then follows from a continuity argument.
\end{proof}

\begin{proof}[Proof of Lemma \ref{gueri} completed]
We assume \eqref{hyplemme3}, and for the sake of contradiction, that for some $0 \leq \p_1 < \p < \p_2$, the infimum $E_{\min}(\p)$ is achieved by some function $u_\p$. In view of Lemma \ref{constantc}, we have $c < \sqrt{2}$. The analyticity of $\u_p$, provided by Proposition \ref{analyticity}, and Corollary \ref{yesindeed} yields
$$\partial_N u_\p = 0 \on \R^N,$$
so that $u_\p$ does not depend on the variable $x_N$. In particular, the energy would be infinite, unless $u_\p$ is constant, that is $\p = 0$, which gives the desired conclusion.
\end{proof}

\subsection{Proof of Theorem \ref{symetrie}}

The proof of Theorem \ref{symetrie} readily follows the arguments of Lopes in \cite{Lopes1}, and involves, as in the quoted reference, symmetries with respect to hyperplanes. If $e$ is a unit vector in $\R^N$, and if $W_e$ denotes the orthogonal to $e$, i.e. $W_e = (\R e)^\perp$, we consider, for any $a \in \R$, the affine hyperplane $V_{a, e} \equiv \{ a e \} + W_e$ containing $a e$ and parallel to $W_e$. Let $S_{a, e}$ be the symmetry with respect to $V_{a, e}$. The main ingredient in the proof of Theorem \ref{symetrie} is

\begin{lemma}
\label{honolulu}
Let $e$ be a unit vector such that $\langle e, e_1 \rangle = 0$, where $e_1$ is a unit vector on the $x_1$-axis. There exists a number $a \in \R$ such that $u_\p$ is symmetric with respect to the hyperplane $V_{a, e}$, that is 
$$u_\p = u_\p \circ S_{a, e}.$$
\end{lemma}

\begin{proof}
We may assume without loss of generality that $e = e_N$ is a unit vector on the $x_N$-axis and use the notations and results in the proof of Lemma \ref{concavite}. If follows from \eqref{depart0} applied to $\alpha = \frac{\p}{2}$, that there exists some $a \in \R$ such that 
\begin{equation}
\label{tahiti}
p(T_a^+ u_\p) = p(T_a^- u_\p) = 2 \alpha = \p.
\end{equation}
In view of \eqref{eq:tincteur}, we have
$$E(T_a^\pm u_\p) = 2 E(u_\p, \Gamma_a^\pm),$$
so that by summation, we obtain
$$E(T_a^+ u_\p) + E(T_a^- u_\p) = 2 E(u_\p),$$
whereas by minimality of $u_\p$ and \eqref{tahiti}, we have $E(T_a^\pm u_\p) \geq E(u_\p)$. It follows that 
$$E(T_a^+ u_\p) = E(T_a^- u_\p) = E(u_\p) = E_{\min}(\p),$$
so that $T_a^+ u_p$ and $T_a^- u_p$ are minimizers for $E_{\min}(\p)$, hence solutions to \eqref{TWc} for some $c \geq 0$, and therefore real-analytic. Since they coincide with $u_\p$ on $\Gamma_a^+$ (resp. $\Gamma_a^-$) , they coincide everywhere, that is $u_\p = T_a^+ u_\p = T_a^- u_\p$, which yields the desired conclusion.
\end{proof}

\begin{remark}
An alternative approach would be to use a unique continuation principle for elliptic operators instead of analyticity.
\end{remark}

\begin{proof}[Proof of Theorem \ref{symetrie} completed]
In the two-dimensional case, Lemma \ref{honolulu} (with $e = e_2$) already provides the proof, since in that case, it is sufficient to translate the origin to the point $a =(0, a_2)$ to conclude. In dimension three, the argument is a little more involved. We first apply Lemma \ref{honolulu} with $e = e_2$, to find some number $a_2$ such that
$$u_\p = u_\p \circ S_{a_2, e_2}.$$
Next we consider some angle $\alpha \in [0, \frac{\pi}{2}]$, and the unit vector $e_\alpha = \cos(\alpha) e_2 + \sin(\alpha) e_3$. Applying Lemma \ref{honolulu} with $e_\alpha$, we find a number $a_\alpha \in \R$ such that
$$u_\p = u_\p \circ S_{a_\alpha, e_\alpha}.$$
Notice that $R_{2 \alpha} = S_{a_{\alpha}, e_\alpha} \circ S_{a_2, e_2}$ is an affine rotation with axis $D$ parallel to $e_1$, and angle $2 \alpha$. In particular, if $\boG$ denotes the group generated by the rotation $R_{2 \alpha}$, then
\begin{equation}
\label{rotor}
u_\p = u_\p \circ g, \forall g \in \boG.
\end{equation}
If $\alpha$ is chosen so that $\frac{\alpha}{\pi}$ is irrational, then $2 \alpha \Z + 2 \pi \Z$ is dense in $\R$. In particular, denoting $R_\beta$, the affine rotation with axis $D$ and angle $\beta$, there exists a sequence of rotations $(g_n)_{n \in \N}$ in $\boG$ such that
$$g_n(x) \to R_\beta(x), \ {\rm as} \ n \to + \infty,$$
for any $x \in \R^3$. Taking the limit $n \to + \infty$ in \eqref{rotor}, we are led to the identity $u_\p = u_\p \circ R_\beta$, so that $u_\p$ is axisymmetric around axis $D$. This completes the proof of Theorem \ref{symetrie} in dimension three.
\end{proof}

\subsection{An upper bound for $E_{\min}(\p)$ in dimension two}

In this subsection, we prove

\begin{lemma}
\label{bornesupn2}
Assume $N = 2$. There exists some universal constant $K_0$ such that we have the upper bound
\begin{equation}
\label{estimsupE2}
E_{\min}(\p) \leq \sqrt{2} \p - \frac{48 \sqrt{2}}{\boS_{KP}^2} \p^3 + K_0 \p^4,
\end{equation}
for any $\p$ sufficiently small. Here, $\boS_{KP}$ denotes the action $S(w)$ of the ground-state solutions $w$ to equation \eqref{SW}.
\end{lemma}

\begin{proof}
We construct a comparison map using the formal scaling given by \eqref{scaling}. We start with a non-trivial solution $w$ to equation \eqref{SW}, which is a scalar function, and construct phase and modulus for the comparison map. For that purpose, we recall that, by \cite{deBoSau1, deBoSau2}, $w$ is a smooth function and belongs to $L^q(\R^2)$ for any $q > 1$. Moreover, its gradient belongs to $L^q(\R^2)$ for any $q > 1$ as well. Its Fourier transform satisfies the relation 
\begin{equation}
\label{SL-fou}
\widehat{w}(\xi) = \frac{1}{2} \frac{\xi_1^2}{|\xi|^2 + \xi_1^4} \widehat{w^2}(\xi).
\end{equation}
As a consequence of \eqref{SL-fou}, there exists a smooth function $v$ which solves $w = \partial_1 v$. Indeed, at least formally, the distribution $v$ whose Fourier transform is given by 
\begin{equation}
\label{def-v}
\widehat{v}(\xi) = - \frac{i}{2} \frac{\xi_1}{|\xi|^2 + \xi_1^4} \widehat{w^2}(\xi), 
\end{equation}
has the desired property. More precisely, the necessary properties of $v$ can be deduced from properties of the kernel $H_0$ given by
$$\widehat{H_0}(\xi) = \frac{\xi_1}{|\xi|^2 + \xi_1^4}.$$
First, we notice that, by Theorems 3 and 4 of \cite{Graveja7}, the kernel $H_0$ belongs to $L^q(\R^2)$ for any $2 < q < + \infty$. Since the function $w^2$ as well as any of its derivatives is a smooth function in $L^1(\R^2) \cap L^\infty(\R^2)$, the map $v$ defined in \eqref{def-v} is also smooth and belongs to $L^q(\R^2)$ for any $q > 2$, and then relations \eqref{SL-fou} and \eqref{def-v} yield as expected $w = \partial_1 v$. We also deduce that the gradient of $v$ belongs to $L^q(\R^2)$ for any $q > 1$. Indeed, the first order partial derivatives of $v$ are given by
$$\partial_j \widehat{v}(\xi) = \frac{1}{2} \frac{\xi_1 \xi_j}{|\xi|^2 + \xi_1^4} \widehat{w^2}(\xi).$$
By Propositions 1 and 2 of \cite{Graveja7}, the kernels $H_j$ defined by
$$\widehat{H_j}(\xi) = \frac{\xi_1 \xi_j}{|\xi|^2 + \xi_1^4}$$
belong to $L^q(\R^2)$ for any $1 < q < 3$, if $j = 1$, $1 < q < \frac{3}{2}$, otherwise. Since $w^2$ is in $L^q(\R^2)$ for any $q \geq 1$, the function $\nabla v$ does belong to $L^q(\R^2)$ for any $q > 1$. 

We are now in position to define a comparison map $w_\epsilon$ making use of the formal scaling \eqref{scaling}
\begin{equation}
\label{ansatz}
w_\epsilon = \varrho_\epsilon \exp i \varphi_\epsilon, 
\end{equation}
where $\varrho_\epsilon$ and $\varphi_\epsilon$ are the functions defined by 
\begin{equation}
\label{rho-epsilon}
\varrho_\epsilon(x_1, x_2) = 1 - \frac{t \epsilon^2}{2} w \bigg( \epsilon x_1, \frac{\epsilon^2}{\sqrt{2}} x_2 \bigg),
\end{equation}
and
\begin{equation}
\label{phi-epsilon}
\varphi_\epsilon(x_1, x_2) = \frac{t \epsilon}{\sqrt{2}} v \bigg( \epsilon x_1, \frac{\epsilon^2}{\sqrt{2}} x_2 \bigg).
\end{equation}
Here, $t$ denotes some positive parameter to be fixed later, whereas the constant $\epsilon$ is chosen so that the identity
\begin{equation}
\label{p-epsilon}
\p = \frac{\epsilon}{72} \int_{\R^2} w^2
\end{equation}
holds. We claim that the map $w_\epsilon$ belongs to $W(\R^2)$. Indeed, the integrability properties for $v$ and $w$ first show that the functions $\varrho_\epsilon - 1$, $\nabla \varrho_\epsilon$ and $\nabla \varphi_\epsilon$ belong to $L^q(\R^2)$ for any $q > 1$, whereas $\varphi_\epsilon$ is in $L^q(\R^2)$ for any $q > 2$. On the other hand, it follows from definitions \eqref{ansatz}, \eqref{rho-epsilon} and \eqref{phi-epsilon}, and the boundedness of $v$ and $w$ that there exists some constant $K > 0$ such that
\begin{align*}
|\Re(w_\epsilon) - 1| = |\varrho_\epsilon \cos(\varphi_\epsilon) - 1| \leq K \big( |\varrho_\epsilon - 1| + \varphi_\epsilon^2 \big), & \ |\Im(w_\epsilon)| = |\varrho_\epsilon \sin(\varphi_\epsilon)| \leq K |\varphi_\epsilon|,\\
|\nabla \Re(w_\epsilon)| \leq K \big( |\nabla \varrho_\epsilon| + |\varphi_\epsilon| |\nabla \varphi_\epsilon| \big), & \ |\nabla \Im(w_\epsilon)| \leq K \big( |\nabla \varphi_\epsilon| + |\varphi_\epsilon| |\nabla \varrho_\epsilon| \big),
\end{align*}
so that $w_\epsilon$ does belong to $W(\R^2)$.

The next step is to determine the value of the parameter $t$ so that $p(w_\epsilon) = \p$. For that purpose, we compute the momentum of $w_\epsilon$ making use of formula \eqref{celemoment}. This gives
\begin{equation}
\label{p-ansatz}
p(w_\epsilon) = \frac{\epsilon t^2}{2} \int_{\R^2} w^2 - \frac{\epsilon^3 t^3}{8} \int_{\R^2} w^3.
\end{equation}
Therefore, there exists a positive number $t_\p$ such that
$$p(w_\epsilon) = \p.$$
We expand $t_\p$ using equations \eqref{p-epsilon} and \eqref{p-ansatz}
\begin{equation}
\label{t-asympt}
t_\p = \frac{1}{6} + 18 \Bigg( \frac{\int_{\R^2} w^3}{\Big( \int_{\R^2} w^2 \Big)^3}\Bigg) \p^2 + \underset{\p \to 0}{\boO} \big( \p^4 \big).
\end{equation}
On the other hand, using definition \eqref{ansatz} and formula \eqref{PolE}, the value of $E(w_\epsilon)$ is given by
\begin{align*}
E(w_\epsilon) = & t^2 \Bigg( \bigg( \frac{\epsilon}{\sqrt{2}} \int_{\R^2} w^2 + \frac{\epsilon^3}{4 \sqrt{2}} \int_{\R^2} \Big( (\partial_1 w)^2 + (\partial_2 v)^2 \Big) + \frac{\epsilon^5}{8 \sqrt{2}} \int_{\R^2} (\partial_2 w)^2 \bigg) - t \bigg( \frac{3 \epsilon^3}{4 \sqrt{2}} \int_{\R^2} w^3\\
+ & \frac{\epsilon^5}{4 \sqrt{2}} \int_{\R^2} w \partial_2 v \bigg) + t^2 \bigg( \frac{5 \epsilon^5}{32 \sqrt{2}} \int_{\R^2} w^4 - \frac{\epsilon^7}{16 \sqrt{2}} \int_{\R^2} w^2 (\partial_2 v)^2 \bigg) \Bigg).
\end{align*}
Since for $t = t_\p$, $p(w_\epsilon) = \p$, we may invoke the definition of $E_{\min}(\p)$ to conclude that there exists some constant $K > 0$ only depending on $w$ such that 
$$E_{\min}(\p) \leq E(w_\epsilon) \leq \frac{\epsilon t_\p^2}{\sqrt{2}} \Bigg( \int_{\R^2} w^2 - \frac{\epsilon^2}{4} \bigg( \int_{\R^2} \Big( (\partial_1 w)^2 + (\partial_2 v)^2 \Big) \bigg) + \frac{3 \epsilon^2}{4} \bigg( \int_{\R^2} w^3 \bigg) t_\p \Bigg) + K \epsilon^5.$$
By \eqref{p-epsilon} and \eqref{t-asympt}, we are led to
\begin{equation}
\label{E-f}
E_{\min}(\p) \leq \sqrt{2} \p \Bigg( 1 + A_2(w) \p^2 \Bigg) + K \p^4,
\end{equation}
where the number $A_2(w)$ is given in view of \eqref{E-KP}, by
$$A_2(w) = 2592 \frac{E_{KP}(w)}{\Big( \int_{\R^2} w^2 \Big)^3}.$$
In order to complete the proof of inequality \eqref{estimsupE2}, we optimize the value of the constant $A_2(w)$ with a suitable choice of the solution $w$ to equation \eqref{SW}. For this purpose, we first relate the coefficient $A_2(w)$ with the action $S(w)$ of the considered solution $w$. By \cite{deBoSau1,Graveja7}, the following equalities hold
\begin{equation}
\label{E-S-w2}
E_{KP}(w) = - \frac{1}{6} \int_{\R^2} w^2, \ {\rm and} \ S(w) = \frac{1}{3} \int_{\R^2} w^2,
\end{equation}
so that
$$A_2(w) = - \frac{48}{S(w)^2}.$$
Since $S(w) > 0$ by the second equation of \eqref{E-S-w2}, it remains to minimize the action $S(w)$ among all non-trivial solutions $w$ to equation \eqref{SW}. In view of the definition of ground states, the minimizer $\boS_{KP}$ is precisely the action of a ground-state solution to equation \eqref{SW}, so that the optimal inequality provided by \eqref{E-f} is exactly \eqref{estimsupE2}.
\end{proof}

\section{Properties of solutions on $\T_n^N$}
\label{torus}

The purpose of this section is to present a number of results concerning solutions to \eqref{TWc} on the torus which are of presumably independent interest. Many of the results presented here have already been derived elsewhere (in particular in \cite {BetOrSm1}), possibly in a slightly different form. In order to be self-consistent, we wish to give here a self-contained presentation. Since many of the results in this section are valid in higher dimensions, we consider more generally the torus $\T_n^N$ in dimension $N$, and let $v \equiv v_c^n$ be an arbitrary non-trivial solution to \eqref{TWc} on $\T_n^N$.

As we have seen in the introduction, working on tori has a number of important advantages. The first one is that they are compact, allowing to establish quite easily existence of minimizing solutions. The second one is that the torus has some invariance by translation. Working on tori introduces however a number of small new difficulties. Some of them are related to the identification $\T_n^N \simeq [-\pi n, \pi n]^N$, which we will clarify next, as well as we recall the notion of unfolding.

\subsection{Working on tori}

Working on tori presents a number of peculiarities which we would like to point out in this subsection. We begin with the usual definition $\T_n^N = \R^N / (2 \pi n \Z)^N$ obtained by the identification $x \sim x'$ if and only if $x - x' \in (2 \pi n \Z)^N$. For $\alpha = (\alpha_1, \ldots, \alpha_N) \in \R^N$, the cube $C_\alpha \equiv \underset{i = 1}{\overset{N}{\Pi}} [- \pi n + \alpha_i, \pi n + \alpha_i[$ contains a unique element of each equivalence class ($C_\alpha$ is often termed a fundamental domain). It may therefore be identified with $\T_n^N$. Given $\alpha \in \R^N$, the unfolding $\tau_\alpha$ of $\T_n^N$ associated to $\alpha$ is by definition the one-to-one mapping
\begin{eqnarray*}
\tau_\alpha : & \T_n^N & \longrightarrow \Omega_n^N \equiv [- \pi n, \pi n[^N\\
& p = [(x_1 + \alpha_1, \ldots, x_N + \alpha_N)] & \longmapsto (x_1, \ldots, x_N).
\end{eqnarray*}
This corresponds to a translation of the origin in $\R^N,$ and thus on the torus. For a given function $f$ defined on $\T_n^N,$ each unfolding $\tau_\alpha$ induces a $2 \pi n$-periodic function $f_\alpha$ defined on $\Omega_n^N$.

In some computations (in particular, dealing with integration by parts for functions which are not necessarily all periodic), we will need to estimate boundary integrals. The following lemma provides a choice of a ``good'' unfolding of the torus, by averaging.

\begin{lemma}
\label{average0}
Let $f$ be a $2 \pi n$-periodic function of $L^1(\Omega_n^N)$, and let $A$ be a subset of $[- \pi n, \pi n]$. There exists some $\alpha_N$ in $[- \pi n, \pi n] \setminus A$ such that
\begin{equation}
\label{average00}
\bigg| \int_{[-\pi n, \pi n]^{N-1} \times \{ \alpha_N \}} f(x) dx \bigg| \leq \frac{1}{2 \pi n - |A|} \int_{\Omega_n^N} |f(x)| dx. 
\end{equation}
In particular, we may find an unfolding $\tau_\alpha$ of the torus $\T_n^N$ such that
$$\bigg| \int_{[- \pi n, \pi n]^{N-1} \times \{ - \pi n, \pi n \}} f_\alpha(x) dx \bigg| \leq \frac{2}{2 \pi n - |A|} \int_{\Omega_n^N} |f(x)| dx,$$
and
$$\bigg| \int_{\partial \Omega_n^N} f_\alpha(x) dx \bigg| \leq \bigg( \frac{2}{2 \pi n - |A|} + \frac{N - 1}{\pi n} \bigg) \int_{\Omega_n^N} |f(x)| dx.$$
\end{lemma}

\begin{proof}
Integrate the l.h.s of \eqref{average00} for $\alpha_N \in [- \pi n, \pi n] \setminus A$ and use the mean-value theorem.
\end{proof}

\begin{remark}
The trace of $f_\alpha$ is well-defined for almost every unfolding. In the sequel, we will no longer distinguish $f$ and $f_\alpha$: this hopefully will not lead to any confusion.
\end{remark}

\begin{remark}
Recall that, on $X_n^N \equiv H^1(\T_n^N)$, we have defined the momentum as
$$p_n(v) = \frac{1}{2} \int_{\T_n^N} \langle i \partial_1 v , v \rangle.$$
For a map $u$ in the space
$$Y_n^N \equiv \{u \in H^1(\Omega_n^N, \C), u \equiv 1 \on \partial \Omega_n^N \} \subset X_n^N,$$
we claim that we have
\begin{equation}
\label{momentJ}
p_n(u) = m_n(u) \equiv \frac{1}{2} \int_{\Omega_n^N} \langle Ju , \zeta_1 \rangle.
\end{equation}
Here, the Jacobian $Ju$ of $u$ denotes the $2$-form defined on $\Omega_n^N$ by 
$$Ju \equiv \frac{1}{2} d(u \times du) = \sum_{1 \leq i < j \leq N} (\partial_i u \times \partial_j u) dx_i \wedge dx_j,$$
and $\zeta_1$ denotes the $2$-form defined on $\Omega_n^N$ by
\begin{equation}
\label{eq:zeta}
\zeta_1(x) \equiv - \frac{2}{N - 1} \sum_{i = 2}^N x_i dx_1 \wedge dx_i.
\end{equation}
Finally, $\langle \cdot, \cdot \rangle$ stands for the scalar product of $2$-forms. Claim \eqref{momentJ} reduces in fact to an integration by parts. Indeed, we write
\begin{equation}
\label{integrepartie}
2 \langle Ju, \zeta_1 \rangle dx_1 \wedge \ldots \wedge dx_N = d(u \times du) \wedge \star \zeta_1 = d((u \times du) \wedge \star \zeta_1) + (u \times du) \wedge d(\star \zeta_1),
\end{equation}
where $\star$ denotes the Hodge-star operator for differential forms. The special choice of $\zeta_1$ yields the identity
$$(u \times du) \wedge d(\star \zeta_1) = 2 \langle i \partial_1 u, u \rangle dx_1 \wedge \ldots \wedge dx_N.$$
Combining with \eqref{integrepartie} and integrating on the torus, we are led to
\begin{equation}
\label{integre2}
\int_{\T_n^N} \langle i \partial_1 u, u \rangle - \int_{\Omega_n^N} \langle Ju, \zeta_1 \rangle = - \frac{1}{2} \int_{\partial \Omega_n^N} (u\times du)_\top \wedge (\star \zeta_1)_\top.
\end{equation}
If $u$ belongs to $Y_n^N$, then the boundary term is zero, and claim \eqref{momentJ} follows.
 
As we will see later, the term $m_n(u)$ explicitly appears in Pohozaev's formula. On the torus $\T_n^N$ however, $m$ is {\bf not} well-defined (due to the fact that the $2$-form $\zeta_1$ is not periodic, and hence well-defined on the torus). We will circumvent this difficulty by choosing suitable unfoldings.
\end{remark}

\subsection{Lifting properties and topological sectors}
\label{luis}

In several places of this paper, we have to face the following situation. Let $R \geq 1$ and $\ell_0 \in \N^*$ be given, and consider $\ell$ points $x_1$, $\ldots$, $x_\ell$ on $\T_n^N$ with $\ell \leq \ell_0$. Assume that we have
\begin{equation}
\label{relou}
|x_i - x_j| \geq 2 R,
\end{equation}
for any $i \neq j$, and that $v$ is a map in $H^1(\T_n^N)$ such that
\begin{equation}
\label{tropcool}
|v(x)| \geq \frac{1}{2}, \ \forall x \in \boO_n \Big( \frac{R}{2} \Big) = \T_n^N \setminus \underset{j = 1}{\overset{\ell}{\cup}} B \Big( x_j, \frac{R}{2} \Big).
\end{equation}
The problem we wish to investigate is the following: find conditions such that one may lift the map $v$ on $\boO_n(R) = \T_n^N \setminus \underset{j = 1}{\overset{\ell}{\cup}} B(x_j, R)$ as $v = \varrho \exp i \varphi$, with $\varphi \in H^1(\boO_n(R))$. In dimension three, we have a simple answer.

\begin{lemma}
\label{massimo3}
Assume $N = 3$. Given any numbers $E > 0$, $R > 1$ and $\ell_0 \geq 1$, there exists a constant $n(E, R, \ell_0)$, such that if $n \geq n(E, R, \ell_0)$, and $v \in H^1(\T_n^3, \C)$ satisfies \eqref{relou}, \eqref{tropcool} and
\begin{equation}
\label{liftprop2}
E_n(v) \leq E,
\end{equation}
then $v = |v| \exp i \varphi$ on $\boO_n(R)$, with $\varphi \in H^1(\boO_n(R), \R)$.
\end{lemma}

\begin{proof}
We consider first a special case.

\setcounter{case}{0}
\begin{case}
\label{entire}
${\bf \boO_n(R) = \T_n^3}${\bf, i.e. }${\bf |v| \geq \frac{1}{2}}$ {\bf on $\T_n^3$.} In this case we may write $v = |v| w$, with $|w| = 1$, and perform the Hodge-de-Rham decomposition of $w \times dw$. Since $d(w \times dw) = 0$, the Hodge-de-Rham decomposition is written as
$$w \times dw = d\varphi + \underset{j = 1}{\overset{3}{\sum}} \alpha_j dx_j,$$
where $\varphi$ is a $H^1$-function, each $\alpha_j$ is a real number, and $dx_j$ stands for the canonical harmonic forms on $\T_n^3$. One then checks that
$$w(x) = \exp i \Big(\varphi(x) + \sum_{j=1}^3 \alpha_j x_j + \theta \Big),$$
for some constant $\theta \in \R$. Periodicity implies that $\alpha_j$ has the form $\alpha_j = \frac{k_j}{n}$, for some integer $k_j$. The $L^2$-orthogonality of the Hodge-de-Rham decomposition yields
\begin{equation}
\label{delasau}
\| w \times dw \|_{L^2(\T_n^3)}^2 = \| d\varphi \|_{L^2(\T_n^3)}^2 + 8 \pi^3 n \sum_{j = 1}^3 k_j^2,
\end{equation}
which implies by \eqref{liftprop2},
$$\sum_{j = 1}^3 k_j^2 \leq \frac{1}{8 \pi^3 n} \| \nabla w \|_{L^2(\T_n^3)}^2 \leq \frac{E}{\pi^3 n}.$$
Choosing $n(E,0,0) = \frac{2}{\pi^3} E$, the previous inequality implies $k_j = 0$ for $n \geq n(E,0,0)$. On the other hand, it follows from \eqref{delasau} that $\varphi \in H^1(\T_n^N, \R)$, so that the conclusion of Lemma \ref{massimo3} holds.
\end{case}

\begin{case}
{\bf The general case}. Arguing by a density argument, we may assume without loss of generality, that $v$ is smooth. Next we claim that there exists a map $\tilde{v} = |\tilde{v}| \tilde{w}$, with $\tilde{w} = 1$, such that $|\tilde{v}| \geq \frac{1}{2}$ on $\T_n^3$, and
\begin{equation}
\label{outside}
\tilde{v} = v \ {\rm on} \ \boO_n(R),
\end{equation}
and
\begin{equation}
\label{devilliers}
\| \nabla \tilde{w} \|_{L^2(\T_n^3)}^2 \leq K(R, \ell_0) E(v),
\end{equation}
where $K(R, \ell_0)$ is some constant depending only on $R$ and $\ell_0$. The conclusion then follows using the argument of Case \ref{entire} for the function $\tilde{v}$. In particular, it only remains to prove Claims \eqref{outside}-\eqref{devilliers}.
\end{case}
\end{proof}

\begin{proof}[Proof of Claims \eqref{outside}-\eqref{devilliers}]
The construction of $\tilde{v}$ relies on some standard topological arguments. Using the mean-value inequality, there exists some radius $\frac{R}{2} < R_j < R$ such that
\begin{equation}
\label{roncero}
\int_{\partial B(x_j, R_j)} e(v) \leq \frac{2}{R} \int_{B(x_j, R)} e(v).
\end{equation}
On the other hand, $v$ is a continuous function on $\partial B(x_j, R_j)$ such that $|v| \geq \frac{1}{2}$, so that it can be written as $v = \varrho \exp i \varphi$ on $\partial B(x_j, R_j)$. Denoting $\varrho_j$, resp. $\varphi_j$, the harmonic extensions of $\varrho$, resp. $\varphi$, on $B(x_j, R_j)$, we consider the function $\tilde{v}$ defined by
\begin{equation}
\label{agulla}
\begin{split}
\tilde{v} = & v \ {\rm on} \ \T_n^3 \setminus \underset{j = 1}{\overset{\ell}{\cup}} B(x_j, R_j),\\
\tilde{v} = & \varrho_j \exp i \varphi_j \ {\rm on} \ B(x_j, R_j),
\end{split}
\end{equation}
so that $\tilde{v}$ satisfies \eqref{outside}, and $|\tilde{v}| \geq \frac{1}{2}$ on $\T_n^3$ (by \eqref{tropcool} and the maximum principle). Moreover, using standard trace theorems, the harmonicity of $\varphi_j$ gives
$$\int_{B(x_j, R_j)} |\nabla \tilde{w}|^2 = \int_{B(x_j, R_j)} |\nabla \varphi_j|^2 \leq K(R) \int_{\partial B(x_j, R_j)} |\nabla \varphi|^2 \leq K(R) \int_{\partial B(x_j, R_j)} e(v),$$
where $K(R)$ denotes some constant only depending on $R$, so that, by \eqref{roncero},
$$\int_{B(x_j, R_j)} |\nabla \tilde{w}|^2 \leq K(R) \int_{B(x_j, R)} e(v).$$
Claim \eqref{devilliers} then follows from definition \eqref{agulla}, and the fact that $\ell \leq \ell_0$.
\end{proof}

In dimension two, the situation is very different. In particular, the minimal energy of harmonic $1$-forms does not depend on the size $n$ of the torus (by conformal invariance of the energy), so that the statement of Lemma \ref{massimo3} does not extend. As an example, one may take for instance the map $v_n = \exp i \frac{x_1}{n}$, whose energy does not depend on $n$, and which is not liftable. However, when the energy is small, the situation is parallel to the three-dimensional case.

\begin{lemma}
\label{massimo2}
Assume $N = 2$, and consider a function $v \in H^1(\T_n^2)$ such that $E(v) \leq E_0 = \frac{\pi^2}{3}$, and
$$|v| \geq \frac{1}{2} \ {\rm on} \ \T_n^2,$$
then $v = |v| \exp i \varphi$ on $\T_n^2$, with $\varphi \in H^1(\T_n^2, \R)$.
\end{lemma}

\begin{proof}
The proof is identical to Case \ref{entire} of the proof of Lemma \ref{massimo3}, apart from equation \eqref{delasau}, which becomes here 
$$\| w \times dw \|_{L^2(\T_n^2)}^2 = \| d\varphi \|_{L^2(\T_n^2)}^2 + 4 \pi^2 \sum_{j = 1}^3 k_j^2,$$
so that
$$\sum_{j = 1}^3 k_j^2 \leq \frac{1}{4 \pi^2} \| \nabla w \|_{L^2(\T_n^2)}^2 \leq \frac{2 E_0}{\pi^2} \leq \frac{2}{3}.$$
Hence, the numbers $k_j$ are identically equal to $0$, so that the conclusion of Lemma \ref{massimo2} holds, following the lines of Case \ref{entire} of the proof of Lemma \ref{massimo3}.
\end{proof}

If we wish to have a lifting property when the energy is not small, we need to restrict the class of test maps. Indeed, although the zero set of Ginzburg-Landau maps on $\T_n^2$ may be not empty, a restriction on the Ginzburg-Landau energy allows us to define a notion of degree with suitable continuity properties. First, notice that by Sobolev's embedding theorem, for $v \in H^1(\T_n^2)$ and $j \in \{ 1, 2 \}$, the restriction $v$ to $I_{n, r}^j$ is continuous for almost any $r \in [-\pi n,\pi n]$, where we have set
$$I_{n, r}^1 = \{ r \} \times [-\pi n, \pi n] \simeq \S_n^1, \
{\rm and } \ I_{n, r}^2 = [-\pi n, \pi n] \times \{ r \} \simeq \S_n^1.$$
In particular, if $v$ does not vanish on $I_{n, r}^j$, in view of periodicity, we may define the degree of $\frac{v}{|v|}$ restricted to $I_{n, r}^j$. Denoting $B^j(v)$, the subset of numbers $r$ in $[-\pi n, \pi n]$ for which the restriction of $v$ to $I_{n, r}^j$ is continuous and does not vanish, we set
$$T_n^{d_1, d_2} = \bigg\{ u \in H^1(\T_n^2), \ {\rm s.t.} \ \forall j \in \{ 1, 2 \}, \exists B^j \subset B^j(u), \ {\rm s.t.} \ \left\{ \begin{array}{ll} |B^j| \geq 2 \pi \big( n - n^\frac{3}{4} \big),\\ {\rm deg}(u, I_{n,r}^j) = d_j, \forall r \in B^j \end{array} \right. \bigg\},$$
for any $(d_1, d_2) \in \Z^2$. It follows from this definition that $T_n^{d_1, d_2} \cap T_n^{d'_1, d'_2} = \emptyset$ if $(d_1, d_2) \neq (d'_1, d'_2)$, and $\underset{d \in \Z^2}{\cup} T_n^d \neq H^1(\T_n^2)$. In the introduction, we have introduce the set $\boS_n^0$, which we next define as
\begin{equation}
\label{supersecteur}
\boS_n^0 = T_n^{0, 0}.
\end{equation}
We claim that if a map $v$ in $\boS_n^0$ satisfies suitable topological assumptions, then $v$ has a lifting.

\begin{lemma}
\label{sectorisation}
Assume $N = 2$, and that $v \in \boS_n^0 = T_n^{0, 0}$ satisfies \eqref{relou}, \eqref{tropcool} and 
\begin{equation}
\label{zerobord}
\deg \Big( \frac{v}{|v|}, \partial B(x_j, R) \Big) = 0, \forall j \in \{ 1, \ldots, \ell \}.
\end{equation}
Then, there exists a constant $n(R, \ell_0)$ depending only on $R$ and $\ell_0$ such that, if $n \geq n(R, \ell_0)$, then $v = |v| \exp i \varphi$ on $\boO_n(R)$, with $\varphi \in H^1(\boO_n(R), \R)$.
\end{lemma}

\begin{proof}
The proof is very similar to the proof of Lemma \ref{massimo3}. Using \eqref{zerobord}, we may construct in the three-dimensional case as in the two-dimensional case a map $\tilde{v}$ satisfying \eqref{outside}, and such that $|\tilde{v}| \geq \frac{1}{2}$ on $\T_n^2$. Moreover, since $\ell \leq \ell_0$, we may check that $\tilde{v}$ belongs to $\boS_n^0$ for $n$ sufficiently large, so that we may restrict ourselves to the case $\boO_n(R) = \T_n^2$. In this situation, denoting $v = |v| w$, with $|w| = 1$, the argument of Case \ref{entire} of the proof of Lemma \ref{massimo3} gives that there exist a function $\varphi \in H^1(\T_n^2)$, a constant $\theta$ and integers $k_j$ such that
$$w(x) = \exp i \Big(\varphi(x) + \sum_{j=1}^2 \frac{k_j x_j}{n} + \theta \Big),$$
for any $x \in \T_n^2$. It follows that $v$ is in $T_n^{k_1, k_2}$, so that $k_1 = k_2 = 0$, which completes the proof of Lemma \ref{sectorisation}.
\end{proof}

We will use the following consequence in the spirit of Lemma \ref{massimo2}.

\begin{cor}
\label{minimo2}
Assume $N = 2$, and that $v \in \boS_n^0$ satisfies \eqref{relou}, \eqref{tropcool} and 
\begin{equation}
\label{austin}
E_n \Big( v, \boO_n \Big( \frac{R}{2} \Big) \Big) < \frac{\pi}{8}.
\end{equation}
Then, there exists a constant $n(R, \ell_0)$ depending only on $R$ and $\ell_0$ such that, if $n \geq n(R, \ell_0)$, then $v = |v| \exp i \varphi$ on $\boO_n(R)$, with $\varphi \in H^1(\boO_n(R), \R)$.
\end{cor}

\begin{proof}
Corollary \ref{minimo2} is a direct consequence of Lemma \ref{sectorisation}, once it is proved that assumption \eqref{zerobord} is a consequence of assumption \eqref{austin}. This last fact follows from a direct computation of the topological degree of $\frac{v}{|v|}$ on $\partial B(x_j, R)$. Indeed, by the mean-value inequality, there exists some radius $\frac{R}{2} \leq r_j \leq R$ such that
$$\int_{\partial B(x_j, r_j)} e(v) \leq \frac{2}{R} E_n \Big( v, \boO_n \Big( \frac{R}{2} \Big) \Big),$$
so that by Cauchy-Schwarz's inequality,
\begin{equation}
\label{ron}
\int_{\partial B(x_j, r_j)} |\nabla v| \leq \sqrt{\frac{8 \pi r_j}{R}} E_n \Big( v, \boO_n \Big( \frac{R}{2} \Big)\Big)^\frac{1}{2}.
\end{equation}
However, the topological degree of $\frac{v}{|v|}$ on $\partial B(x_j, r_j)$ is defined by
$$\deg \Big( \frac{v}{|v|}, \partial B(x_j, r_j) \Big) = \frac{1}{2 \pi} \int_{\partial B(x_j, r_j)} \frac{v}{|v|} \wedge \partial_\tau \Big( \frac{v}{|v|} \Big),$$
where $\tau$ denotes the properly oriented unit tangent vector to $\partial B(x_j, r_j)$. Hence, we have by \eqref{tropcool}, \eqref{austin} and \eqref{ron},
$$\deg \Big( \frac{v}{|v|}, \partial B(x_j, r_j) \Big) \leq \frac{1}{2 \pi} \int_{\partial B(x_j, r_j)} \frac{|\nabla v|}{|v|} \leq \sqrt{\frac{8}{\pi}} E_n \Big( v, \boO_n \Big( \frac{R}{2} \Big)\Big)^\frac{1}{2} < 1,$$
so that
$$\deg \Big( \frac{v}{|v|}, \partial B(x_j, r_j) \Big) = 0.$$
Since $v$ does not vanish on $B(x_j, R) \setminus B(x_j, r_j)$, assumption \eqref{zerobord} holds. The conclusion then follows invoking Lemma \ref{sectorisation}.
\end{proof}

It remains to verify that these sets have appropriate properties with respect to the methods of calculus of variations. For that purpose we restrict ourselves to the sublevel sets $E_{n, \Lambda}$ of $H^1(\T_n^2)$ defined by
$$E_{n, \Lambda} = \{ u \in H^1(\T_n^2), \ {\rm s.t.} \ E_n(u) \leq \Lambda \}.$$
and set
$$S^{d_1, d_2}_{n, \Lambda} = E_{n, \Lambda} \cap T_n^{d_1, d_2}.$$
The following result was readily proved by Almeida (see Theorem 6 in \cite{Almeida1}).

\begin{theorem}[\cite{Almeida1}] 
\label{lemmalulu}
Let $\Lambda > 0$ be given. There exists an integer $n_\Lambda$, such that for every $n \geq n_{\Lambda}$, we have the partition
$$E_{n, \Lambda} = \bigcup_{(d_1, d_2) \in \Z^2} S^{d_1, d_2}_{n, \Lambda}.$$
Moreover, the degree application
\begin{align*}
\deg : E_{n, \Lambda} & \to \Z^2\\
u & \mapsto \deg(u) = (d_1(u), d_2(u))
\end{align*}
is continuous on $E_{n, \Lambda}$, so that $S^{d_1, d_2}_{n, \Lambda}$ is a closed subset of $H^1(\T_n^2)$, whereas $\dot{S}^{d_1, d_2}_{n, \Lambda} \equiv \{ u \in H^1(\T_n^2), \ {\rm s.t.} \ E_n(u) < \Lambda \} \cap T_n^{d_1, d_2}$ is an open subset of $H^1(\T_n^2)$.
\end{theorem}

Notice that in \cite{Almeida1} there is a small parameter 
$\varepsilon > 0$ appearing in the energy functional. In our case, $\varepsilon$ corresponds to $\varepsilon = \frac{1}{n}$, and one recovers the context of Theorem 6 in \cite{Almeida1} (in the particular case of the torus $\T^2 = \R / (2 \pi \Z)$) performing the change of scale $\tilde{x} = \frac{x}{n}$, with the exception of one major difference. Indeed, the sets $T_n^{d_1,d_2}$ are defined in \cite{Almeida1} by
$$T_n^{d_1, d_2} = \bigg\{ u \in H^1(\T_n^2), \ {\rm s.t.} \ \forall j \in \{ 1, 2 \}, \exists B^j \subset B^j(u), \ {\rm s.t.} \ \left\{ \begin{array}{ll} |B^j| \geq \frac{3 \pi n}{2},\\ {\rm deg}(u, I_{n,r}^j) = d_j. \end{array} \right. \bigg\}.$$
The difference in the assumption on the length of the sets $B_j$ comes from the fact that the proof performed in \cite{Almeida1} requires that the sets $B_j$ are at some distance larger than $\frac{\pi n}{4}$ ($\frac{1}{8}$ in the $\varepsilon$ context of \cite{Almeida1}) from the boundary of a suitable unfolding of the torus $\T_n^2$. However, it can be proved using exactly the same arguments as in \cite{Almeida1} that this assumption can be removed by a less restrictive one, where the sets $B_j$ are at some distance only larger than $\frac{\pi n^\frac{3}{4}}{2}$ ($\frac{\varepsilon^\frac{1}{4}}{2}$ in the $\varepsilon$ context of \cite{Almeida1}) from the boundary of a suitable unfolding of the torus $\T_n^2$, so that Theorem 6 in \cite{Almeida1} extends to the case considered here.

\subsection{Pointwise estimates on $\T_n^N$}
\label{ellipticpoint}

Our first result provides local bounds. For a domain $\boU \subset \T_n^N$, we consider the local energy density
$$ E_n(v, \boU) \equiv \int_{\boU} e(v) \equiv \int_{\boU} \bigg( \frac{|\nabla v|^2}{2} + \frac{(1 - |v|^2)^2}{4} \bigg).$$ 
As for the whole space, we have on the torus $\T_n^N$ the pointwise estimates.

\begin{lemma}
\label{tarquini1}
Let $n \in \N^*$, and let $v$ be a finite energy solution to \eqref{TWc} on $\T_n^N$. There exist some constants $K(N)$ and $K(c, k, N)$ such that
$$\Big\| 1 - |v| \Big\|_{L^\infty(\T_n^N)} \leq \max \Big\{ 1 , \frac{c}{2} \Big\},$$
$$\| \nabla v \|_{L^\infty(\T_n^N)} \leq K(N) \Big( 1 + \frac{c^2}{4} \Big)^\frac{3}{2},$$
and more generally,
$$\| v \|_{C^k(\T_n^N)}\leq K(c, k, N), \forall k \in \N.$$
\end{lemma}

\begin{lemma}
\label{tarquini2}
Let $n \in \N^*$ and $r > 0$. Assume that $v$ is a finite energy solution to \eqref{TWc} on $\T_n^N$. There exists some constant $K(N)$ such that for any $x_0 \in \T_n^N$,
$$\Big\| 1 - |v| \Big\|_{L^\infty (B(x_0, \frac{r}{2}))} \leq \max \Big\{ K(N) \Big( 1 + \frac{c^2}{4} \Big)^2 E_n \big( v, B(x_0, r) \big)^\frac{1}{N+2}, \frac{K(N)}{r^N} E_n \big( v, B(x_0, r) \big)^\frac{1}{2} \Big\},$$
\end{lemma}

The proofs are, almost word for word, identical to the proofs of Lemmas \ref{tarquini10} and \ref{tarquini20} respectively. Therefore, we omit them.

\subsection{Upper bounds for the velocity}

We first notice that, if $v$ is non-constant, there is at most one value of $c$ for which $v$ might be a solution to \eqref{TWc}: we sometimes emphasize this fact writing for a non-trivial solution $c = c(v)$.

We consider next, for a solution $v$ to \eqref{TWc}, the discrepancy term
$$\Sigma_n(v) = \sqrt{2} p_n(v) - E_n(v).$$
The main result of this subsection asserts, that, in dimensions two and three, if $\Sigma_n(v) > 0$, the speed $c(v)$ can be bounded by a function of $\Sigma_n(v)$ and the energy $E_n(v)$. More precisely, we have 

\begin{theorem}
\label{bornec}
Assume $N = 2$ or $N = 3$, and let $E_0 > 0$ and $\Sigma_0 > 0$ be given. Let $v$ be a non-trivial finite energy solution to \eqref{TWc} in $X_n^2 \cap \boS_n^0$, resp. $X_n^3$, with $c = c(v) \in \R$, $E_n(v) \leq E_0$ and 
$$0 < \Sigma_0 \leq \Sigma_n(v).$$
Then, there is some constant $n_0 \in \N$ depending only on $E_0$ and $\Sigma_0$ such that, if $n \geq n_0$, then
$$|c(v)| \leq K \frac{E_n(v)}{|\Sigma_n(v)|},$$
where $K > 0$ is some universal constant.
\end{theorem} 

\begin{remark}
In contrast with the results in \cite{Graveja2}, which asserts that every finite energy solution to \eqref{TWc} with $c > \sqrt{2}$ is constant, the corresponding result is presumably not true on general tori.
\end{remark}

The proof of Theorem \ref{bornec} relies on Pohozaev's formula, which we specify next for solutions to equation \eqref{TWc} on the torus $\T_n^N$.

\begin{lemma}
\label{pohozaev}
Let $n \in \N^*$, and let $v$ be a solution to \eqref{TWc} on $\T_n^N$. We have, for any unfolding,
\begin{equation}
\label{eq:pohoici}
\begin{split}
& \frac{N - 2}{2} \int_{\Omega_n^N} |\nabla v|^2 + \frac{N}{4} \int_{\Omega_n^N} (1 - |v|^2)^2 - c(v) \frac{N - 1}{2} \int_{\Omega_n^N}
\langle Jv, \zeta_1 \rangle\\
& = \pi n \int_{\partial \Omega_n^N} \bigg( \frac{|\nabla
v|^2}{2} + \frac{(1 - |v|^2)^2}{4} \bigg) - \int_{\partial \Omega_n^N} \partial_\nu v \cdot \bigg( \sum_{j = 1}^N x_j \partial_j v \bigg),
\end{split}
\end{equation} 
where $\zeta_1$ is the $2$-form defined by \eqref{eq:zeta}. 
\end{lemma}

\begin{remark}
1. Notice that $\zeta_1$ is {\bf not} periodic and therefore \eqref{eq:pohoici} depends on the choice of the unfolding.\\
2. Identity \eqref{eq:pohoici} actually holds for any subdomain $\boU \subset \Omega_n^N$ replacing the integrals on $\Omega_n^N$ by integrals on $\boU$, and the boundary integrals by integrals on $\partial \boU$. In particular, if $0 < R < \pi n$, then $B(0, R) \subset \Omega_n^N$ and Pohozaev's identity yields the inequality
$$\bigg| \frac{N - 2}{2} \int_{B(0, R)} |\nabla v|^2 + \frac{N}{4} \int_{B(0, R)} (1 - |v|^2)^2 - c \frac{N - 1}{2} \int_{B(0, R)} \langle Jv, \zeta_1 \rangle \bigg| \leq R \int_{\partial B(0, R)} e(v).$$
\end{remark}

The starting point in order to prove Theorem \ref{bornec} and to bound $c(v)$ is formula \eqref{eq:pohoici}. The use of this formula requires to have an upper bound of the boundary terms on the r.h.s as well as a lower bound for the quantity
\begin{equation}
\label{eq:tolb}
\bigg| \int_{\Omega_n^N} \langle Jv, \zeta_1 \rangle \bigg|,
\end{equation}
which depends on the unfolding. We have already noticed that \eqref{eq:tolb} is related to the momentum $p_n(v)$ (they would actually even be equal if $v$ were constant on $\partial \Omega_n^N$). An appropriate choice of the unfolding allows to obtain the suitable bounds as the next proposition shows.

\begin{prop}
\label{lem:vbt} 
Assume $N = 2$ or $N = 3$, and let $E_0 > 0$ be given. Let $v$ be a non-trivial finite energy solution to \eqref{TWc} in $X_n^2 \cap \boS_n^0$, resp. $X_n^3$, with $E_n(v) \leq E_0$. Given any $\delta_0 > 0$, there exists a constant $n_0 \in \N$ depending only on $E_0$ and $\delta_0$, such that, if $n \geq n_0$, then there exists an unfolding of $\T_n^N$ such that
\begin{equation}
\label{ttb}
\bigg|p_n(v) - \frac{1}{2} \int_{\Omega_n^N} \langle Jv, \zeta_1 \rangle \bigg| \leq \frac{E_n(v)}{\sqrt{2}} + \delta_0,
\end{equation}
and
\begin{equation}
\label{ttbb}
n \int_{\partial \Omega_n^N} e(v) \leq 2 \int_{\T_n^N} e(v).
\end{equation}
\end{prop}

For the proof of Proposition \ref{lem:vbt}, we invoke several elementary lemmas. The first one is used throughout the paper.

\begin{lemma}
\label{degiorgi}
Let $I$ be an interval of $\R$, such that $|I| \geq 1$. Given any $\delta > 0$, there exists a constant $\mu_0(\delta) > 0$, such that if $u \in H^1(\R, \C)$ satisfies
\begin{equation}
\label{bb}
\int_I e(u) \leq \mu_0(\delta),
\end{equation}
then
\begin{equation}
\label{bbbb}
\Big| 1 - |u| \Big| \leq \delta \on I. 
\end{equation}
\end{lemma}

\begin{proof}
By the energy bound, we have
\begin{equation}
\label{aa}
\frac{1}{2} \int_I |\nabla u|^2 + \frac{1}{4} \int_I (1 - |u|^2)^2 \leq \mu_0.
\end{equation}
Hence, from the inequality $ab \leq \frac{1}{2} (a^2 + b^2)$, it holds
$$\int_I |\nabla u| |1 - |u|^2| \leq \frac{1}{\sqrt{2}} \int_I |\nabla u|^2 + \frac{\sqrt{2}}{4} \int_I (1 - |u|^2)^2 \leq \sqrt{2} \mu_0,$$
that is
$$\int_I |\nabla \xi(|u|)| \leq \sqrt{2} \mu_0,$$
where the function $\xi$ is defined by $\xi(t) = t - \frac{t^3}{3}$. In particular, $\xi$ has a strict local maximum at $t = 1$. Going back to \eqref{aa}, the mean-value inequality and the fact that $|I| \geq 1$ yields the existence of some point $x_0 \in I$ such that
$$|1 - |u(x_0)|^2| \leq 2 \sqrt{\mu_0}.$$
Combining both the previous inequalities, we deduce that
$$\sup_{x \in I} \Big| \xi(|u(x)|) - \xi(1) \Big| \leq \int_I |\nabla \xi(|u|)| + \frac{1}{3} \Big| 1 - |u(x_0)| \Big| \Big| 2 - |u(x_0)| - |u(x_0)|^2 \Big| \leq \Big( \frac{8}{3} + \sqrt{2} \Big) \mu_0,$$
from which \eqref{bbbb} follows invoking the coercivity of $\xi$ at $t = 1$. 
\end{proof}
 
\begin{lemma}
\label{degiorgibis}
Let $u$ be in $H^1([a,b])$ with $|a - b| \geq 1$, such that $u=\varrho \exp i\varphi$, with $\varphi(a)=\varphi(b)$. Assume moreover that for some $0 \leq \delta < \frac{1}{2}$, $u$ satisfies \eqref{bb}. Then, we have
$$\bigg| \int_I \langle i \dot{u}, u \rangle \bigg| \leq \frac{\sqrt{2}}{1 - \delta} \int_I e(u).$$
\end{lemma}

\begin{proof}
Since by assumptions, $u = \rho \exp i \varphi$ on $I$, we have 
$$\langle i \dot{u}, u \rangle = - \rho^2 \dot{\varphi}, \ {\rm and} \ |\dot{u}|^2 = \rho^2 \dot{\varphi}^2 + \dot{\rho}^2.$$
We next compute
$$\bigg| \int_I \langle i \dot{u}, u \rangle \bigg| = \bigg| \int_I \rho^2 \dot{\varphi} \bigg| = \bigg|\int_I (\rho^2 - 1) \dot{\varphi} \bigg| \leq \frac{\sqrt{2}}{1 - \delta} \int_I e(u),$$
where we used the results of Lemma \ref{colisee} and Lemma \ref{degiorgi} for the last inequality.
\end{proof} 

In dimension two, a related result is

\begin{lemma}
\label{navona}
Let $0 < \delta < \frac{1}{2}$ be given. There exists some constant $\mu_1(\delta)>0$, such that, for any map $u \in H^1(\T_n^2)$ which satisfies
\begin{equation}
\label{navona1}
\int_{\T_n^2} e(u) \leq \mu_1(\delta),
\end{equation}
we have the estimate
$$\bigg| \int_{\T_n^2} \langle i \partial_j u, u \rangle \big| \leq \frac{\sqrt{2}}{1 - \delta} \int_{\T_n^2} e(u), \forall j \in \{ 1, 2 \}.$$
\end{lemma}

\begin{proof}
If we knew that $|u| \leq 1 - \delta$, then the proof would essentially follow the same arguments as the proof of Lemma \ref{degiorgibis}. However, in contrast with the one-dimensional case, smallness of the energy does not allow to draw that conclusion. To overcome the difficulty, we introduce an approximation of $u$ for which the proof of Lemma \ref{degiorgibis} applies. Indeed, for $\lambda > 1$ given to be determined later, we consider a map $u_\lambda \in H^1(\T_n^2, \C)$ solution to the minimization problem
$$F_\lambda(u_\lambda) = \inf \{F_\lambda(v), v \in H^1(\T_n^2, \C) \},$$
where
$$F_\lambda(v) = \frac{\lambda}{2} \int_{\T_n^2} |u - v|^2 + \int_{\T_n^2} e(v).$$
Existence of $u_\lambda$ is straightforward. By minimality of $u_\lambda$, we have
\begin{equation}
\label{eq:bonzai}
\frac{\lambda}{2} \int_{\T_n^2} |u - u_\lambda|^2 + \int_{\T_n^2} e(u_\lambda) \leq \int_{\T_n^2} e(u),
\end{equation}
and the Euler-Lagrange equation writes
$$- \Delta u_\lambda = \lambda (u - u_\lambda) + u_\lambda (1 - |u_\lambda|^2) \on \T_n^2.$$
In particular $u_\lambda$ is smooth. We compute the difference
$$\langle i \partial_j u_\lambda, u_\lambda \rangle - \langle i \partial_j u, u \rangle = \langle i \partial_j u_\lambda, u_\lambda - u \rangle + \langle i \partial_j (u_\lambda - u), u \rangle,$$
so that integrating by parts, we are led to the identity
$$\int_{\T_n^2} \Big( \langle i \partial_j u_\lambda, u_\lambda \rangle- \langle i \partial_j u, u \rangle \Big) = \int_{\T_n^2} \Big( \langle i \partial_j u_\lambda, u_\lambda - u \rangle + \langle i (u - u_\lambda), \partial_j u \rangle \Big),$$
and hence
\begin{equation}
\label{eq:glou}
\bigg| \int_{\T_n^2} \Big( \langle i \partial_j u_\lambda, u_\lambda \rangle - \langle i \partial_j u, u \rangle \Big) \bigg| \leq \| u - u_\lambda \|_{L^2(\T_n^2)} \Big( \| \nabla u \|_{L^2(\T_n^2)} + \| \nabla u_\lambda \|_{L^2(\T_n^2)} \Big) \leq \frac{4}{\sqrt{\lambda}} \int_{\T_n^2} e(u),
\end{equation}
where we have used \eqref{eq:bonzai} for the last inequality. We choose therefore the value of the parameter $\lambda = \lambda(\delta)$ so that
\begin{equation}
\label{eq:flotte}
\frac{1}{\sqrt{\lambda(\delta)}} = \frac{1}{2 \sqrt{2}} \bigg( \frac{1}{1 - \delta} - \frac{1}{1 - \frac{\delta}{2}} \bigg).
\end{equation}
For this choice of $\lambda$, we claim that there exists a constant $\mu_1(\delta) > 0$, such that, if \eqref{navona1} is satisfied, then
\begin{equation}
\label{eq:trois}
|u_\lambda| \geq 1 - \frac{\delta}{2} \on \T_n^2.
\end{equation}
We postpone the proof of Claim \eqref{eq:trois}, and complete the proof of Lemma \ref{navona}. Indeed, invoking Lemma \ref{massimo2}, and using \eqref{navona1} and \eqref{eq:bonzai}, with $\mu_1(\delta)$ sufficiently small, we can assume that $u_\lambda$ is written as $u_\lambda = \varrho \exp i \varphi$ on $\T_n^2$, with $\varphi \in H^1(T_n^2)$. In view of Claim \eqref{eq:trois}, the same argument as in Lemma \ref{degiorgibis} then shows that
$$\bigg| \int_{\T_n^2} \langle i \partial_j u_\lambda, u_\lambda \rangle \bigg| \leq \frac{\sqrt{2}}{1 - \frac{\delta}{2}} \int_{\T_n^2} e(u),$$
so that the conclusion follows from \eqref{eq:glou} and the choice \eqref{eq:flotte} of $\lambda(\delta)$.
\end{proof}

\begin{proof}[Proof of Claim \eqref{eq:trois}]
We write
$$\big| u_\lambda (1 - |u_\lambda|^2) \big| \leq 2 \Big( \big| 1 - |u_\lambda|^2 \big| \1_{\{ |u_\lambda| \leq 2 \}} + \big( |u_\lambda|^2 - 1 \big)^\frac{3}{2} \1_{\{ |u_\lambda| \geq 2 \}} \Big),$$
so that if $u$ satisfies \eqref{navona1} with $0 \leq \mu_1(\delta) \leq 1$, then
$$\| u_\lambda (1 - |u_\lambda|^2) \|_{L^2 + L^\frac{4}{3}(\T_n^2)} \leq 2 \Big( \| 1 - |u_\lambda|^2 \|_{L^2(\T_n^2)} + \| 1 - |u_\lambda|^2 \|_{L^2(\T_n^2)}^\frac{3}{2} \Big) \leq 10 \bigg( \int_{\T_n^2} e(u) \bigg)^\frac{1}{2}.$$
By \eqref{eq:flotte}, we have
$$\| \lambda (u - u_\lambda) \|_{L^2(\T_n^2)} \leq \lambda \delta \bigg( \int_{\T_n^2} e(u) \bigg)^\frac{1}{2},$$
so that
$$\| \Delta u_\lambda \|_{L^2 + L^\frac{4}{3}(\T_n^2)} \leq
10 (\lambda \delta +1) \bigg( \int_{\T_n^2} e(u) \bigg)^\frac{1}{2}.$$
By \eqref{eq:bonzai}, we are led to
$$\| \nabla u_\lambda \|_{H^1 + W^{1, \frac{4}{3}}(\T_n^2)} \leq
K (\lambda \delta + 1) \bigg( \int_{\T_n^2} e(u) \bigg)^\frac{1}{2},$$
where $K$ is some universal constant, so that, by Sobolev's embedding theorem,
$$\| \nabla u_\lambda \|_{L^4(B(x, 1))} \leq K (\lambda \delta + 1) \bigg( \int_{\T_n^2} e(u) \bigg)^\frac{1}{2},$$
for any point $x \in \T_n^2$. It follows therefore from Morrey's embedding theorem that we have
$$\big| u_\lambda(x) - u_\lambda(y) \big| \leq K (\lambda \delta + 1) \bigg( \int_{\T_n^2} e(u) \bigg)^\frac{1}{2} |x - y|^\frac{1}{2} \leq K (\lambda \delta + 1) \mu_1(\delta)^\frac{1}{2} |x - y|^\frac{1}{2},$$
for any $|x - y| \leq 1$. To conclude, assume by contradiction that there is a point $x_0$ such that $|u_\lambda(x_0)| \leq 1 - \frac{\delta}{2}$. Then, we have $|u_\lambda(x)| \leq 1 - \frac{\delta}{4}$ for any $x \in B(x_0, r_0)$, where the radius $r_0$ is given by $r_0 = \frac{\delta^2}{16 K^2 (\lambda \delta +1)^2 \mu_1(\delta)}$, so that integrating, we obtain
$$\int_{B(x_0, r_0)} (1 - |u_\lambda|^2)^2 \geq 
\frac{\pi r_0^2 \delta^2}{16} = \frac{\pi \delta^6}{(16)^3 K^4 (\lambda \delta + 1)^4 \mu_1(\delta)^2},$$
which implies using \eqref{navona1},
$$\mu_1(\delta)^3 \geq K \delta^{10},$$
and leads to a contradiction if the number $\mu_1(\delta)$ is chosen sufficiently small.
\end{proof}

We are now in position to give the proof of Proposition \ref{lem:vbt}.

\begin{proof}[Proof of Proposition \ref{lem:vbt}]
The starting point is formula \eqref{integre2}, the main point being to estimate the boundary term. Since the computations depend on the dimensions, we distinguish two cases $N = 2$ and $N = 3$. 

\setcounter{case}{0}
\begin{case}
\label{2d-case}
${\bf N = 2}$. We have $\zeta_1 = - 2 x_2 dx_1 \wedge dx_2$, so that 
$\star \zeta_1 = - 2 x_2$. Inserting this identity into \eqref{integre2}, we obtain
\begin{equation}
\label{eq:cor3}
\int_{\T_n^2} \langle i \partial_1 v, v \rangle - \int_{\Omega_n^N} \langle Jv, \zeta_1 \rangle = n \pi \int_{[-\pi n, \pi n] \times \{- \pi n, \pi n \}} \langle i \partial_ 1 v, v \rangle.
\end{equation}
for any unfolding of the torus. Let $\delta > 0$ be fixed, to be determined later, and consider the subset $A$ of $[- \pi n, \pi n]$ defined by $\alpha \in A $ if and only if
$$\int_{[- \pi n, \pi n] \times \{ \alpha \} } e(v) \geq \frac{\mu_0(\delta)}{2},$$
where $\mu_0(\delta)$ is the constant provided by Lemma \ref{degiorgi}. Notice in particular that by integration
\begin{equation}
\label{regan}
|A| \leq \frac{2}{\mu_0(\delta)} E_n(v) \leq \frac{2 E_0}{\mu_0(\delta)} \leq 2 \pi n^\frac{3}{4},
\end{equation}
as soon as $n \geq \Big( \frac{E_0}{\pi \mu_0(\delta)} \Big)^\frac{4}{3}$. Using Lemma \ref{degiorgi}, it follows that
$$|v| \geq 1 - \delta \ {\rm on} \ [- \pi n, \pi n] \times \{ \alpha \},$$
for any $\alpha \in [- \pi n, \pi n] \setminus A$. Since $v \in \boS_n^0$, the topological degree of $\frac{v}{|v|}$ is equal to $0$ on $[- \pi n, \pi n] \times \{ \alpha \}$, for any $\alpha$ in a subset $B$ of $[- \pi n, \pi n] \setminus A$ with $|B| \geq 2\pi (n - 2 n^\frac{3}{4})$ by \eqref{regan}, so that $v$ may be written as
\begin{equation}
\label{wilkinson}
v = |v| \exp i \phi \ {\rm on} \ [- \pi n, \pi n] \times \{ \alpha \},
\end{equation}
for any $\alpha \in B$. We next apply Lemma \ref{average0} to the function $f \equiv e(v)$ and the complementary $B^c$ of the set $B$. This yields an unfolding of the torus such that $\pm \pi n \notin A$,
$$\int_{[- \pi n, \pi n] \times \{- \pi n, \pi n\}} e(v) \leq \mu_0(\delta), \ \int_{[- \pi n, \pi n] \times \{- \pi n, \pi n\}} e(v) \leq \frac{1}{\pi \big( n - 2 n^\frac{3}{4} \big)} E_n(v),$$
and
\begin{equation}
\label{creche}
\int_{\partial \Omega_n^N} e(v) \leq \bigg( \frac{1}{\pi \big( n - 2 n^\frac{3}{4} \big)} + \frac{1}{\pi n} \bigg) E_n(v).
\end{equation}
Invoking \eqref{wilkinson} to apply Lemma \ref{degiorgibis} to this choice of unfolding, we have therefore
\begin{equation}
\label{creche1}
\bigg| n \pi \int_{[- \pi n, \pi n] \times \{- \pi n, \pi n\}} \langle i \partial_1 v, v \rangle \bigg| \leq \frac{n \sqrt{2} E_n(v)}{\big( n - 2 n^\frac{3}{4} \big)(1 - \delta)}.
\end{equation}
We first fix $\delta < \frac{1}{2}$ so that 
$$\frac{2 \delta - \delta^2}{(1 - \delta)^2} E_0 \leq \sqrt{2} \delta_0.$$
Equation \eqref{creche1} becomes
\begin{equation}
\label{creche2}
\bigg| n \pi \int_{[- \pi n, \pi n] \times \{- \pi n, \pi n\}} \langle i \partial_1 v, v \rangle \bigg| \leq \frac{\sqrt{2}}{(1 - \delta)^2} E_n(v) \leq \sqrt{2} E_n(v) + 2 \delta_0,
\end{equation}
and equation \eqref{creche} gives
\begin{equation}
\label{creche3}
n \int_{\partial \Omega_n^N} e(v) \leq \frac{2 - \delta}{\pi (1 - \delta)} E_n(v) \leq 2 E_n(v),
\end{equation}
for any $n \geq \frac{16}{\delta^4}$ The conclusion then follows, choosing $n_0 \geq \max \Big\{ \Big( \frac{E_0}{\pi \mu_0(\delta)} \Big)^\frac{4}{3},\frac{16}{\delta^4} \Big\}$, and combining \eqref{eq:cor3}, \eqref{creche2} and \eqref{creche3}.
\end{case}

\begin{case}
${\bf N = 3}$. In this case, the $2$-form $\zeta_1$ is written as $\zeta_1 = - x_2 dx_1 \wedge dx_2 - x_3 dx_1 \wedge dx_3$, so that 
$\star \zeta_1 = - x_2 dx_3 + x_3 dx_2$, and
$$- \frac{1}{2} \int_{\partial \Omega_n^N} (u \times du)_\top \wedge (\star \zeta_1)_\top = n \pi \sum_{i = 2}^3 \int_{C_i} \langle i \partial_1 u, u \rangle,$$
where, for any $i \in \{ 2, 3 \}$, the set $C_i$ is the union of two squares, namely $C_2 = [-\pi n, \pi n] \times \{- \pi n, \pi n \} \times [-\pi n, \pi n]$, and $C_3 = [-\pi n, \pi n]^2 \times \{- \pi n, \pi n \}$. The end of the argument is then essentially the same as in Case \ref{2d-case}, replacing Lemmas \ref{degiorgi} and \ref{degiorgibis} by Lemma \ref{navona}, with the exception of one major difference. Indeed, the proof of Lemma \ref{navona} does not require to have a lifting of $v$, so that we can directly apply Lemma \ref{average0} to the set $A$ (with $\mu_0(\delta)$ replaced by the constant $\mu_1(\delta)$ provided by Lemma \ref{navona}) in the three-dimensional case.
\end{case}
\end{proof}

\begin{proof}[Proof of Theorem \ref{bornec}]
We may assume without loss of generality that $p_n(v) > 0$. Let $\delta_0 > 0$ be given, and consider the unfolding provided by Proposition \ref{lem:vbt} for any $n \geq n_0$. It follows from \eqref{ttb} that
$$\Sigma_n(v) \leq \frac{1}{\sqrt{2}} \int_{\Omega_n^N} \langle Jv, \zeta_1 \rangle + \sqrt{2} \delta_0.$$
Choosing $\delta_0$ so that $\sqrt{2} \delta_0 < \frac{\Sigma_0}{2}$, we are led to the inequality
\begin{equation}
\label{minorsigma2}
\Sigma_n(v) \leq \sqrt{2} \int_{\Omega_n^N} \langle Jv, \zeta_1 \rangle,
\end{equation}
for any $n \geq n_0$. On the other hand, we may invoke Lemma \ref{pohozaev} to assert that
$$\bigg| \frac{N - 2}{2} \int_{\Omega_n^N} |\nabla v|^2 + \frac{N}{4} \int_{\Omega_n^N} (1 - |v|^2)^2 - c(v) \frac{N - 1}{2} \int_{\Omega_n^N} \langle Jv, \zeta_1 \rangle \bigg| \leq \pi n \int_{\partial \Omega_n^N} e(v),$$
which yields, combined with \eqref{ttbb} and provided $n \geq n_0$,
\begin{equation}
\label{sondage}
|c(v)| \int_{\Omega_n^N} \langle Jv, \zeta_1 \rangle \leq K E_n(v), 
\end{equation}
for some universal constant $K > 0$. Combining \eqref{minorsigma2} and \eqref{sondage}, we deduce
$$|c(v)| \Sigma_n(v) \leq K E_n(v),$$
which yields the desired conclusion.
\end{proof}

\subsection{Concentration of energy}

The results in this section are a first step towards the proof of Proposition \ref{concentrationcompacite}. By a standard covering argument, we first deduce from Lemma \ref{tarquini2}.

\begin{lemma}
\label{covering0}
Let $c_0 > 0$ and $E_0 > 0$ be given. Let $v$ be a finite energy solution to \eqref{TWc} on $\T_n^N$ such that $|c(v)| \leq c_0$ and $E_n(v) \leq E_0$. Given any $\delta > 0$, there exists a number $\ell_0 \in \N$ depending only on $c_0$, $E_0$ and $\delta$, such that there exists a finite number $\ell(v) \leq \ell_0$ of points $x_1$, $\ldots$, $x_{\ell(v)}$ in $\T_n^N$ which satisfy
\begin{equation}
\label{petit}
\Big| 1 - |v| \Big| \leq \delta \on \T_n^N \setminus \underset{i = 1}{\overset{\ell(v)}{\cup}} B(x_i, 1),
\end{equation}
and, for any $i \in \{ 1, \ldots, \ell(v) \}$,
\begin{equation}
\label{grand}
\Big| 1 - |v(x_i)| \Big| \geq \delta.
\end{equation}
\end{lemma}

\begin{proof}
It follows from Lemma \ref{tarquini1} that $v$ is continuous on $\T_n^N$, so that the set $\boV_\delta$ defined by
$$\boV_\delta = \{ x \in \T_n^N, \ {\rm s.t.} \ |1 - |v|| \geq \delta \},$$
is compact, and is therefore included in a finite collection of balls $B(x_i, \frac{1}{5})$ with $x_i \in \boV_\delta$. Using Vitali's covering lemma, there exists a finite subcollection of balls $(B(x_i, \frac{1}{5}))_{1 \leq i \leq \ell(v)}$ so that
$$\boV_\delta \subset \underset{i = 1}{\overset{\ell(v)}{\cup}} B(x_i, 1),$$
and
\begin{equation}
\label{cecilia}
B \Big( x_i, \frac{1}{5} \Big) \cap B \Big( x_j, \frac{1}{5} \Big) = \emptyset, \forall 1 \leq i \neq j \leq \ell(v).
\end{equation}
In particular, conclusions \eqref{petit} and \eqref{grand} hold for this subcollection. On the other hand, we deduce from Lemma \ref{tarquini2} that
$$E_n \Big( v, B \Big( x_i, \frac{1}{5} \Big) \Big) \geq K(N, c_0, \delta),$$
where $K(N, c_0, \delta)$ is some positive constant depending on $N$, $c_0$ and $\delta$, so that invoking \eqref{cecilia},
$$E_0 \geq \sum_{i = 1}^{\ell(v)} E_n \Big( v, B \Big( x_i, \frac{1}{5} \Big) \Big) \geq K(N, c_0, \delta) \ell(v).$$
Hence, there exists some integer $\ell_0 \in \N$ depending only on $c_0$, $E_0$ and $\delta$ so that $\ell(v) \leq \ell_0$, which completes the proof of Lemma \ref{covering0}.
\end{proof}

Considering clusters of the balls $B(x_i, 1)$, and enlarging possibly the radius, we may assume that their mutual distance is even larger in view of the following abstract but elementary lemma.

\begin{lemma}
\label{clustering}
Let $X$ be a metric space, and consider $\ell$ distinct points $x_1$, $\ldots$, $x_\ell$ in $X$. Let $\mu_0 > 0$ and $0 < \kappa \leq \frac{1}{2}$ be given. Then, there exists $\mu > 0$ such that
$$\mu_0 \leq \mu \leq \Big( \frac{2}{\kappa} \Big)^\ell \mu_0,$$
and a subset $\{ x_j \}_{j \in J}$ of $\{ x_i \}_{1 \leq i \leq \ell}$ such that
\begin{equation}
\label{eq:A13}
\underset{i = 1}{\overset{\ell}{\cup}} B(x_i, \mu_0) \subset \underset{j \in J}{\cup} B(x_j, \mu),
\end{equation}
and
\begin{equation}
\label{eq:A13bis}
\dist(x_j, x_k) \geq \frac{\mu}{\kappa}, \forall j \neq k \in J.
\end{equation}
\end{lemma}

\begin{proof}
The proof is by iteration, in at most $\ell$ steps. First,
consider the collection $\{ x_i \}_{1 \leq i \leq \ell}$. If
\eqref{eq:A13} and \eqref{eq:A13bis} are verified with $\mu = \mu_0$, there is nothing else to do. Otherwise, take two points, say $x_1$ and $x_2$ such that $\dist(x_1, x_2) \leq \kappa^{-1} \mu_0$, consider the collection $\{ x_2, \ldots, x_\ell \}$, and set $\mu = 2 \kappa^{-1} \mu_0$. If \eqref{eq:A13} is verified, we stop. Otherwise, we go on in the same way. If the process does not stop in $\ell - 1$ steps, at the $\ell^{\rm th}$ step, we are left with one single ball of radius $\mu = 2^\ell \kappa^{- \ell} \mu_0$, and \eqref{eq:A13bis} is void.
\end{proof}

We may specify Lemma \ref{clustering} to the points $x_1$, $\ldots$, $x_\ell(v)$, provided by Lemma \ref{covering0}. It follows that there exists some $1 \leq \mu \leq \Big( \frac{2}{\kappa} \Big)^{\ell(v)}$, and a subset $J$ of $\{1, \ldots, \ell(v) \}$ such that
\begin{equation}
\label{vin}
\underset{i = 1}{\overset{\ell(v)}{\cup}} B(x_i, 1) \subset \underset{j \in J}{\cup} B(x_j, \mu), \ {\rm and} \ |x_j - x_k| \geq \frac{\mu}{\kappa}, \forall j \neq k \in J.
\end{equation}
In particular,
\begin{equation}
\label{margaux}
|v| \geq 1 - \delta \geq \frac{1}{2} \on \boO_n(\mu) \equiv \T_n^N \setminus \underset{j \in J}{\cup} B(x_j, \mu),
\end{equation}
and for any $j \in J$, $\Big| 1 - |v(x_j)| \Big| \geq \delta$.

The main result in this subsection is

\begin{prop}
\label{concentration}
Let $E_0$, $c_0$, $\delta$ and $v$ be as in Lemma \ref{covering0}, and assume the points $x_1$, $\ldots$, $x_\ell(v)$, and the set $J$ are such that \eqref{vin} and \eqref{margaux} are satisfied.\\
i) Let $0 < \kappa < \frac{1}{64}$. There exists some radius $\mu \leq R \leq \frac{\mu}{2 \kappa}$ such that
\begin{equation}
\label{couronne}
E_n \Big(v, \underset{j \in J}{\cup} \big( B(x_j, 2R) \setminus B(x_j, R) \big) \Big) \leq 2 \frac{E_n(v)}{|\ln(\kappa)|}.
\end{equation}
ii) If one may write $v = \varrho \exp i \varphi$ on $\boO_n(\mu)$, where $\varphi$ is a smooth real-valued function, then there exists a constant $K(c_0)$ possibly depending on $c_0$ such that we have the estimate
\begin{equation}
\label{dentifrice}
\bigg| \int_{\boU_n(2 R)} \bigg( \frac{c}{2} (1 - \varrho^2) \partial_1 \varphi - e(v) \bigg) \bigg| \leq K(c_0) \bigg( \delta \int_{\boU_n(2 R)} e(v) + \frac{E_n(v)}{|\ln(\kappa)|} \bigg),
\end{equation}
where $\boU_n(2 R) \equiv \T_n^N \setminus \underset{j \in J}{\cup} B(x_j, 2 R)$.\\
iii) If moreover $c_0 < \sqrt{2}$, and $0 < \delta \leq \delta(c_0) \equiv \min \Big\{ \frac{\sqrt{2} - c_0}{2 \sqrt{2} (K(c_0) + 1)}, \frac{1}{2} \Big\}$, then
\begin{equation}
\label{azerty}
\int_{\boU_n(2 R)} e(v) \leq \frac{4 K(c_0) E_n(v)}{(2 - c_0^2)|\ln(\kappa)|}.
\end{equation}
\end{prop}

\begin{proof}
For statement i), we set $R_k = 2^k \mu$ for any $k \geq 1$, so that $\mu \leq R_k \leq \frac{\mu}{2 \kappa}$, provided $1 \leq k \leq k_0 - 1$, where $k_0$ is the largest integer less or equal to $\frac{|\ln(\kappa)|}{\ln(2)}$. Since the sets $B(x_j, R_{k + 1}) \setminus B(x_j, R_k)$ are disjoints, we have by summation
$$\sum_{k = 1}^{k_0 - 2} E_n \Big(v, \underset{j \in J}{\cup} \big( B(x_j, R_{k + 1}) \setminus B(x_j, R_k) \big) \Big) \leq E_n(v).$$
In particular, there exists some integer $k_1 \in \{ 1, \ldots, k_0 - 2 \}$ such that
$$E_n \Big(v, \underset{j \in J}{\cup} \big( B(x_j, R_{k_1 + 1}) \setminus B(x_j, R_{k_1}) \big) \Big) \leq \frac{E_n(v)}{k_0 - 2} \leq \frac{2 \ln(2)}{|\ln(\kappa)|} E_n(v).$$
Conclusion \eqref{couronne} follows from choosing $R = R_{k_1}$.\\
We divide the proof of statement ii) into three steps.

\setcounter{step}{0}
\begin{step}
\label{step3.2}
There exists some constant $K(c_0)$ possibly depending on $c_0$ such that
\begin{equation}
\label{beta}
\bigg| \frac{c}{2} \int_{\boU_n(2 R)} (1 - \varrho^2) \partial_1 \varphi - \int_{\boU_n(2 R)} \varrho^2 |\nabla \varphi|^2 \bigg| \leq K(c_0) \frac{E_n(v)}{|\ln(\kappa)|}.
\end{equation}
\end{step}

Set, for any $j \in J$,
$$\varphi_j = \frac{1}{|\partial B(x_j, R)|} \int_ {\partial B(x_j, R)} \varphi,$$
and consider a smooth non-negative function $\chi$ defined on $\R_+$ such that $\chi(t) = 1$, for any $0 \leq t \leq 1$, and $\chi(t) = 0$, for any $t \geq 2$. We consider the map $\tilde{\varphi}$ defined by
$$\tilde{\varphi}(x) = \chi \bigg( \frac{|x - x_j|}{R} \bigg) \varphi_j + \bigg(1 - \chi \bigg( \frac{|x - x_j|}{R} \bigg) \bigg) \varphi(x), \forall x \in B(x_j, 2 R),$$
and
$$\tilde{\varphi}(x) = \varphi (x), \ {\rm elsewhere}.$$
Notice in particular that $\tilde{\varphi} = \varphi_j$ on the ball $B(x_j, R)$, and that an elementary computation based on Poincar\'e-Wirtinger's inequality shows that
\begin{equation}
\label{prolongement}
\begin{split}
\int_{B(x_j, 2 R)} |\nabla \tilde{\varphi}|^2 & \leq K \int_{B(x_j, 2 R) \setminus B(x_j, R)} \bigg( \frac{|\varphi - \varphi_j|^2}{R^2} + |\nabla \varphi|^2 \bigg)\\
& \leq K \int_{B(x_j, 2 R) \setminus B(x_j, R)} |\nabla \varphi|^2 \leq K \frac{E_n(v)}{|\ln(\kappa)|},
\end{split}
\end{equation}
where $K$ is some universal constant. We now multiply the first equation in \eqref{PolTWc} by $\tilde{\varphi}$ on $\T_n^N$ and integrate by parts. By \eqref{margaux}, this leads to
\begin{equation}
\label{terminator}
\begin{split}
\bigg| \frac{c}{2} \int_{\boU_n(2 R)} (\varrho^2 - 1) \partial_1 \varphi - \int_{\boU_n(2 R)} \varrho^2 |\nabla \varphi|^2 \bigg| & = \bigg| - \frac{c}{2} \int_{\boC(R)} (\varrho^2 - 1) \partial_1 \tilde{\varphi} + \int_{\boC(R)} \varrho^2 \nabla \varphi \nabla \tilde{\varphi} \bigg|\\
& \leq K(c_0) \bigg( \int_{\boC(R)} e(v) \bigg)^\frac{1}{2} \bigg( \int_{\boC(R)} |\nabla \tilde{\varphi}|^2 \bigg)^\frac{1}{2},
\end{split}
\end{equation}
where we have set $\boC(R) \equiv \underset{j \in J}{\cup} \big( B(x_j, 2R) \setminus B(x_j, R) \big)$. Invoking \eqref{couronne} and \eqref{prolongement}, the l.h.s of \eqref{terminator} can be estimated so to obtain inequality \eqref{beta}.

\begin{step}
There exists some constant $K(c_0)$ possibly depending on $c_0$ such that
\begin{equation}
\label{alpha2bis}
\bigg| c \int_{\boU_n(2 R)} \varrho (1 - \varrho^2) \partial_1 \varphi - \int_{\boU_n(2 R)} \bigg( 2 \varrho |\nabla \varrho|^2 + \varrho (1 - \varrho^2)^2 \bigg) \bigg| \leq K(c_0) \bigg( \delta \int_{\boU_n(R)} e(v) + \frac{E_n(v)}{|\ln(\kappa)|} \bigg).
\end{equation}
\end{step}

The argument is somewhat similar. We consider the function $\tilde{\varrho}$ defined on $\T_n^N$ by
$$\tilde{\varrho}(x) = \chi \bigg( \frac{|x - x_j|}{R} \bigg) \varrho_j + \bigg(1 - \chi \bigg( \frac{|x - x_j|}{R} \bigg) \bigg) \varrho(x), \forall x \in B(x_j, 2 R),$$
and
$$\tilde{\varrho}(x) = \varrho (x), \ {\rm elsewhere},$$
where $\varrho_j = \frac{1}{|\partial B(x_j, R)|} \int_ {\partial B(x_j, R)} \varrho$, and $\chi$ is as in Step \ref{step3.2}. Notice that $\tilde{\varrho} = \varrho_j$ on the balls $B(x_j, R)$, and as for \eqref{prolongement}, we have
\begin{equation}
\label{prolongement2}
\int_{B(x_j, 2R)} |\nabla \tilde{\varrho}|^2 
\leq K \frac{E_n(v)}{|\ln(\kappa)|}.
\end{equation}
Moreover, since $|1 - \varrho| \leq \delta $ outside $\underset{j \in J}{\cup} B(x_j, R)$, it follows that
\begin{equation}
\label{unifo}
|1 - \tilde{\varrho}| \leq \delta \on \T_n^N.
\end{equation}
We now multiply the second equation in \eqref{PolTWc} by $\tilde{\varrho}^2 - 1$, and integrate by parts on $\T_n^N$. This leads to
\begin{align*}
\int_{\boU_n(2 R)} \bigg( 2 \varrho |\nabla \varrho|^2 + \varrho (1 - \varrho^2)^2 - & c \varrho (1 - \varrho^2) \partial_1 \varphi \bigg) = \int_{\T_n^N} \varrho (1 - \tilde{\varrho}^2) |\nabla \varphi|^2\\
& - \int_{\boC(R)} \bigg( \nabla \varrho \nabla \tilde{\varrho}^2 + \varrho (1 - \varrho^2) (1 - \tilde{\varrho}^2) - c \varrho (1 - \tilde{\varrho}^2) \partial_1 \varphi \bigg).
\end{align*}
Bounding the r.h.s of this equality using \eqref{margaux}, \eqref{prolongement2} and \eqref{unifo}, we derive inequality \eqref{alpha2bis}.

\begin{step}
Proof of estimate \eqref{dentifrice} completed.
\end{step}

Since $|1 - \varrho| \leq \delta$ on $\boU(2 R)$, we first observe that
\begin{equation}
\label{centaure1}
\bigg| \int_{\boU_n(2 R)} \bigg( \varrho (1 - \varrho^2) \partial_1 \varphi - (1 - \varrho^2) \partial_1 \varphi \bigg) \bigg| \leq K \delta \int_{\boU_n(2 R)} e(v) ,
\end{equation}
and similarly,
\begin{equation}
\label{centaure2}
\bigg| \int_{\boU_n(2 R)} \bigg( 2 \varrho |\nabla \varrho|^2 + \varrho (1 - \varrho^2)^2 - 2 |\nabla \varrho|^2 - (1 - \varrho^2)^2 \bigg) \bigg| \leq K \delta \int_{\boU_n(2 R)} e(v),
\end{equation}
where $K$ is some universal constant. Combining \eqref{couronne}, \eqref{beta}, \eqref{alpha2bis}, \eqref{centaure1} and \eqref{centaure2}, we deduce conclusion \eqref{dentifrice}.

We finally turn to the proof of statement iii). In view of Lemma \ref{colisee}, we have
$$\bigg| \int_{\boU_n(2 R)} (1 - \varrho^2) \partial_1 \varphi \bigg| \leq \frac{\sqrt{2}}{1 - \delta} \int_{\boU_n(2 R)} e(v),$$
so that \eqref{dentifrice} leads to
$$\Big( 1 - \frac{c}{\sqrt{2} (1 - \delta)} - K(c_0) \delta \Big) \int_{\boU_n(2 R)} e(v) \leq K(c_0) \frac{E_n(v)}{|\ln(\kappa)|},$$
and conclusion \eqref{azerty} follows.
\end{proof}

\subsection{Subsonic vortexless solutions}

We next specify a little further the analysis assuming that the solution verifies the additional condition
\begin{equation}
\label{subsoniccasen}
0 < c(v) < \sqrt{2}.
\end{equation}
We set, for such a solution, $\varepsilon(v) = \sqrt{2 - c(v)^2}$. As on the whole space $\R^N$, we have

\begin{prop}
\label{minoration1}
Given $E_0 > 0$, let $v$ be a non-trivial finite energy solution to \eqref{TWc} in $X_n^2 \cap \boS_n^0$, resp. $X_n^3$, satisfying \eqref{subsoniccasen}, and $E_n(v) \leq E_0$. Then there exists some integer $n_0$ only depending on $E_0$ such that
$$\Big\| 1 - |v| \Big\|_{L^\infty(\T_n^N)} \geq \frac{\varepsilon(v)^2}{10}, \forall n \geq n_0.$$
\end{prop}

\begin{proof}
The proof is similar to the proof of Proposition \ref{minoration}, invoking Lemmas \ref{massimo3} and \ref{sectorisation} to construct a lifting of $v$. Therefore, we omit it.
\end{proof}

\section{Asymptotics of solutions on expanding tori}
\label{Torusinfini}

In this section, we consider the asymptotics $n \to + \infty$ for a sequence $(v^n)_{n\in \N^*}$ of solutions of \eqref{TWc} on tori $\T_{n}^N$. Our purpose is to carry out the asymptotic analysis of the sequence, in the spirit of concentration-compactness. We assume throughout that there exists a constant $K > 0$, which is independent of $n$, such that, for any $n \in \N$,
$$E_n(v^n) \leq K, \ 0 \leq p_n(v^n) \leq K, \ {\rm and} \ c(v^n) \leq K.$$
Passing possibly to a subsequence, we may assume, without loss of generality, that for some numbers $E \geq 0$, $\p \geq 0$ and $c \geq 0$, we have
\begin{equation}
\label{hypconv}
E_n(v^n) \to E, \ p_n(v^n) \to \p, \ {\rm and} \ c(v^n) \to c, \ {\rm as} \ n \to + \infty.
\end{equation}
The main result in this section is

\begin{theorem}
\label{bigconc-compc}
Assume $N = 2$ or $N = 3$, and let $(v^n)_{n \in \N^*}$ be a sequence of solutions of \eqref{TWc} in $X_n^2 \cap \boS_n^0$, resp. $X_n^3$, satisfying \eqref{hypconv}. Assume moreover that $E > 0$ and 
\begin{equation}
\label{asympsubsonic}
0 < c < \sqrt{2}.
\end{equation}
There exist an integer $\ell_0$ depending only on $c$ and $E$, $\ell$ non-trivial finite energy solutions $\v_1$, $\ldots$, $\v_\ell$ to \eqref{TWc} on $\R^N$ of speed $c$ with $1 \leq \ell \leq \ell_0$, $\ell$ points $x_1^n$, $\ldots$, $x_\ell^n$, and a subsequence of $(v^n)_{n \in \N^*}$ still denoted $(v^n)_{n \in \N^*}$ such that
\begin{equation}
\label{intermarche1}
|x_i^n - x_j^n| \to + \infty, \ {\rm as} \ n \to + \infty,
\end{equation}
and
\begin{equation}
\label{limitlimit1}
v^n(\cdot + x_i^n) \to \v_i(\cdot) \ {\rm in} \ C^k(K), \ {\rm as} \ n \to + \infty,
\end{equation}
for any $1 \leq i \neq j \leq \ell$, any $k \in \N$, and any compact set $K \subset \R^N$. Moreover, we have the identities
\begin{equation}
\label{limitenergy1}
E = \lim_{n \to + \infty} \big( E_n(v^n) \big) = \sum_{i = 1}^\ell E(\v_i), \ {\rm and} \ \p = \lim_{n \to + \infty} \big( p_n(v^n) \big) = \sum_{i = 1}^\ell p(\v_i).
\end{equation}
\end{theorem}

In Theorem \ref{bigconc-compc}, the tori $\T_n^N$ are identified with the subdomains $\Omega_n^N$ of $\R^N$, using possibly a suitable unfolding, so that convergence \eqref{intermarche1} makes sense. We will also need a variant of Theorem \ref{bigconc-compc} in the sonic case.

\begin{theorem}
\label{asympsonic}
Assume $N = 3$, and let $(v^n)_{n \in \N^*}$ be as in Theorem \ref{bigconc-compc} with assumption \eqref{asympsubsonic} replaced by 
$$c = \sqrt{2}.$$
Let $0 < \delta < 1$ be given. There exist an integer $\ell_0$ depending only on $E$, $\ell$ non-trivial finite energy solutions $\v_1$, $\ldots$, $\v_\ell$ to \eqref{TWc} on $\R^N$ of speed $\sqrt{2}$ with $0 \leq \ell \leq \ell_0$, $\ell$ points $x_1^n$, $\ldots$, $x_\ell^n$, and a subsequence of $(v^n)_{n \in \N^*}$ still denoted $(v^n)_{n \in \N^*}$ such that \eqref{intermarche1} and \eqref{limitlimit1} hold. Moreover, there exist real numbers $\mu \geq 0$ and $\nu$ such that we have the identities
\begin{equation}
\label{limitenergy2}
E = \lim_{n \to + \infty} \big( E_n(v^n) \big) = \sum_{i = 1}^\ell E(\v_i) + \mu, \ {\rm and} \ \p = \lim_{n \to + \infty} \big( p_n(v^n) \big) = \sum_{i = 1}^\ell p(\v_i) + \nu,
\end{equation}
and the inequality
\begin{equation}
\label{tropical}
|\mu - \sqrt{2} \nu| \leq K \delta \mu,
\end{equation}
where $K$ is some universal constant.
\end{theorem}

\begin{remark}
In contrast with the subsonic case, where $\ell \geq 1$, i.e. there is always at least one non-trivial finite energy solution on $\R^N$, with speed $c$, namely $\v_1$, appearing in the limiting behaviour, here, we may well have $\ell = 0$. In this situation we have $E = \mu$ and $\p = \nu$.
\end{remark}

The first observation which paves the way to the proof of Theorem \ref{bigconc-compc} is that the elliptic estimates derived in Subsection \ref{ellipticpoint} lead to the compactness of sequences of solutions, when we consider their restrictions to bounded domains. More precisely, as a direct consequence of Lemma \ref{tarquini1}, we have

\begin{lemma}
\label{localcompacite}
Let $(v^n)_{n \in \N^*}$ be as in Theorem \ref{bigconc-compc} or \ref{asympsonic}. There exists a finite energy solution $\v$ to \eqref{TWc} with speed $c$, and a subsequence of $(v^n)_{n \in \N}$ still denoted $(v^n)_{n \in \N}$ such that
$$v^n \to \v \ {\rm in} \ C^k(K), \ {\rm as} \ n \to + \infty,$$
for any $k \in \N$, and any compact set $K \subset \R^N$.
\end{lemma}

\begin{proof}
Since the family of speeds $(c_n)_{n \in \N}$ is bounded, it follows from Lemma \ref{tarquini1}, that, for any $k \in \N$, there exists a constant $K(k)$ such that
$$\| v^n \|_{C^k(\T_n^N)} \leq K(k), \forall n \in \N.$$
Hence, it follows from Ascoli's theorem, that, considering for any $j \in \N$, the compact balls $\overline{B(0, j)}$, there exists a subsequence (depending possibly on $j$), and a smooth map $\v^j$ on $B(0, j)$ such that
$$v_n \underset{n \to + \infty}{\to} \v^j \ {\rm in} \ C^k(B(0, j)), \ \forall k \in N.$$
Taking the limit in equation \eqref{TWc} for $v^n$, one verifies that the map $\v^j$ solves \eqref{TWc}, with speed $c$, the limit of the speeds $c(v^n)$ as $n \to + \infty$. To conclude, we let $j \to + \infty$, and we invoke a diagonal argument, so that in particular $\v = \v^j$ does not depend on the ball $B(0, j)$.
\end{proof}

Lemma \ref{localcompacite} does not handle the invariance by translations of the equation, which is a source of non-compactness. In particular, if we do not take care of this invariance, the limit $v$ provided by Lemma \ref{localcompacite} might well be equal to a constant, so that there is no hope to have conservation of energy and momentum. In order to handle the invariance by translations, we invoke the following general result.

\begin{lemma}
\label{clusteringinfini}
Let $\ell \in \N^*$, and consider for any $n \in \N^*$, a family $\{ x_1^n, \ldots, x_\ell^n \}$ of $\ell$ points in $\T_n^N$. There exist a subset $J_\infty$ of $\{1, \ldots, \ell\}$, a non-decreasing injection $\sigma: \N \to \N$, and sequences $(\kappa_k)_{k\in \N}$, $(R_k)_{k \in \N}$ and $(n(k))_{k \in \N}$ such that
$$0 < \kappa_k < \frac{1}{64}, \ \kappa_k \to 0, \ {\rm as} \ k \to + \infty, \ {\rm and} \ 1 \leq R_k < \Big( \frac{2}{\kappa_k} \Big )^\ell, \forall k \in \N,$$
and such that we have the relations
$$\underset{i = 1}{\overset{\ell}{\cup}} B \Big( x_i^{\sigma(n)}, 1 \Big) \subset \underset{j \in J_\infty}{\cup} B \Big( x_j^{\sigma(n)}, R_k\Big), \ {\rm and} \ \dist \Big( x_i^{\sigma(n)}, x_j^{\sigma(n)} \Big) \geq \frac{R_k}{4 \kappa_k}, \forall i \neq j \in J_\infty,$$
for any $n \geq n(k)$.
\end{lemma}

\begin{proof}
We iterate Lemma \ref{clustering} applied to the family $\{x_1^n, \ldots, x_\ell^n\}$ with values of $\kappa$ going to zero. More precisely, we introduce a new parameter $m \in \N$ which will eventually go to $+ \infty$ and take $\kappa$ of the form
$$\kappa_m = \frac{1}{m + 64}.$$
Starting with $m = 0$, Lemma \ref{clustering} yields for any $n \in \N$, a subset $J_0^n$ of $\{1, \ldots, \ell\}$, and a number $0 < \mu_0^n \leq \Big( \frac{2}{\kappa_0} \Big)^\ell$ such that
$$\underset{i = 1}{\overset{\ell}{\cup}} B(x_i^n, 1) \subset \underset{j \in J_0^n}{\cup} B(x_j^n, \mu_0^n), \ {\rm and} \ \dist(x_i^n , x_j^n) \geq \frac{\mu_0^n}{\kappa_0}, \forall i \neq j \in J_0^n.$$
We may extract a subsequence $\sigma_0: \N \to \N$ such that $J_0^{\sigma_0(n)}$ does not depend on $n$, so that we may denote it $J_0$, and such that the sequence $(\mu_0^{\sigma_0(n)})_{n \in \N}$ has a limit which we denote $\mu_0$. In particular, there exists some integer $n_0$ such that
$$\underset{i = 1}{\overset{\ell}{\cup}} B \Big( x_i^{\sigma_0(n)}, 1 \Big) \subset \underset{j \in J_0}{\cup} B \Big( x_j^{\sigma_0(n)}, 2 \mu_0 \Big), \ {\rm and} \ \dist \Big( x_i^{\sigma_0(n)} , x_j^{\sigma_0(n)} \Big) \geq \frac{\mu_0}{2 \kappa_0}, \forall i \neq j \in J_0,$$
for any $n \geq n_0$. We next proceed the same way with the family $\{ x_1^{\sigma_0(n)}, \ldots, x_\ell^{\sigma_0(n)} \}_{n \in \N}$ and $\kappa_1$, so that by Lemma \ref{clustering}, we find a subset $J_1^{\sigma_0(n)}$ of $\{ 1, \ldots, \ell \}$, and a number $0 < \mu_1^n \leq \Big( \frac{2}{\kappa_1} \Big)^\ell$ such that
$$\underset{i = 1}{\overset{\ell}{\cup}} B \Big( x_i^{\sigma_0(n)}, 1 \Big) \subset \underset{j \in J_1^n}{\cup} B \Big( x_j^{\sigma_0(n)}, \mu_1^n \Big), \ {\rm and} \ \dist \Big( x_i^{\sigma_0(n)} , x_j^{\sigma_0(n)} \Big) \geq \frac{\mu_1^n}{\kappa_1}, \forall i \neq j \in J_1^n,$$
for any $n \in \N$. We may extract a new subsequence $\tilde{\sigma}_1: \N \to \N$ such that, for $\sigma_1 = \tilde{\sigma}_1 \circ \sigma_0$, the set $J_1^{\sigma_1(n)}$ does not depend on $n$, so that we may denote it $J_1$, and such that the sequence $(\mu_1^{\sigma_1(n)})_{n \in \N}$ has a limit which we denote $\mu_1$. In particular, there exists some integer $n_1$ such that
$$\underset{i = 1}{\overset{\ell}{\cup}} B \Big( x_i^{\sigma_1(n)}, 1 \Big) \subset \underset{j \in J_1}{\cup} B \Big( x_j^{\sigma_1(n)}, 2 \mu_1 \Big), \ {\rm and} \ \dist \Big( x_i^{\sigma_1(n)} , x_j^{\sigma_1(n)} \Big) \geq \frac{\mu_1}{2 \kappa_1}, \forall i \neq j \in J_1,$$
for any $n \geq n_1$. We then iterate this argument to construct for any $j \in \N$, some non-decreasing injection $\tilde{\sigma}_j: \N \to \N$, some subset $J_j$ of $\{ 1, \ldots, \ell \}$, some number $1 < \mu_j \leq \Big( \frac{2}{\kappa_j} \Big)^\ell$, and some integer $n_j$, such that, if we set $\sigma_j = \tilde{\sigma}_j \circ \sigma_{j-1}$, then we have
\begin{equation}
\label{vin03}
\underset{i = 1}{\overset{\ell}{\cup}} B \Big( x_i^{\sigma_j(n)}, 1 \Big) \subset \underset{k \in J_j}{\cup} B \Big( x_k^{\sigma_j(n)}, 2 \mu_j \Big), \ {\rm and} \ \dist \Big( x_i^{\sigma_j(n)} , x_k^{\sigma_j(n)} \Big) \geq \frac{\mu_j}{2 \kappa_j}, \forall i \neq k \in J_j,
\end{equation}
for any $n \geq n_j$. We then set for any $n \in \N$, following the usual diagonal argument $\sigma(n) = \sigma_n(n)$ and $R_n = 2 \mu_n$. In view of \eqref{vin03}, we have for this choice
$$\underset{i = 1}{\overset{\ell}{\cup}} B \Big( x_i^{\sigma(n)}, 1 \Big) \subset \underset{j \in J_m}{\cup} B \Big( x_j^{ \sigma(n)}, R_m \Big), \ {\rm and} \ \dist \Big( x_i^{\sigma(n)} , x_j^{\sigma(n)} \Big) \geq \frac{R_m}{4 \kappa_m}, \forall i \neq j \in J_m,$$
for any $m \in N$, and any $n \geq n_m$. To conclude, we finally extract a subsequence $(\alpha(m))_{m \in \N}$ such that $J_{\alpha(m)}$ does not depend on $m$.
\end{proof}

In the course of the proof of Theorem \ref{bigconc-compc}, we will combine Lemma \ref{clusteringinfini} with Proposition \ref{concentration}. For that purpose, some decay properties of travelling wave solutions, which we recalled in Subsection \ref{propdecay}, turn out to be central in the arguments.

\begin{proof}[Proof of Theorem \ref{bigconc-compc}]
Set $\varepsilon = \sqrt{2 - c^2}$, so that we may assume, without loss of generally that $\frac{\varepsilon}{2} \leq \varepsilon(v^n) = \sqrt{2 - c(v^n)^2} \to \varepsilon$, as $n \to + \infty$, and that the energy $E_n(v^n)$ is uniformly bounded by some constant $E_0$. Our starting point is Lemma \ref{covering0}, which we apply to $v^n$, with $c_0 = \sqrt{2 - \frac{\varepsilon^2}{4}}$, and $\delta > 0$ taken as
\begin{equation}
\label{cestnotrechoix}
\delta = \delta_0(c_0) = \inf \Big\{ \frac{\sqrt{2} - c_0}{2 \sqrt{2} (K(c_0) + 1)}, \frac{1}{2} \Big\},
\end{equation}
where $K(c_0)$ is the constant appearing in Proposition \ref{concentration}. This yields a finite number $1 \leq \ell_n \leq \ell(\delta_0)$ of points $x_1^n$, $\ldots$, $x_{\ell_n}^n$ in $\T_n^N$ such that
$$\Big| 1 - |v^n| \Big| \leq \delta \on \T_n^N \setminus \underset{i = 1}{\overset{\ell_n}{\cup}} B(x_i^n, 1),$$
and
\begin{equation}
\label{langres}
\Big| 1 - |v^n(x_i^n)| \Big| \geq \delta (c_0), \forall 1 \leq i \leq \ell_n.
\end{equation}
Notice that the collection is not empty (i.e. $l_n \geq 1$), otherwise the map $v^n$ would be constant, in view of Proposition \ref{minoration1}. Passing possibly to a subsequence, we may assume that the number $\ell_n$ is independent of $n$, so that we may denote it $\ell$. We next apply Lemma \ref{clusteringinfini} to the family $\{ x_1^n, \ldots, x_\ell^n \}_{n \in \N^*}$ . Passing to a further subsequence, there exist a subset $J_\infty$ of $\{ 1, \ldots, \ell \}$, and sequences $(\kappa_k)_{k \in \N}$, $(n(k))_{k \in \N}$ and $(R_k)_{k \in N}$ such that
$$0 < \kappa_k < \frac{1}{64}, \ \kappa_k \to 0, \ {\rm as} \ k \to + \infty, \ {\rm and} \ 1 \leq R_k < \Big( \frac{2}{\kappa_k} \Big)^\ell,$$
for any $k \in \N$, and such that we have the relations
\begin{equation}
\label{bienloin2}
\underset{i = 1}{\overset{\ell}{\cup}} B \Big( x_i^{\sigma(n)}, 1 \Big) \subset \underset{j \in J_\infty}{\cup} B \Big( x_j^{\sigma(n)}, R_k\Big), \ {\rm and} \ \dist \Big( x_i^{\sigma(n)}, x_j^{\sigma(n)} \Big) \geq \frac{R_k}{4 \kappa_k}, \forall i \neq j \in J_\infty,
\end{equation}
for any $n \geq n(k)$. We next apply Lemma \ref{localcompacite} to the sequences of maps $(v^n(\cdot + x_j^n))_{n \in \N^*}$, for any $j \in J_\infty$, to assert that there exists some finite energy solution $\v_j$ with speed $c$ to \eqref{TWc} such that we have, up to a subsequence,
\begin{equation}
\label{epoisses}
v^n(\cdot + x_j^n) \to \v_j \ {\rm in} \ C^m(B(0, R)), \ {\rm as} \ n \to + \infty,
\end{equation}
for any $R > 0$, and any $m \in \N$. It follows from \eqref{langres} that
$$\Big| 1 - |\v_j(0)| \Big| \geq \delta (c_0),$$
so that $\v_j$ is not constant. At this stage, it remains to establish identities \eqref{limitenergy1}.

\noindent {\it Identity for the limiting energy}. We introduce the number $\mu_k = \frac{R_k}{8 \sqrt{\kappa_k}}$ for any $k \in \N$, as well as the exterior set
$$\boU_k = \T_n^N \setminus \underset{j \in J_\infty}{\cup} B \Big( x_j^n, \frac{R_k}{8 \kappa_k} \Big) = \T_n^N \setminus \underset{j \in J_\infty}{\cup} B \Big( x_j^n, \frac{\mu_k}{\sqrt{\kappa_k}} \Big),$$
so that $\mu_k \to + \infty$ as $k \to + \infty$. In view of relations \eqref{bienloin2}, we are in position to apply Proposition \ref{concentration} to $v^n$, with $\mu = \mu_k$ and $\kappa = \sqrt{\kappa_k}$, which yields
\begin{equation}
\label{exterior}
\int_{\boU_k} e(v^n) \leq \frac{K'(c_0) E_0}{|\ln(\kappa_k)| \varepsilon^2} =\underset{k \to + \infty}{o} (1), \forall n \geq n(k),
\end{equation}
where $K'(c_0)$ is some constant possibly depending on $c_0$. In view of convergence \eqref{epoisses}, and since $\frac{\mu_k}{\sqrt{\kappa_k}} \to + \infty$, as $k \to + \infty$, we have for any fixed $k$,
\begin{equation}
\label{interior}
\underset{n \to + \infty}{\lim} \bigg( E_n \Big( v^n, B \Big(x_j^n, \frac{\mu_k}{\sqrt{\kappa_k}} \Big) \Big) \bigg) = E \Big( \v_j, B \Big(0, \frac{\mu_k}{\sqrt{\kappa_k}} \Big) \Big) = E(\v_j) + \underset{k \to + \infty}{o} (1).
\end{equation}
Combining \eqref{exterior} and \eqref{interior}, we deduce that
$$\underset{n \to + \infty}{\lim} \bigg( E_n(v^n) \bigg) = \underset{j \in J_\infty}{\sum} E(\v_j) + \underset{k \to + \infty}{o} (1).$$

\noindent {\it Identity for the limiting momentum}. Let $r >0$ be such that that $|\v_j| \geq \frac{1}{2}$ on $\R^N \setminus B(0, r)$, for any $j \in J_\infty$, so that in particular we may write $\v_j = \varrho_j \exp i \varphi_j$ on $\R^N \setminus B(0, r)$. Let $0 \leq \chi \leq 1$ be a smooth function with compact support on $\R^N$ such that $\chi = 1$ on $B(0, r)$. In view of Lemma \ref{gdev}, we have
$$\p_j \equiv p(\v_j) = \tilde{p}(\v_j) = \frac{1}{2} \int_{\R^N} \bigg( \langle i \partial_1 \v_j, \v_j \rangle + \partial_1 \Big( (1 - \chi ) \varphi_j \Big) \bigg).$$
Since the integrand is integrable on $\R^N$, and since $r_k \equiv \frac{\mu_k}{\sqrt{\kappa_k}} \to + \infty$, as $k \to + \infty$, we have
\begin{equation}
\label{vert0}
p(\v_j) = p_k(\v_j) + \underset{k \to + \infty}{o} (1),
\end{equation}
where we have set
$$p_k(\v_j) = \frac{1}{2} \int_{B(0, r_k)} \langle i \partial_1 \v_j, \v_j \rangle + \frac{1}{2 r_k} \int_{\partial B(0, r_k )} \varphi_j(x) x_1 dx.$$
We now go back to the sequence $(v^n)_{n \in \N^*}$. Using Lemma \ref{massimo3}, we may write $v^n = \varrho^n \exp i \varphi^n$ on the set $\boO_n(2 R_k) \equiv \T_n^N \setminus \underset{j \in J_\infty}{\cup} B(x_j^n, 2 R_k)$ in dimension three for any $n \geq n(k)$. The existence of such a lifting in the two-dimensional case is more involved. Using \eqref{exterior}, we can invoke Corollary \ref{minimo2} to write $v^n = \varrho^n \exp i \varphi^n$ on the set $\boO_n(2 \mu_k) \equiv \T_n^N \setminus \underset{j \in J_\infty}{\cup} B(x_j^n, 2 \mu_k)$, so that, since $v^n$ does not vanish on each annulus $B(x_j^n, 2 \mu_k) \setminus B(x_j^n, 2 R_k)$, we may also write $v^n = \varrho^n \exp i \varphi^n$ on the set $\boO_n(2 R_k)$ in dimension two. In particular, $\langle i \partial_1 v^n, v^n \rangle = - (\varrho^n)^2 \partial_1 \varphi^n$ for any $n \geq n(k)$, so that, since $\boU_k$ is included in $\boO_n(2 R_k)$, we have
$$\int_{\boU_k} \langle i \partial_1 \v^n, \v^n \rangle = - \int_{\boU_k} (\varrho^n)^2 \partial_1 \varphi^n = \int_{\boU_k} \Big( 1 - (\varrho^n)^2 \Big) \partial_1 \varphi^n + \sum_{j \in J_\infty} \frac{1}{r_k} \int_{\partial B(x_j^n, r_k)} \varphi_i(x) x_1 dx.$$
It follows that
\begin{equation}
\label{vert1}
p_n(v^n) = \sum_{j \in J_\infty} \boP_{j,k}(v^n) + \frac{1}{2} \int_{\boU_k} \Big(1 - (\varrho^n)^2 \Big) \partial_1 \varphi^n,
\end{equation}
where we have set
$$\boP_{j,k}(v^n) = \frac{1}{2} \int_{B(x_j^n, r_k)} \langle i \partial_1 v^n, v^n \rangle + \frac{1}{2 r_k} \int_{\partial B(x_j^n, r_k)} \varphi^n(x) x_1 dx.$$
We deduce from convergence \eqref{epoisses} that, for any fixed $k$, we have
\begin{equation}
\label{vert2}
\boP_{j,k}(v^n) \to p_k(\v_j), \ {\rm as} \ n \to + \infty.
\end{equation}
On the other hand, we have by Lemma \ref{colisee} and inequality \eqref{exterior},
\begin{equation}
\label{vert3}
\bigg| \int_{\boU_k} \Big( 1 - (\varrho^n)^2 \Big) \partial_1 \varphi^n \bigg| \leq 2 \sqrt{2} \int_{\boU_k} e(v^n) \leq \frac{K'(c_0) E_0}{|\ln(\kappa_k)| \varepsilon^2} = \underset{k \to + \infty}{o}(1).
\end{equation}
Combining \eqref{vert0}, \eqref{vert1}, \eqref{vert2} and \eqref{vert3}, we obtain
$$\lim_{n \to + \infty} \bigg( p_n(v^n) \bigg) = \sum_{j \in J_\infty} p(\v_j),$$
so that the proof is complete.
\end{proof}

\begin{proof}[Proof of Theorem \ref{asympsonic}]
The beginning of the proof is identical to the proof of Theorem \ref{bigconc-compc} above, with the exception of one major difference: whereas the choice of $\delta$ (when applying Lemma \ref{covering0}) is given in the proof of Theorem \ref{bigconc-compc} by \eqref{cestnotrechoix} (which would yield $0$ in the sonic case), here we use the parameter $\delta$ which is provided in the statement of Theorem \ref{asympsonic}. Another difference concerns the integer $\ell^n$: it might be equal to $0$, and at this stage is bounded by a number possibly depending on $\delta$. All the arguments and estimates extend to the sonic case, except estimates \eqref{exterior} for the integral of the energy density on $\boU_k$, and estimate \eqref{vert3} for the momentum on $\boU_k$. Since the total energy is bounded, passing possibly to a further subsequence, we may assume that there exists some number $0 \leq \mu \leq E$ such that
\begin{equation}
\label{defectenergy}
\int_{\boU_k} e(v^n) \to \mu, \ {\rm as } \ n \to + \infty.
\end{equation}
Combining \eqref{defectenergy} with \eqref{interior}, we deduce the first equality in \eqref{limitenergy2}. For the second equality, using Lemma \ref{massimo3}, we may write $v^n = \varrho^n \exp i \varphi^n$ on the set $\boO_n(2 R_k)$, and passing possibly to another subsequence, we may assume similarly that, for some number $\nu \in \R$, we have
\begin{equation}
\label{defectmomentum}
\frac{1}{2} \int_{\boU_k} \Big( 1 - (\varrho^n)^2 \partial_1 \varphi^n \Big) \to \nu, \ {\rm as } \ n \to + \infty,
\end{equation}
so that combining \eqref{defectmomentum} with \eqref{vert0}, \eqref{vert1} and \eqref{vert2}, we deduce the identity for the limiting momentum in \eqref{limitenergy2}. For inequality \eqref{tropical}, we invoke inequality\eqref{dentifrice} in Proposition \ref{concentration} which yields
$$\bigg| \int_{\boU_k} \bigg( \frac{\sqrt{2}}{2} \Big( 1 - (\varrho^n)^2 \Big) \partial_1 \varphi^n - e(v^n) \bigg) \bigg| \leq K(c_0) 
\bigg( \delta \int_{\boU_k} e(v^n) + \underset{k \to + \infty}{o} (1)
 \bigg),$$
which yields the desired conclusion, taking the limit $n \to + \infty$. Finally, in order to see that the number $\ell$ may be bounded independently of $\delta$, we invoke Lemma \ref{soniccase}. Indeed, each function $\v_j$ is a non-trivial finite energy solution to \eqref{TWc} on $\R^3$, so that by Lemma \ref{soniccase}, $E(\v_j) \geq \boE_0$. Hence,
$$E = \sum_{j = 1}^\ell E(\v_j) + \mu \geq \ell \boE_0,$$
so that $\ell \leq \ell_0 \equiv \frac{E}{\boE_0}$ is bounded independently of $\delta$.
\end{proof}

\section{Properties of $E_{\min}^n(\p)$ and $u_\p^n$}
\label{epinfini}

In this section, we provide the proofs of Propositions \ref{existnn} and \ref{concentrationcompacite}. We also show that $E_{\min}^n(\p)$ converges to $E_{\min}(\p)$ as $n \to + \infty$.

\subsection{Proof of Proposition \ref{existnn}}

The first task is to establish the existence of a minimizer for \eqref{PnNp}. For that purpose, we consider a minimizing sequence $(w_k)_{k \in \N}$ for \eqref{PnNp}. Since $E_n(w_k)$ is uniformly bounded with respect to $k$, $(w_k)_{k \in \N}$ is bounded in $H^1(\T_n^N)$, so that, passing possibly to a subsequence, we may assume that
\begin{equation}
\label{garuet}
w_k \rightharpoonup u_\p^n \ {\rm in} \ H^1(\T_n^N), \ {\rm as} \ k \to + \infty,
\end{equation}
for some $u_\p^n \in H^1(\T_n^N)$. By weak lower semi-continuity and Rellich's compactness theorem, we infer that
\begin{equation}
\label{honnibal}
E_n(u_\p^n) \leq \underset{k \to + \infty}{\liminf} \big( E_n(w_k) \big) = E_{\min}^n(\p),
\end{equation}
and
\begin{equation}
\label{steyn}
p_n(u_\p^n) = \lim_{k \to + \infty} \bigg( \frac{1}{2} \int_{\T_n^N} \langle i w_k, \partial_1 w_k \rangle \bigg) = \p.
\end{equation}
The sequel of the proof is different according to the dimension. The two-dimensional case is substantially more difficult due to the topological constraint on the test functions.

\setcounter{case}{0}
\begin{case}
${\bf N = 3}$. Since $u_\p^n$ belongs to $\Gamma_n^3(\p)$ by \eqref{steyn}, it is a minimizer for \eqref{PnNp}, so that the Lagrange multiplier rule implies that
$$dE_n(u_\p^n) = c_\p^n dp_n(u_\p^n),$$
for some $c_\p^n \in \R$. The previous equality is precisely the weak formulation for the equation
$$i c_\p^n \partial_1 u_\p^n + \Delta u_\p^n + u_\p^n (1 - |u_\p^n|^2) = 0 \on \T_n^3,$$
whose finite energy solutions are smooth by standard elliptic theory.
\end{case}

\begin{case}
${\bf N = 2}$. In order to prove that $u_\p^n$ is a minimizer for $(\boP_n^2(\p))$, it remains to verify in view of \eqref{honnibal} and \eqref{steyn}, that
\begin{equation}
\label{white}
u_\p^n \in \boS_n^0,
\end{equation}
a difficulty which was not present in the three-dimensional case. To prove \eqref{white}, we are going to show that a suitable choice of the minimizing sequence yields strong converge to $u_\p^n$. This will yield the conclusion in view of the closeness of $\boS_n^0 \cap E_{n,\Lambda}$ for any fixed $\Lambda$. The main tool is Ekeland's variational principle.

Indeed, we consider some number $\Lambda$ such that $\Lambda > E_{\min}(\p)$. Using Corollary \ref{assezdirect}, there exists some integer $n(\Lambda)$ such that $E_{\min}^n(\p) < \Lambda$ for any $n \geq n(\Lambda)$. In particular, using Ekeland's variational principle (see \cite{Ekeland1}), we can construct a minimizing sequence $(w_k)_{k \in \N}$ for $(\boP_n^2(\p))$ such that
\begin{equation}
\label{blanco}
E_{\min}^n(\p) \leq E_n(w_k) < \Lambda, \forall k \in \N,
\end{equation}
and
\begin{equation}
\label{sella}
E_n(w_k) - E_n(w) \leq \frac{1}{k} \| w_k - w \|_{H^1(\T_n^2)}, \forall w \in \Gamma_n^2(\p), \forall k \in \N^*.
\end{equation}
Letting $\delta > 0$, and $\psi \in H^1(\T_n^2)$, and invoking Theorem \ref{lemmalulu} and \eqref{blanco}, the function $w_k - \delta \psi$ belongs to $E_{n, \Lambda} \cap \boS_n^0$ for any $\delta$ sufficiently small, and any $n$ sufficiently large. Moreover,
$$p_n(w_k - \delta \psi) = p_n(w_k) - \delta \int_{\T_n^2} \langle i \partial_1 w_k, \psi \rangle + \delta^2 p_n(\psi) \to \p, \ {\rm as} \ \delta \to 0,$$
so that the function $z_{k, \delta} = \sqrt{\frac{\p}{p_n(w_k - \delta \psi)}} (w_k - \delta \psi)$ belongs to $\Gamma_n^2(\p)$ for $\delta$ sufficiently small. Setting $w = z_{k, \delta}$ in inequality \eqref{sella}, and taking the limit $\delta \to 0$ after dividing by $\delta$, we are led to
$$\lambda_k dp_n(w_k)(\psi) - dE_n(w_k)(\psi) \leq \frac{1}{k} \Big\| \frac{dp_n(w_k)(\psi)}{2} w_k - \psi \Big\|_{H^1(\T_n^2)},$$
where $\lambda_k = \frac{1}{2 \p} dE_n(w_k)(w_k)$. By \eqref{blanco}, this gives
$$\Big| \lambda_k dp_n(w_k)(\psi) - dE_n(w_k)(\psi) \Big| \leq \frac{K(\Lambda)}{k} \| \psi \|_{H^1(\T_n^2)},$$
where $K(\Lambda)$ is some constant only depending on $\Lambda$. In particular, choosing $\psi = u_\p^n$, we are led to
$$\lambda_k dp_n(w_k)(u_\p^n) - dE_n(w_k)(u_\p^n) \to 0, \ {\rm as} \ k \to + \infty.$$
On the other hand, it follows from \eqref{garuet} that
$$dp_n(w_k)(u_\p^n) \to 2 p_n(u_\p^n) = 2 \p, \ {\rm as} \ k \to + \infty,$$
and
$$dE_n(w_k)(u_\p^n) \to \int_{\T_n^2} \Big( |\nabla u_\p^n|^2 - |u_\p^n|^2 (1 - |u_\p^n|^2) \Big), \ {\rm as} \ k \to + \infty,$$
so that
$$\lambda_k \to \frac{1}{2 \p} \int_{\T_n^2} \Big( |\nabla u_\p^n|^2 - |u_\p^n|^2 (1 - |u_\p^n|^2) \Big), \ {\rm as} \ k \to + \infty.$$
Hence, using \eqref{garuet} and Rellich's compactness theorem, we have
$$\int_{\T_n^2} |\nabla w_k|^2 \to 2 \p \underset{k \to + \infty}{\lim} \big( \lambda_k \big) + \int_{\T_n^2} |u_\p^n|^2 (1 - |u_\p^n|^2) = \int_{\T_n^2} |\nabla u_\p^n|^2, \ {\rm as} \ k \to + \infty,$$
which proves the strong $H^1$-convergence of the sequence $(w_k)_{k \in \N}$ towards $u_\p^n$. In particular, since $E_{n, \Lambda} \cap \boS_n^0$ is closed by Theorem \ref{lemmalulu}, $u_\p^n$ belongs to $\boS_n^0$, so that $u_\p^n$ is a minimizer for $(\boP_n^2(\p))$. Moreover, the set $\{ u \in H^1(\T_n^2), \ {\rm s.t.} \ E_n(u) < \Lambda \} \cap \boS_n^0$ is open by Theorem \ref{lemmalulu}, so that the Lagrange multiplier rule implies that
$$i c_\p^n \partial_1 u_\p^n + \Delta u_\p^n + u_\p^n (1 - |u_\p^n|^2) = 0 \on \T_n^2,$$
for some $c_\p^n \in \R$. Hence, $u_\p^n$ is also smooth in the two-dimensional case.
\end{case}

We finally turn to \eqref{supc} and \eqref{fastoche} to complete the proof of Proposition \ref{existnn}. We first notice that, by Corollary \eqref{assezdirect},
$$\limsup_{n \to + \infty} \Big( E_{\min}^n(\p) \Big) \leq E_{\min}(\p), \forall \p > 0,$$
so that
$$\liminf_{n \to +\infty} \Big( \Sigma(u_p^n) \Big) \geq \Xi(\p).$$
In particular, if $\Xi(\p) > 0$, it follows that there exists some integer $n(\p)$, and some number $\Sigma_0 > 0$ such that
$$\Sigma(u_p^n) \geq \Sigma_0, \forall n \geq n(\p).$$
Invoking Theorem \ref{bornec}, we obtain \eqref{supc}, whereas \eqref{fastoche} follows from Lemma \ref{tarquini1}.

\subsection{Proof of Proposition \ref{concentrationcompacite}}

By Proposition \ref{existnn}, for given $\p > 0$, the sequence $(u_\p^n)_{n \in \N^*}$ is a sequence of finite energy solutions of \eqref{TWc}, with uniformly bounded energy, and such that $(c(u_\p^n))_{n \in \N^*}$ is bounded. By Proposition \ref{troisfoisrien}, it converges up to a subsequence towards a non-trivial finite energy solution $u_\p$ to \eqref{TWc} on $\R^N$ of speed $c$, which satisfies in particular $\partial_1 u_\p \neq 0$. Hence, in view of convergence \eqref{racinededeux} of Proposition \ref{troisfoisrien}, we have
$$c_\p^n \equiv c(u_\p^n) \to c, \ {\rm as} \ n \to + \infty.$$
Moreover, we deduce from the results of \cite{Graveja2} that
$$0 < c \leq \sqrt{2}.$$
On the other hand, we may assume up to another subsequence that
$$E(u_\p^n) = E_{\min}^n(\p) \to \limsup_{n \to + \infty} \Big( E_{\min}^n(\p) \Big), \ {\rm as} \ n \to + \infty.$$
We next distinguish two cases.

\setcounter{case}{0}
\begin{case}
\label{sub-case}
${\bf 0 < c < \sqrt{2}}$. In this case, we may apply directly Theorem \ref{bigconc-compc} to the sequence $(u_\p^n)_{n \in \N^*}$. This yields \eqref{intermarche}, \eqref{limitlimit} as well as
\begin{equation}
\label{baiedesomme}
\p = \sum_{i = 1}^\ell \p_i, \ {\rm and} \ \limsup_{n \to + \infty} \Big( E_{\min}^n(\p) \Big) = \sum_{i = 1}^\ell E(u_i).
\end{equation}
Moreover, since $c > 0$, it follows from Lemma \ref{encorepoho} that $\p_i = p(u_i) > 0$. In view of the definition of $E_{\min}$, we have $E(u_i) \geq E_{\min}(\p_i)$, whereas by Corollary \ref{assezdirect}, we have $\underset{n \to +\infty}{\limsup} \big( E_{\min}^n(\p) \big) \leq E_{\min}(\p)$, so that identity \eqref{baiedesomme} yields
$$\sum_{i = 1}^\ell E_{\min}(\p_i) \leq E_{\min}(\p).$$
Comparing with inequality \eqref{subadditivity} of Corollary \ref{subadditif}, which is precisely the reverse inequality, we deduce that all inequalities actually have to be identities, that is 
$$E(u_i) = E_{\min}(\p_i), \ {\rm and} \ \limsup_{n \to + \infty} \Big( E_{\min}^n(\p) \Big) = E_{\min}(\p),$$
and the proof is complete in the considered case.
\end{case}

\begin{case}
${\bf c = \sqrt{2}}$. We are going to show that this case is excluded, so that Case \ref{sub-case} always holds. For that purpose, we apply Theorem \ref{asympsonic} instead of Theorem \ref{bigconc-compc} to the sequence $(u_\p^n)_{n \in \N^*}$, with a parameter $\delta > 0$ to be determined later. It yields a number $\ell \geq 1$ of finite energy solutions $u_i$ of speed $\sqrt{2}$ on $\R^N$, and numbers $\mu \geq 0$, and $\nu \geq 0$ such that
\begin{equation}
\label{okcoral1}
|\mu - \sqrt{2} \nu| \leq K \delta \mu,
\end{equation}
where $K$ is some universal constant,
\begin{equation}
\label{okcoral2}
\p = \p' + \nu, \ {\rm where} \ \p' = \sum_{i = 1}^\ell p(u_i) = \sum_{i = 1}^\ell \p_i,
\end{equation}
and
$$\limsup_{n \to + \infty} \Big( E_{\min}^n(\p) \Big) = \sum_{i = 1}^\ell E(u_i) + \mu.$$
Invoking as above Corollary \ref{assezdirect}, we are led to
$$\sum_{i = 1}^\ell E_{\min}(\p_i) + \mu \leq E_{\min}(\p).$$
In view of Corollary \ref{cordepoho}, we have $\Sigma(u_i) < 0$, so that $E(u_i) > \sqrt{2} \p_i$. Combining with \eqref{okcoral1} and \eqref{okcoral2}, we obtain
$$\sqrt{2} \p = \sqrt{2} \bigg( \sum_{i = 1}^\ell \p_i + \nu \bigg) \leq E_{\min}(\p) (1 + K \delta),$$
that is
$$\Xi(\p) \leq K \delta E_{\min}(\p).$$
Since $\delta$ was arbitrary, we may let it go to zero, so that $\Xi (\p) \leq 0$, which is a contradiction with assumption \eqref{themaincondition}, and completes the proof of Proposition \ref{concentrationcompacite}.
\end{case}

\begin{remark}
In the course of the proof, we have proved the identity
$$\limsup_{n \to + \infty} \big( E^n(u_\p^n) \big) = \limsup_{n \to + \infty} \big( E_{\min}^n(\p) \big) = E_{\min}(\p).$$
\end{remark}

\section{Proof of the main theorems}
\label{proofmain}

\subsection{ Proof of Theorem \ref{principaltheo}}

If $\p > \p_0$, it follows from Theorem \ref{proemin} that $\Xi(\p) > 0$, so that Propositions \ref{troisfoisrien} and \ref{concentrationcompacite} apply. In particular, it follows from Proposition \ref{concentrationcompacite} that there exist some integer $\ell \geq 1$, and some positive numbers $\p_1$, $\ldots$, $\p_\ell$ such that $E_{\min}(\p_i)$ is achieved by some map $u_i$ for any $i \in \{ 1, \ldots, \ell \}$, with
\begin{equation}
\label{queoui}
\p = \sum_{i = 1}^\ell \p_i, \ {\rm and} \ E_{\min}(\p) = \sum_{i = 1}^\ell E_{\min}(\p_i).
\end{equation}
We claim that 
\begin{equation}
\label{mainclaim}
\ell = 1.
\end{equation}
Indeed, assume by contradiction that $\ell \geq 2$. Then, it follows from Corollary \ref{subadditif} that $E_{\min}$ is linear on the interval $(0, \p)$. By Lemma \ref{gueri}, we deduce that $E_{\min}(\q)$ is not achieved for any $0 < \q < \p$. This contradicts the fact that $E_{\min}(\p_i)$ is achieved by $u_i$, for any $i \in \{ 1, \ldots, \ell \}$, and establishes therefore \eqref{mainclaim}. 

Going back to \eqref{queoui}, we have
$$\p = \p_1 = p(u_1), \ {\rm and} \ E_{\min}(\p) = E_{\min}(\p_1) = E(u_1).$$
This shows that $E_{\min}(\p)$ is achieved by the map $u_1 = u_\p$, which belongs to $W(\R^N)$, up to a multiplication by a constant of modulus one, by Corollary \ref{bienpose}.

\subsection{The two-dimensional case: proof of Theorem \ref{dim2} and Proposition \ref{T2bis}}

In this section, we provide the proof of the main results in dimension two, namely the proofs to Theorem \ref{dim2} and Proposition \ref{T2bis}.

\begin{proof}[Proof of Theorem \ref{dim2}]

In view of Theorem \ref{principaltheo}, it is sufficient to show that
$$\p_0 = 0 \ {\rm if} \ N = 2.$$
Going back to the definition of $\p_0$ and the properties stated in Theorem \ref{proemin}, this is equivalent to show that
\begin{equation}
\label{shrek1}
\Xi(\p) > 0, \ \forall \p > 0.
\end{equation}
Since the function $\Xi$ is non-decreasing, it is sufficient to check that property for $\p$ sufficiently small. By Lemma \ref{bornesupn2}, we have for any $\p$ sufficiently small,
\begin{equation}
\label{shrek2}
\Xi_{\min}(\p) \geq \frac{48 \sqrt{2}}{\boS_{KP}^2} \p^3 - K_0 \p^4,
\end{equation}
which yields \eqref{shrek1}, then the desired conclusion.
\end{proof}

\begin{proof}[Proof of Proposition \ref{T2bis}]
We divide the proof into several steps.

\setcounter{step}{0}
\begin{step}
\label{contador}
There exists $\p_1 > 0$ such that, if $0 < \p < \p_1$, and $u$ is a finite energy solution to \eqref{TWc} on $\R^2$ such that $p(u) = \p$, and $E(u) = E_{\min}(\p)$, then
$$| u(x) | \geq \frac{1}{2}, \, \forall x \in \R^2.$$
\end{step}

This is a consequence of inequality \eqref{elinfini20} of Lemma \ref{tarquini20}, and the facts that $0 \leq c \leq \sqrt{2}$, and $E_{\min}(\p) \leq \sqrt{2} \p$. 

\begin{step}
\label{evans}
If $0 < \p < \p_1$, and $u$ is a finite energy solution to \eqref{TWc} on $\R^2$ such that $p(u) = \p$ and $E(u) = E_{\min}(\p)$,
then
$$K_2 \p \leq \varepsilon(u) \leq K_3 \p,$$
where $K_2 > 0$ and $K_3$ are some universal constant. In particular, \eqref{c-estim} is established.
\end{step}

The second inequality follows from Step \ref{contador}, Lemma \ref{bornevarepsilon}, and the fact that $E_{\min}(\p) \leq \sqrt{2} \p$. For the first one, we invoke equality \eqref{acapulco} for a minimizer $u_\p$, which yields in particular,
\begin{equation}
\label{armstrong}
\Xi(\p) = \Sigma(u_\p) \leq \frac{\varepsilon(u_\p)^2}{\sqrt{2}} p(u_\p) = \frac{\varepsilon(u_\p)^2}{\sqrt{2}} \p.
\end{equation}
The conclusion follows using \eqref{shrek2}.

\begin{step}
Proof of inequality \eqref{estimE2}.
\end{step}

The lower bound for $\Xi$ in inequality \eqref{estimE2} is provided by Lemma \ref{bornesupn2}. Concerning the upper bound, we invoke inequality \eqref{armstrong} and Step \ref{evans} to obtain
$$\Xi(\p) \leq K \p^3.$$

\begin{step}
Proof of inequality \eqref{p-cube-estim}.
\end{step}

Combining inequality \eqref{acapulco} with Step \ref{contador} and Step \ref{evans}, we obtain
$$\int_{\R^2} |\partial_2 u_\p|^2 \leq \int_{\R^2} |\partial_2 u_\p|^2 +\Xi(\p) \leq K \varepsilon(u_\p)^2 \p \leq K \p^3,$$
whereas Lemma \ref{tropbien} yields a similar estimate for the two other terms on the l.h.s of \eqref{p-cube-estim}.

\begin{step}
Proof of inequality \eqref{Min-estim}.
\end{step}

In view of the invariance by translation, we may assume without loss of generality that the infimum of $|u_\p|$ is achieved at the point $0$, that is
$$|u_\p(0)| = \underset{x \in \R^2}{\min} |u_\p(x)|.$$
Inequality \eqref{Min-estim} is then a consequence of \eqref{normsup} for $v = u_\p$, and \eqref{armstrong}.
\end{proof}

\subsection {The three dimensional case: proof of Theorem \ref{dim3}}

\begin{lemma}
\label{pzeropasnul}
Let $N = 3$. We have
\begin{equation}
\label{petitp}
\p_0 \geq \frac{\boE_0}{\sqrt{2}},
\end{equation}
where $\boE_0 > 0$ is the constant provided in Lemma \ref{soniccase}.
\end{lemma}

\begin{proof}
Assume by contradiction that $\p_0 < \frac{\boE_0}{\sqrt{2}}$. Then, by Theorem \ref{principaltheo}, $E_{\min}(\p)$ is achieved for any $\p > \p_0$, by some map $u_\p$. On the other hand, $E_{\min}(\p) \leq \sqrt{2} \p$, so that if $\p_0 < \p < \frac{\boE_0}{\sqrt{2}}$, then
$$E(u_\p) < \boE_0,$$
whereas, in view of Lemma \ref{soniccase}, there is no finite energy solution to \eqref{TWc} with energy smaller than $\boE_0$. This gives a contradiction.
\end{proof}

\begin{lemma}
\label{atable}
Given any $\p > \p_0$, let $u_\p$ be a minimizer of $E_{\min}(u_\p)$ given by Theorem \ref{principaltheo}. Then, there exists a function $u_{\p_0} \in W(\R^3)$ such that $u_\p \to u_{\p_0}$ in $C^\infty_{\rm loc}(\R^3)$, as $\p \to \p_0$, with $p(u_{\p_0}) = \p_0$, and $E(u_{\p_0}) = \sqrt{2} \p_0$. In particular, $E_{\min}(\p_0)$ is achieved.
\end{lemma}

\begin{proof}
We first notice that it follows from Theorem \ref{proemin} and Corollary \ref{tresutile} that, there exists a universal constant $K > 0$, such that, for any $\p > \p_0$, we have
\begin{equation}
\label{ma}
\| 1 - |u_\p| \|_{L^\infty(\R^3)} \geq \frac{K}{E_{\min}(\p_0)^\alpha} \geq K.
\end{equation}
Without loss of generality, we may assume, in view of the invariance by translation, that 
$$|u_\p(0)| = \inf_{x \in \R^3} |u_\p(x)|,$$
so that it follows from \eqref{ma} that
$$|u_\p(0)| \leq 1 - K < 1.$$
Arguing as in the proof of Proposition \ref{troisfoisrien}, there exists a non-trivial finite energy solution $u_1$ to \eqref{TWc} with $c = \underset{\p \to \p_0}{\limsup} \big( c(\p) \big)$, such that, passing possibly to a subsequence, we have
$$u_\p \to u_1 \ {\rm in} \ C^k(K), \ {\rm as} \ \p \to \p_0,$$
for any compact set $K$ in $\R^3$, and any $k \in \N$. Moreover, we have
$$E(u_{1}) = \liminf_{\p \to \p_0} (E(u_\p)) = E_{\min}(\p_0) = \sqrt{2} \p_0, \ {\rm and} \ |u_{1}(0)| \leq 1 - K < 1,$$
so that $u_{1}$ is non-trivial. Assuming first that $c = \sqrt{2}$, we next apply Theorem \ref{asympsonic} to the sequence $(u_\p)_{\p > \p_0}$, with a parameter $\delta > 0$ to be determined later. Indeed, following the lines of the proofs of Theorems \ref{bigconc-compc} and \ref{asympsonic}, we may verify that their statements remain valid for any sequence $(v^n)_{n \in \N}$ of finite energy solutions to ${\rm (TWc_n)}$ on $\R^N$ satisfying \eqref{hypconv}. Hence, there exist a number $\ell \geq 1$ of finite energy solutions $u_i$ of \eqref{TWc} on $\R^3$, and numbers $\mu \geq 0$ and $\nu \geq 0$ such that
\begin{align*}
|\mu - \sqrt{2} \nu| & \leq K \delta \mu,\\
\p_0 & = \p' + \nu,
\end{align*}
where $\p' = \underset{i = 1}{\overset{\ell}{\sum}} p(u_i) = \underset{i = 1}{\overset{\ell}{\sum}} \p_i$, and
$$\sqrt{2} \p_0 = E_{\min}(\p_0) = \sum_{i = 1}^\ell E(u_i) + \mu,$$
so that, by Theorem \ref{proemin},
$$E_{\min}(\p_0 - \nu) = \sqrt{2} (\p_0 - \nu) = E_{\min}(\p_0) - \sqrt{2} \nu \geq \sum_{i = 1}^\ell E(u_i) - K \delta \mu.$$
In view of inequality \eqref{subadditivity} of Corollary \ref{subadditif}, we deduce that
$$E(u_i) \leq E_{\min}(\p_i) + K \delta \mu \leq \sqrt{2} \p_i + K \delta E_{\min}(\p_0).$$
Taking $i = 1$ and letting $\delta \to 0$, we deduce that
$$\Sigma(u_1) \geq 0.$$
It then follows from Corollary \ref{cordepoho} that
$$c = \limsup_{\p \to \p_0} \big( c(u_\p) \big) < \sqrt{2},$$
which gives a contradiction. Hence, we may assume that $c < \sqrt{2}$, and apply Theorem \ref{bigconc-compc} to the sequence $(u_\p)_{\p > \p_0}$. We then conclude as in the proof of Theorem \ref{principaltheo} that $\ell = 1$, $p(u_1) = \p_0$, and $E_{\min}(\p_0) = E(u_1)$.
\end{proof} 

\begin{proof}[Proof of Theorem \ref{dim3} completed]
It follows from Lemma \ref{pzeropasnul} that $\p_0$ is not equal to zero, from Theorem \ref{principaltheo} that $E_{\min}(\p)$ is achieved for $\p > \p_0$, and from Lemma \ref{atable}, that $E_{\min}(\p_0)$ is achieved, with $E_{\min}(\p_0) = \sqrt{2} \p_0$. In view of Theorem \ref{proemin}, $E_{\min}$ is affine on $(0, \p_0)$, so that it is not achieved on $(0, \p_0)$ by Lemma \ref{gueri}. Hence, statements i) and ii) of Theorem \ref{dim3} hold, such as statement iii), which follows from Lemma \ref{soniccase} and \eqref{petitp}. Finally, statement iv) results from Lemma \ref{bornepsfin}. Indeed, this yields that any minimizer $u_{\p_0}$ of $E_{\min}(\p_0) = E(u_{\p_0}) = \sqrt{2} \p_0$, satisfies
$$\varepsilon(u_{\p_0}) \geq \frac{K}{\p_0^{8 \alpha + 2}} > 0,$$
for some universal constants $\alpha$ and $K$, so that $c(u_{\p_0}) < \sqrt{2}$, and the conclusion follows using the monotonicity properties of the function $\p \mapsto c(u_\p)$.
\end{proof}

\begin{acknowledgement}
The authors are grateful to L. Almeida, D. Chiron, A. Farina, C. Gallo, P. G\'erard, S. Gustafson, M. Maris, K. Nakanishi, G. Orlandi, D. Smets, and T.-P. Tsai for interesting and helpful discussions. The authors are also thankful to the referee for providing valuable comments and suggestions. The first and second authors acknowledge support from the ANR project JC05-51279, "\'Equations de Gross-Pitaevskii, d'Euler, et ph\'enom\`enes de concentration", of the French Ministry of Research.
\end{acknowledgement}

\bibliographystyle{plain}
\bibliography{Bibliogr}

\end{document}